\documentclass[11pt]{amsart}
\usepackage[dvips]{graphicx} 
\usepackage{enumitem,amssymb,amscd,color,latexsym,epsfig,xypic,hyperref,stmaryrd}
\usepackage[mathscr]{euscript}

\DeclareFontFamily{OT1}{rsfs}{}
\DeclareFontShape{OT1}{rsfs}{n}{it}{<-> rsfs10}{}
\DeclareMathAlphabet{\curly}{OT1}{rsfs}{n}{it}

\DeclareFontFamily{U}{mathb}{\hyphenchar\font45}
\DeclareFontShape{U}{mathb}{m}{n}{
      <5> <6> <7> <8> <9> <10> gen * mathb
      <10.95> mathb10 <12> <14.4> <17.28> <20.74> <24.88> mathb12
      }{}
\DeclareSymbolFont{mathb}{U}{mathb}{m}{n}

\makeatletter
\newcommand{\eqnum}{\refstepcounter{equation}\textup{\tagform@{\theequation}}}
\makeatother

\renewcommand\;{\hspace{.6pt}}
\newcommand\ZZ{\textstyle{\Z\big[\frac12\big]}}
\newcommand\PP{\mathbb P}
\newcommand\EE{\mathbb E}
\newcommand\LL{\mathbb L}
\newcommand\C{\mathbb C}

\newcommand\R{\mathbb R}
\newcommand\N{\mathbb N}
\newcommand\Z{\mathbb Z}

\newcommand\cO{\mathcal O}
\newcommand\cA{\mathcal A}

\newcommand\cE{\mathcal E}
\newcommand\cF{\mathcal F}

\newcommand\cI{\mathcal I}
\newcommand\cJ{\mathcal J}

\newcommand\m{\mathfrak m}

\renewcommand\({\big(}
\renewcommand\){\big)}
\newcommand\wt{\widetilde}
\makeatletter
\newcommand{\so}{\ \ext@arrow 0359\Rightarrowfill@{}{\hspace{3mm}}\ }
\makeatother
\newcommand{\rt}[1]{\xrightarrow{\ #1\ }}
\newcommand\To{\longrightarrow}

\newcommand\into{\hookrightarrow}
\newcommand\INTO{\ \ar@{^(->}[r]<-.2ex>}
\newcommand{\Into}{\,\ensuremath{\lhook\joinrel\relbar\joinrel\rightarrow}\,}
\newcommand\onto{\to\hspace{-3mm}\to}
\newcommand\Onto{\To\hspace{-4.5mm}\to\ }
\renewcommand\Mapsto{\longmapsto}

\renewcommand\_{^{}_}
\newcommand\take{\!\smallsetminus}

\newcommand\Langle{\big\langle}
\newcommand\Rangle{\big\rangle}
\newfont{\bigtimesfont}{cmsy10 scaled \magstep5}
\newcommand{\bigtimes}{\mathop{\lower0.9ex\hbox{\bigtimesfont\symbol2}}}
\renewcommand\={\ =\ }
\newcommand\dbar{\overline\partial}
\newcommand\udot{^{\bullet}}

\DeclareMathSymbol{\lefttorightarrow}{3}{mathb}{"FC}
\DeclareMathSymbol{\righttoleftarrow}{3}{mathb}{"FD}

\newcommand\ev{\operatorname{ev}}
\newcommand\Gr{\operatorname{Gr}}

\newcommand\rk{\operatorname{rank}}
\newcommand\vir{\operatorname{vir}}
\newcommand\vd{\operatorname{vd}}

\newcommand\im{\operatorname{im}}
\newcommand\id{\operatorname{id}}

\renewcommand\Im{\operatorname{Im}}
\newcommand\Hom{\operatorname{Hom}}
\renewcommand\hom{\curly H\!om}

\newcommand\Ext{\operatorname{Ext}}

\newcommand\Sym{\operatorname{Sym}}
\renewcommand\Re{\operatorname{Re}}

\newcommand\beq[1]{\begin{equation}\label{#1}}
\newcommand\eeq{\end{equation}}
\newcommand\beqa{\begin{eqnarray*}}
\newcommand\eeqa{\end{eqnarray*}}

\newcommand\arXiv[1]{\href{http://arxiv.org/abs/#1}{arXiv:#1}}
\newcommand\mathAG[1]{\href{http://arxiv.org/abs/math/#1}{math.AG/#1}}

\makeatletter \@addtoreset{equation}{section} \makeatother
\renewcommand{\theequation}{\thesection.\arabic{equation}}
\newtheorem{defn}[equation]{Definition}
\newtheorem{thm}[equation]{Theorem}
\newtheorem{thm*}{Theorem}
\newtheorem{lem}[equation]{Lemma}

\newtheorem{prop}[equation]{Proposition}

\newtheorem{rmk}[equation]{Remark}

\title[Counting sheaves on Calabi-Yau 4-folds, II]{\ \\ \vspace{-2cm} Complex Kuranishi structures and \\ counting sheaves on Calabi-Yau 4-folds, II\vspace{-3mm}}
\author[J. Oh and R. P. Thomas]{Jeongseok Oh and Richard P. Thomas \vspace{3mm}\\ \emph{I\lowercase{n memory of} B\lowercase{umsig} K\lowercase{im, a wonderful mathematician and a wonderful human being}}}

\begin{document}
\begin{abstract}
We develop a theory of complex Kuranishi structures on projective schemes. These are sufficiently rigid to be equivalent to weak perfect obstruction theories, but sufficiently flexible to admit \emph{global} complex Kuranishi charts.

We apply the theory to projective moduli spaces $M$ of stable sheaves on Calabi-Yau 4-folds. Using real derived differential geometry, Borisov-Joyce produced a virtual homology cycle on $M$. In the prequel to this paper we constructed an algebraic virtual cycle on $M$.
We prove the cycles coincide in homology after inverting 2 in the coefficients. And when Borisov-Joyce's real virtual dimension is odd, their virtual cycle is 2-torsion. \vspace{-7mm}
\end{abstract}
\maketitle


\addtocontents{toc}{\protect\setcounter{tocdepth}{-1}}
\section*{Introduction} 

Let $M$ be a complex projective scheme, which in applications will be a moduli space.
In the theory of virtual cycles we describe $M$ as glued from local models. These express an open set $U\subset M$ as the zeros of a section $s$ of a bundle $E$ on a smooth ambient space $V$,
\beq{affmodel}
\xymatrix@=16pt{
& E\dto  \\
U \ \cong\ s^{-1}(0)\ \subset\hspace{-5mm} & V.\ar@/^{-2ex}/[u]_s}
\eeq
\subsection*{Kuranishi spaces \cite{FOOO, JoKur1}} Here the local models \eqref{affmodel} are $C^\infty$. Their germs (about $U\subset V$) are glued in a weak categorical sense that preserves the virtual dimension $v:=\dim V-\rk E$ (but not usually $\dim V$ or $\rk E$).

A compact $\mu$-Kuranishi space of \cite{JoKur1} admits a global Kuranishi chart \cite[Section 14 and Corollary 4.36]{Jvir} --- i.e. a global $C^\infty$ model \eqref{affmodel} with $U=M$. Its virtual cycle can be defined by perturbing $s$ to be transverse to the zero section of $E$ and taking the homology class of its zero locus.

\subsection*{Perfect obstruction theories \cite{BF, LT1}} Here we work algebraically and glue the \emph{derivatives} of algebraic local models \eqref{affmodel} about $M$. More precisely there should be a global 2-term complex of holomorphic vector bundles $E\udot=\{E^{-1}\to E^0\}$ on $M$ and a map $E\udot\to\LL_M$ which is locally quasi-isomorphic on the open sets $U$ to the following linearisation of \eqref{affmodel},
$$
\xymatrix@R=5pt@C=35pt{E^*|_U \ar[r]^{ds|_U}\ar[dd]_s& \Omega_V|_U \ar@{=}[dd]&& E\udot|_U \ar[dd] \\ && \simeq \\
I/I^2 \ar[r]^d& \Omega_V|_U && \LL_{\;U}.\!\!}
$$
Here $I\subset\cO_V$ is the ideal sheaf of $U$ and $\LL_{\;U}$ its truncated cotangent complex.

If there exists a \emph{global algebraic description} \eqref{affmodel} we can define an algebraic virtual cycle in $A_*(M)$ as the Fulton-MacPherson intersection of the graph $\Gamma_{\!s}\subset E$ with the zero section $0_E\subset E$. But this is rare. In the general case we take the limits of the graphs of $ts$,
\beq{localcone}
\lim_{t\to\infty}\Gamma_{\!ts}\ \subset\ E|_U, \,\text{ that is, the normal cone }\, C_{U/V}\ \subset\ E|_U,
\eeq
in each algebraic local model \eqref{affmodel}. Then we use the perfect obstruction theory to show they glue to a global cone $C$ inside $E_1:=(E^{-1})^*$ (which plays the role of $E|_M$, even when there is no global $E$). Intersecting it with $0_{E_1}$ defines the virtual cycle,
\beq{BFdef}
[M]^{\vir}\=0^{\;!}_{E_1}[C]\ \in\ A_{v}(M).
\eeq

\subsection*{Complex Kuranishi structures} To make \eqref{BFdef} easier to compute we would like to exploit the existence of global projective embeddings $M\subset\PP^N$ to find a \emph{global} Kuranishi chart \eqref{affmodel} which is as holomorphic as possible. (Then the virtual cycle becomes a localised Euler class of $E$.) So in this paper we develop a partially holomorphic version of Kuranishi structures. 

As we already noted it is too much to ask for the Kuranishi structure to be globally holomorphic,  so we define a \emph{complex Kuranishi structure} in a weaker way, following Joyce's $C^\infty$ definitions \cite{JoKur1} but using functions which lie in between smooth and holomorphic. These ``$\cA$\emph{-functions}" are $C^\infty$ on $V$ but holomorphic on a first order neighbourhood of $U$, in a sense made precise in Definition \ref{A} below. This turns out to be enough to relate them to perfect obstruction theories.

\begin{thm*}\label{K=P} Complex Kuranishi structures on $M$ (Definition \ref{cK}) are equivalent to weak perfect obstruction theories on $M$ (Definition \ref{wpot}).
\end{thm*}

Here a \emph{weak perfect obstruction theory} is the data of perfect obstruction theories on an open cover of $M$ that glue in a certain weak sense on overlaps; too weak to give a global perfect obstruction theory in general but strong enough that the local cones \eqref{localcone} glue and the virtual cycle \eqref{BFdef} is defined.

The added flexibility of these complex Kuranishi structures gives them one big advantage over a strictly algebraic or holomorphic definition, and over a perfect obstruction theory.

\begin{thm*}\label{1}
If $M$ is a projective scheme with perfect obstruction theory, it admits a \emph{global complex Kuranishi chart} $(V,E,s)$.
\end{thm*}

Here $E$ is an $\cA$-bundle over $V$, which implies that $E|_M$ is naturally a \emph{holomorphic} bundle. A big advantage that complex Kuranishi structures have over their $C^\infty$ counterparts is that we can recover the (algebraic) Behrend-Fantechi cone in this holomorphic bundle by taking limits of graphs of sections as in \eqref{localcone}.

\begin{thm*}\label{2} In the setting of Theorem \ref{1} the Behrend-Fantechi algebraic cone $C\subset E|_M$ is the $t\to\infty$ limit (as a current) of the graphs $\Gamma_{\!ts}\subset E$ of the sections $ts$. Thus $\Gamma_{\!s}$ and $C$ are homologous in the total space of $E$.
\end{thm*}

Joyce produces his virtual cycle by intersecting $\Gamma_{\!s}$ with the zero section $0_E$ of $E$ \cite[Section 13.2]{Jvir}; this lies in the homology of a neighbourhood retract $V^\circ$ of $M\subset V$, rather than $M$ itself. In contrast we have the virtual cycle \eqref{BFdef} on $M$ and Theorem \ref{2} shows its pushforward to the global Kuranishi chart $V$ (or the neighbourhood retract $V^\circ$, or $E$) can be given the following simple description,
\begin{eqnarray}
\iota_*\ \colon\,A_{v}(M) &\To& H_{2v}(V^{\circ}), \nonumber \\
{[}M]^{\vir}\=0^{\;!}_{E_1}[C] &\Mapsto& [\Gamma_{\!s}]\cap[0_E]. \label{notvir}
\end{eqnarray}
This reproves results of \cite{LT2, Sie} equating algebraic and $C^\infty$ virtual cycles (though they worked in the more complicated orbifold setting) and suggests complex Kuranishi structures should be of independent interest. However, this not how we will use Theorem \ref{2} in the rest of our paper. Instead we will apply it in DT$^4$ theory on moduli spaces of stable sheaves on Calabi-Yau 4-folds, where the natural obstruction theory  is \emph{not perfect} and the virtual cycle \cite{BJ, OT1} is \emph{not} given by \eqref{notvir}.

\subsection*{Moduli spaces of sheaves on Calabi-Yau 4-folds.}
Let $M$ be a projective moduli space of stable sheaves $F$ on a Calabi-Yau 4-fold $X$. 
This admits a very special type of local model \eqref{affmodel} in which the bundle $E$ carries a nondegenerate quadratic form with respect to which $s$ is isotropic \cite{BG,BBJ,BBBJ}. This is called a \emph{Darboux chart},
$$
\xymatrix@=3pt{
& (E,q)\ar[dd]  \\ &&& \qquad q(s,s)\,=\,0. \\
U \ \cong\ s^{-1}(0)\ \subset\hspace{-5mm} & V\ar@/^{-2ex}/[uu]_s}
$$
The natural virtual cotangent bundle to associate to this chart is the self-dual\footnote{Note that under the isomorphism $E\cong E^*$ defined by the quadratic form $q$ the second arrow is the dual $(ds)^*$ of the first.} complex
\beq{toob}
T_V|_M\rt{ds}E|_M\rt{ds}T^*_V|_M.
\eeq
At a point $F\in M$ this is quasi-isomorphic to $\tau^{[-2,0]}\(R\Hom_X(F,F)^\vee[-1]\)$ and its self-duality is the Serre duality $R\Hom_X(F,F)^\vee\cong R\Hom_X(F,F)[4]$.

Even though holomorphic Darboux charts exist only locally, in \cite[Propositions 4.1 and 4.2]{OT1} we prove the existence of a global self-dual 3-term complex like \eqref{toob} on $M$,
\beq{starr}
E\udot\=\big\{E_0\rt{a}E^{-1}\rt{a^*}E^0\big\},
\eeq
representing the virtual cotangent bundle. Here $E_0:=(E^0)^*$ and $E^{-1}\cong E_1$ $\cong E|_M$ is a special orthogonal bundle. Discarding the $E_0$ term in degree $-2$ of \eqref{starr} gives the stupid truncation $\tau E\udot=\{E^{-1}\to E^0\}$ of $E\udot$. This defines a (stupid)  perfect obstruction theory to which we can apply Theorems \ref{1} and \ref{2}. In this way we are able to find an $\cA$-approximation to a \emph{global Darboux chart}. This is enough to give a global \emph{isotropic} normal cone
\beq{cown}
\lim_{t\to\infty}\Gamma_{\!ts}\=C\ \subset\ E_1\=E|_M.
\eeq


For simplicity assume for now that $r$ is even, or equivalently that the virtual dimension
\beq{vddef}
\vd\ :=\ \chi\(\tau^{[-2,0]}\(R\Hom_X(F,F)^\vee[-1]\)\)\=2\rk T_V-\rk E
\eeq
is \emph{even}. Taking the Fulton-MacPherson intersection of the cone $C$ with the zero section of $E$ would give the naive Behrend-Fantechi virtual cycle, and would be the wrong thing to do here.\footnote{Since we removed $E_0$ from $E\udot$ \eqref{starr} the virtual dimension $v=\vd-\rk E_0$ is wrong, for instance. It is also not invariant under replacing $E\udot$ by a quasi-isomorphic complex.}

The correct thing to do is to ``halve" \eqref{toob} by halving the orthogonal bundle $E$ at the same time as discarding the $T_V$ term. There are two ways to do this, one using algebraic geometry \cite{OT1} and Borisov-Joyce's original method using real geometry \cite{BJ}.

The algebro-geometric method uses an auxiliary  positive maximal isotropic subbundle $\Lambda$ of $E_1$ to halve it. (We may assume one exists by passing to an appropriate cover of $M$, and later pushing back down --- it is for this step that we have to invert 2.)  We can then use \emph{cosection localisation} \cite{KL} to localise the intersection of $C$ and $\Lambda\subset E_1$ to the zero section $M$ as in \cite[Section 3]{OT1}. This defines our algebraic virtual cycle \cite[Definition 4.4]{OT1},
\beq{chowdef}
[M]^{\vir}\=(-1)^{\rk\Lambda}\,C\cdot\Lambda\ \in\ A_{\frac12\!\vd}\big(M,\Z\big[\textstyle{\frac12}\big]\big).
\eeq
The real differential geometric method uses a maximal positive definite real subbundle $E^\R$ of $E=E^\R\oplus iE^\R$ to halve it. Projecting $s$ to a section $s^+$ of $E^\R$, the fact that $s$ is isotropic easily implies that the zero locus of $s^+$ is the same as that of $s$, i.e. $(s^+)^{-1}(0)\cong M$ set theoretically. We show in Section \ref{sheaf4} that Borisov-Joyce's virtual cycle is the intersection of its graph with the zero section of $E^\R$,
\beq{BJvar}
[M]^{\vir}_{BJ}\=\Gamma_{\!s^+}\!\cdot0_{E^\R}\ \in\ H_{\vd}(M,\Z).
\eeq
By the projection formula this is the intersection of $iE^\R$ and $\Gamma_{\!s}$ inside $E$. By \eqref{cown} this is the following intersection in $E_1=E|_M$,
\beq{realnos}
C\cdot iE_1^\R\ \in\ H_{\vd}(M,\Z).
\eeq
In Section \ref{ten} we exhibit a homotopy from $\Lambda$ to $E_1^\R$ inside $E_1$. After multiplying by $i$ this takes \eqref{chowdef} into \eqref{realnos} in such a way that no intersection points wander off to infinity. Thus

\begin{thm*}\label{BBJJ} The virtual cycles coincide in homology with coefficients $\Z\big[\frac12\big]$,
$$
[M]^{\vir}_{BJ}\=[M]^{\vir}\quad\mathrm{in\ \,}H_{\vd}\big(M,\Z\big[\textstyle{\frac12}\big]\big).
$$
\end{thm*}

When $\vd$ \eqref{vddef} is odd we define the algebraic virtual cycle to be zero, so the analogue of Theorem \ref{BBJJ} is that the Borisov-Joyce virtual cycle is $2^n$-torsion for some $n$. In fact we can do a little better in this case.

\begin{thm*}\label{Odd}
Suppose $\vd$ \eqref{vddef} is odd. Then the Borisov-Joyce virtual cycle $[M]^{\vir}_{BJ}\in H_{\vd}(M,\Z)$ is 2-torsion.
\end{thm*}

This effectively solves DT$^4$ theory in a half of all cases, showing the invariants vanish.


Although we fixed $M$ to be a projective moduli space of stable sheaves on a Calabi-Yau 4-fold for definiteness, all we really use in this paper is its $(-2)$-shifted symplectic structure to provide the local Darboux charts of \cite{BG,BBJ,BBBJ}. So our results apply equally to any $(-2)$-shifted symplectic projective derived scheme $M$.

\subsection*{Notation: complex analytic spaces}
Though we will eventually be interested only in complex projective schemes, we work throughout in the category of analytic spaces. So we will use the following terminology.\vspace{3mm}

\begin{tabular}{cl} \emph{Terminology} & \emph{Meaning} \\[2mm]
scheme & analytic space \\
variety & \emph{reduced} analytic space \\
subscheme & locally zeros of \emph{finitely many} holomorphic functions \\
projective scheme & its associated analytic space \\
Stein space & analytic space admitting a closed embedding in $\C^N$ \\
open set & in the Euclidean topology \\
$\cO_Y$ & the structure sheaf of holomorphic functions on $Y$
\end{tabular}\medskip

\noindent Demailly \cite{De} uses ``complex analytic scheme" for the first line and ``complex space" for the second. Forstneri\v c \cite{Fo} uses ``nonreduced complex space" and ``reduced complex space" respectively. The above definition of a Stein space $X$ is equivalent to two conditions: (1) the higher cohomology $H^{\ge1}$ of all coherent sheaves on $X$ vanishes, and (2) $X$ has finite embedding dimension: there is a uniform bound on the dimensions of its Zariski tangent spaces. These Stein spaces are the analytic analogues of affine schemes (whose associated analytic spaces are indeed examples).

Whenever we take a neighbourhood $U$ of a subspace $Z$ of a topological space $X$, we can (and will) always take $Z\subset U$ to be a deformation retract since $X$ will be a finite CW complex and $Z$ a subcomplex. When $Z$ is Stein we can take $U$ to be Stein by Siu's theorem \cite[Theorem 3.1.1]{Fo}, \cite{Siu}.

A complex $E\udot$ of coherent sheaves of $\cO_Y$-modules has $E^i$ in degree $i$. When the $E^i$ are locally free we let $E_\bullet:=(E\udot)^\vee=R\hom(E\udot,\cO_Y)$ denote the derived dual with $E_{-i}:=(E^i)^*$ in degree $-i$. Here we use $E^*$ for the dual of a vector bundle $E$, reserving ${}^\vee$ for the derived dual of a complex of sheaves or object of the derived category.

\subsection*{Acknowledgements} Heartfelt thanks go to Dominic Joyce for generous assistance, helpful suggestions, and a patient explanation of his theory of $\mu$-Kuranishi spaces over a lunch of finest Imperial cabbage and potatoes. We are also grateful to Dennis Borisov, Huai-Liang Chang, Daniel Huybrechts, Richard L\"ark\"ang, Tom Leinster, Nikolai Mishachev, Borislav Mladenov, Andr\'e Neves, Mohan Ramachandran, Vivek Shende and Ivan Smith for useful correspondence.

We note that Benjamin Volk has developed a theory of homotopy stably almost complex Kuranishi spaces \cite{Vo} which are presumably a weakening of our complex Kuranishi structures. Also note the recent development in symplectic geometry of global Kuranishi charts for moduli spaces of pseudo-holomorphic curves \cite{AMS}.

J.O. acknowledges support of a KIAS Individual Grant MG063002.
R.P.T. acknowledges support from a Royal Society research professorship and EPSRC grant EP/R013349/1.

J.O. thanks his advisor Bumsig Kim (1968--2021), who passed away during the preparation of this paper, for his constant support and encouragement.

\tableofcontents
\vspace{-1cm}

\addtocontents{toc}{\protect\setcounter{tocdepth}{1}}
\section{Some locally ringed spaces}
\subsection*{Local models} A $\mu$-Kuranishi space $M$ in the sense of Joyce \cite[Chapter 2]{JoKur1} is built from local models $(V,E,s)$ like \eqref{affmodel} in which the ambient space $V$ is a $C^\infty$ manifold, $E$ is a $C^\infty$ bundle over it, and $s$ is a $C^\infty$ section cutting out $U$ (an open subset of $M$). On overlaps these are glued in a weak sense that need not preserve $\dim V$ or $\rk E$, though the virtual dimension $\dim V-\rk E$ is preserved.

When $M$ is a complex projective scheme we would like a holomorphic version of this theory. We now take $(V,E,s)$ to be algebraic or holomorphic, but asking for the gluing maps to be holomorphic on overlaps as well gives something too rigid to allow for global charts, for instance. Since the parts of $V$ away from $U$ play no role, it makes sense to allow the gluing maps to be $C^\infty$ on $V\take U$ and only holomorphic close to $U$.

To make close contact with perfect obstruction theories,\footnote{A perfect obstruction theory may be thought of as a (holomorphic) Kuranishi structure \emph{to first order about $M$}, i.e. at the level of first infinitesimal neighbourhoods.} we need the gluing to preserve the class of functions on $V$ which are both holomorphic on $U$ and which have holomorphic derivatives (also on $U$) down holomorphic vector fields in $T_V|_U$. As a result we want the gluing maps to be

\begin{center}
$C^\infty$ on $V$ and \emph{holomorphic on $2U,$}
\end{center}

\noindent where $2U\subset V$ is the scheme-theoretic first infinitesimal neighbourhood of $U\subset V$ defined by its ideal sheaf $I_{2U}:=I_U^2\subset\cO_V$.

\subsection*{Global models} We would like to exploit the existence of global projective embeddings $M\subset\PP^N$ to get a \emph{global Kuranishi chart} for $M$ --- expressing it as a zero locus like \eqref{affmodel} \emph{globally}. This would involve gluing smooth varieties like $V$ by $C^\infty$ maps which are holomorphic on $2U\subset V$. This leads naturally to a concept of $C^\infty$ manifolds $V$ with
\beq{*}
\text{\emph{a holomorphic structure on the first infinitesimal neighbourhood }}2U.
\eeq
Instead of building such manifolds from local algebraic or analytic varieties $V$ as above, it makes sense to let $V$ be any $C^\infty$ manifold to begin with, building in the holomorphic structure \eqref{*} along $2U$ using the theory of locally ringed spaces. So this Section introduces these locally ringed spaces. \smallskip


\noindent\textbf{Remark.} When reading it helps to be guided by the simple idea \eqref{*} of specifying a holomorphic structure on the first infinitesimal neighbourhood $2U$ inside $V$. And when it gets confusing which of $U$ or $2U$ we are worrying about, the reader should keep in mind that only $U$ makes any global sense (the open sets $U$ glue to form the projective scheme $M,$ after all). The $2U\;$s usually fail to glue, and are merely there as a technical device to specify functions with the right amount of holomorphicity for our purposes.
In particular, while gluing maps between open sets $V$ need not be isomorphisms on $2U$, they are \emph{holomorphic} on $2U$.\medskip

We first introduce the relevant rings of functions, beginning with holomorphic functions on $2U$.

\subsection{Smooth thickenings}\label{2smooth} Suppose $U$ is a subscheme of a scheme $V$, with ideal sheaf $I\subset\cO_V$. We define its \emph{scheme theoretic doubling} $2U=2U_V$ inside $V$ to be the subscheme with ideal sheaf $I_{2U_V}:=I^2$. It is a square zero extension of $U$ by the sheaf $\cI:=I/I^2$:
\beq{sq0}
0\To\cI\rt\iota\cO_{2U_V}\To\cO_U\To0.
\eeq

\begin{defn}\label{2u} A \emph{smooth thickening $2U$ of $U$ of dimension $n$} is a square zero thickening of $U$ which, in a neighbourhood of any point $x\in U$, is isomorphic to a thickening $2U_V$ with $V$ a smooth $n$-dimensional variety.
\end{defn}

We then set $\dim2U:=n$, which is also the dimension of all the Zariski tangent spaces $T_x\;(2U)$.

When $V$ is smooth the complex
\beq{tru}
\big\{I/I^2\rt d\Omega_V|\_U\big\}\ \simeq\ \LL_{\;U}
\eeq
is (quasi-isomorphic to) the \emph{truncated cotangent complex} $\LL_{\;U}$.
By standard deformation theory \cite[Section III.1.2]{Ill} square zero extensions of $U$ by an $\cO_U$-module $\cI$ are classified by their Kodaira-Spencer extension classes\footnote{Illusie calls these Exal classes. He uses the full cotangent complex; this differs from our truncated cotangent complex $\LL_U$ only in degrees $\le-2$ so the $\Ext^1$s are the same.}
$$
\mathrm{KS}\ \in\ \Ext^1(\LL_{\;U},\cI).
$$

\begin{prop}\label{conefree}
A square zero thickening $\mathrm{KS}\,\in\Ext^1(\LL_{\;U},\cI)$ of $U$ by an $\cO_U$-module $\cI$ is a smooth thickening if and only if
$$
\mathrm{Cone}\,\big(\LL_{\;U}[-1]\rt{\mathrm{KS}}\cI\big)
$$
is quasi-isomorphic to a locally free sheaf of constant rank on $U$.
\end{prop}

\begin{proof}
The result is local, so we may assume $U$ is Stein. Fix a smooth thickening $2U_V$ arising from an embedding in a smooth variety $V$. As in \eqref{tru} we may represent $\LL_{\;U}[-1]$ by the complex $I/I^2\to\Omega_V|_U$ in degrees 0 and 1 respectively. The Kodaira-Spencer class of the thickening is the obvious map from this complex to $I/I^2$, whose cone $\Omega_V|_U$ is locally free.\medskip

Conversely, suppose the cone on $\mathrm{KS}\,\in\Hom(\LL_{\;U}[-1],\cI)$ is a locally free sheaf. We first describe the square zero extension defined by KS.

To do so, choose any auxiliary smooth embedding $U\into V$ and denote by $I$ the corresponding ideal. 
The exact triangle $I/I^2\to\Omega_V|_U\to\LL_{\;U}$ of \eqref{tru} and the vanishing of $\Ext^1(\Omega_V|_U,\cI)=0$ for $U$ Stein shows that KS lifts to an element $e\in\Hom(I/I^2,\cI)$. Combined with $\iota$ of \eqref{sq0} we get an inclusion
$$
\xymatrix{I/I^2\ \ar@{^(->}^-{\,\iota\;\oplus\;e\ }[r]& \,\cO_{2U_V}\oplus\,\cI}.
$$
We make the right hand side an algebra --- in which the left hand side is an ideal --- by viewing $\cI$ as an $\cO_{2U_V}$-module and setting its square to 0. Therefore the quotient is an algebra sitting in the exact sequence
\beq{sqext}
0\To\cI\To\frac{\cO_{2U_V}\!\oplus\cI}{I/I^2}\To\cO_U\To0,
\eeq
and its spectrum is the required  square zero extension.

Replacing $U\into V$  by $U\times\{0\}\into V\times\C^N$ adds the acyclic complex $\cO_U^{\oplus N}\rt{{}_\sim}\cO_U^{\oplus N}$ to our representative \eqref{tru} of $\LL_{\;U}$, and adding a surjection $\cO^{\oplus N}_U\to\cI$ to $e$ gives another representative of the Kodaira-Spencer class. Thus we may assume $e$ is a surjection.

Letting $K:=\ker(e)$, we get a diagram on $U$ with exact columns,
$$
\xymatrix@R=16pt{K \ar@{=}[r]& K \\
I/I^2 \ar[r]^d\ar@{->>}[d]_e\ar@{<-_)}[u]-<0pt,10pt><0.3ex>& \Omega_V|\_U \ar@{->>}[d]
\ar@{<-_)}[u]-<0pt,10pt><0.3ex> \\ \cI \ar[r]& \mathrm{Cone\;(KS).}}
$$
Since we assumed $\Omega_V|_U$ and Cone(KS) are locally free sheaves, the injection of sheaves $K\into\Omega_V|_U$ must be an injection of vector bundles. Shrinking $V$ if necessary, we may assume $K$ is a trivial bundle $\cO^{\oplus r}$ and write $K\to I/I^2$ as $([f_1],\dots,[f_r])$ where $f_i$ are sections of $I\subset\cO_V$. Thus our thickening of $U$ by $\cI$ is the scheme theoretic doubling $2U_{V'}$ of $U$ inside the zero locus $V'$ of $(f_1,\dots,f_r)$ in $V$ and, by the above diagram, $\LL_{\;U}$ is the complex $\{\cI\to\Omega_{V'}|_U\}$. Finally, $V'$ is smooth along $U$ since $(df_1,\dots,df_r)\colon\cO^{\oplus r}\to\Omega_V|_U$ is a vector bundle injection.
\end{proof}


\subsection{Schemes and $C^\infty$ functions}
Let $U\subset V$ be an embedding of a scheme $U$ into a smooth variety $V$ with ideal $I\subset \cO_V$. Denoting complex conjugation by $\overline{\phantom{a}}\,$, we get a corresponding ideal
\beq{Iinfty}
I_{\infty}\ :=\ C^\infty_V.\;I\,+\,C^\infty_V.\,\overline{I\,}\ \subset\ C^\infty_V
\eeq
in the sheaf $C^\infty_V$ of \emph{complex-valued} smooth functions on $V$. This suggests a natural definition\footnote{Beware there are many different alternatives. In place of $I_\infty$ we could have used $C^\infty_V\!.\;I$ or $C^\infty_V.\;I+C^\infty_V.\,\overline{\surd\;I\,}$ or even 
$j_*\;j^*\(C^\infty_V.\;I\,+\,C^\infty_V.\,\overline{\surd\;I\,}\,\)$ \cite{AL}. (Here $\surd\;I$ is the radical of $I$ and $j\colon V\take U_{\mathrm{sing}}\into V$ is the complement of the singular locus of the reduced space $U_{\mathrm{red}}$.)} of sheaves of $C^\infty$ functions on $U$.

\begin{defn}\label{SMFTNS}
The sheaf of complex-valued smooth functions on $U$ is
$$
C^\infty_U\ :=\ C^\infty_V/I_{\infty}.
$$
\end{defn}

\noindent Independence from the choice of embedding $U\subset V$ requires the following.
\begin{enumerate}
\item[(i)] Any embedding is, locally about a point $x\in U$, isomorphic to the product of one fixed embedding $U\subset T_x\;U$ and $\{0\}\subset\C^n$. (Here we have shrunk $U$ and $T_x\;U$ denotes its Zariski tangent space at $x$.) 
\item[(ii)] The construction is functorial: a holomorphic map $f\colon U\to U'$ induces a canonical $f^*\colon C^\infty_{U'}\to C^\infty_U$ compatible with composition.
\end{enumerate}
By (i) we can compute $C^\infty_U$ locally, with different embeddings giving isomorphic sheaves. Then (ii) shows these isomorphisms are canonical, so the local sheaves glue canonically to a global sheaf on $U$.

We prove (ii) locally, so we may assume $f\colon U\to U'$ lifts to a holomorphic map $\wt f\colon V\to V'$ between the smooth ambient varieties $V\supset U,\ V'\supset U'$. (For instance we may shrink $V'$ to be an open set in $\C^n$ with coordinates $z_1,\dots,z_n$. Then, on shrinking $V$ if necessary, it admits holomorphic functions $\wt f_i$ extending $f^*z_i$ on $U$. Thus, after possibly shrinking $V$ further, we can take $\wt f:=(\wt f_1,\dots,\wt f_n)\colon V\to V'$.) That $\wt f$ extends $f$ means
$$
\wt f^*I'\ \subseteq\ I,
$$
where $I$ and $I'$ are the ideals of $U\subset V$ and $U'\subset V'$. Thus $\wt f^*\,\overline{I\,}'\subseteq\overline{I\,}\!$, inducing the required map
$$
f^*\,\colon\,\frac{C^\infty_{V'}}{C^\infty_{V'}.\;I'+C^\infty_{V'}.\;\overline{I\,}\;'}\ \To\ \frac{C^\infty_V}{C^\infty_V\!.\;I+C^\infty_V\!.\;\overline{I\,}}\,.
$$
To prove uniqueness we change the choices $\wt f_i$ by $\varepsilon_i\in I$ and assume (by further shrinking) that $V\subset\C^N$, so that subtraction makes sense and we can apply Hadamard's Lemma. For any $g\in C^\infty_{V'}$ this gives
$$
g\circ\(\wt f+\varepsilon\)-g\circ\wt f\=\Langle a,\varepsilon\Rangle+\Langle b,\,\overline\varepsilon\Rangle\ \in\ C^\infty_V\!.\;I\,+\,C^\infty_V\!.\,\overline{I\,}\!.
$$
Here $a$ and $b$ are $(\C^n)^*$-valued $C^\infty$ functions on $V$, which pair with the $\C^n$-valued $\varepsilon$ and $\overline\varepsilon$. Thus $f^*g$ is the same whether we define it via $\wt f$ or $\wt f+\varepsilon$.
\smallskip

The real reason that $\cO_U$ functorially determines $C^\infty_U$ is because it \emph{generates it as a $C^\infty$-ring} in the sense of \cite{Joinfty}. That is, given $N\in\N,\ F\in C^\infty_{\C^N}$ and $f_1,\ldots,f_N\in\(C^\infty_V\){}^{\oplus N}$ we can form $F\(f_1,\ldots,f_N\)\in C^\infty_V$. Hadamard's Lemma shows that altering the $f_i$ by elements of $I_\infty$ does the same to $F\(f_1,\ldots,f_N\)$, so it descends to a map $F\colon\(C^\infty_U\){}^{\oplus N}\to C^\infty_U$. Lemma \ref{IdealFact} below verifies that $\cO_U\subset C^\infty_U$, so by restriction we get a map
$$
F\,\colon\,\cO_U^{\oplus N}\To\,C^\infty_U.
$$
By varying $F$ and taking $N\ge\dim V$ these maps clearly generate $C^\infty_U$ if $U$ is sufficiently small to admit holomorphic coordinates.\medskip

Complex conjugation acts on $C^\infty_V$ and $I_\infty$ \eqref{Iinfty} and therefore $C^\infty_U$. Taking fixed sections defines sheaves of real-valued functions
$$
I_{\infty,\R}\ \subset\ C^\infty_{V,\;\R} \quad\text{and}\quad
C^\infty_{U,\;\R}\=C^\infty_{V,\;\R}\big/I_{\infty,\R}
$$
such that
$$
I_{\infty}\=I_{\infty,\R}\otimes\_{\R}\C, \quad C^\infty_V\=C^\infty_{V,\;\R}\otimes\_{\R}\C \quad\text{and}\quad C^\infty_U\=C^\infty_{U,\;\R}\otimes\_{\R}\C.
$$

\smallskip
The motivation for our choice of Definition \ref{SMFTNS} will be the application to \emph{real} Kuranishi spaces. 
Suppose that $U$ is cut out of $V$ by a holomorphic section $s\in H^0(E)$ of a holomorphic vector bundle $E$. That is, $s\colon E^*\to\cO_V$ has image $I\subset\cO_V$. Thinking of $s$ as a smooth section of the corresponding $C^\infty$ bundle $\mathsf E_{\;\C}:=E\otimes\_{\cO_V\!}C^\infty_V$,
\beq{sentence}
\text{the image of }\,s\,\colon\,\mathsf E_{\;\C}^*\ \To\ C^\infty_V\ \text{ is }\ C^\infty_V\!.\;I
\eeq
--- i.e. it is \emph{strictly smaller} than $I_\infty=C^\infty_V.\;I+C^\infty_V.\,\overline{\!I\;}$.
However, if we ignore the complex (and holomorphic) structure on $E$ and consider it as a \emph{real} $C^\infty$ vector bundle $\mathsf E_\R$  of twice the rank (whose sheaf of sections is the same $E\otimes\_{\cO_V\!}C^\infty_V$, but considered as a module over $C^\infty_{V,\R}\subset C^\infty_V$) then $s$ defines a section of it
\beq{realsec}
s\,\colon\,\mathsf{E}^*_\R\ \To\ C^\infty_{V,\;\R},\ \text{ with image }\ I_{\infty,\R}.
\eeq
(Tensoring with $\C$ then shows that $s+\overline s\colon\mathsf E_{\;\C}^*\oplus\overline{\mathsf E}_{\;\C}^*\to C^\infty_V$ has image $C^\infty_V.\;I+C^\infty_V.\;\overline{I\,}=I_{\infty}$.) 
\medskip


The next Lemma shows $\cO_U\subset C^\infty_U$ and that the data of $V$ and the ideal $I_\infty\subset C^\infty_V$ determine the scheme structure on $U$.

\begin{lem}\label{IdealFact}
The holomorphic functions in $I_\infty$ are just $I\subset\cO_V\subset C^\infty_V$,
$$
\cO_V\cap I_{\infty}\=I.
$$
\end{lem}

\begin{proof}
By Hironaka's principalisation theorem, 
there exists a composition of smooth blowups $\rho:V'\to V$ such that any given point of $\rho^{-1}($supp\,$I)$ is the origin of a choice of local holomorphic coordinates $z_1,\ldots, z_r$  with $\rho^*I=\(z_1^{d_1}\ldots\;z_k^{d_k}\)$, where $1\le k\le r$ and $d_1,\ldots,d_k>0$.

Given $f\in \cO_V\cap I_\infty$ we may use \eqref{Iinfty} to write it locally as $\sum a_i.b_i+\sum c_i.\overline d_i$, with $a_i,c_i\in C^\infty_V$ and $b_i,d_i\in I$. Pulling back we find
$$
\rho^*f\=F_1.z_1^{d_1}\!\ldots z_k^{d_k}\ +\ F_2.\bar{z}_1^{d_1}\!\ldots\bar{z}_k^{d_k}\quad\text{for some }F_i\in C^\infty_{V'}.
$$
Hence along $z_1=0=\dots=z_k$ we have $\rho^*f=0=\partial_{z_i}^j\(\rho^*f\)$ for all $i$ and all $j\le d_i-1$. Since $\rho^*f$ is holomorphic this means it is divisible by $z_1^{d_1}\!\ldots z_k^{d_k}$. Thus $\rho^*f\in\rho^*I$; since $I=\rho_*\rho^*I$ this proves $f\in I$.
\end{proof}

\subsection{$\cA$-functions on smooth varieties}\label{Afn}
As before we fix schemes $U\subset V$ with $V$ smooth.
We will use a class of functions that lies between $\cO_V$ and $C^\infty_V$, consisting of functions which are holomorphic on $U$ and $C^\infty$ on $V\take U$. There are at least two natural choices, corresponding to the sheaves of algebras
\beq{algs}
\cO_V\,\ \subset\,\ \cO_V+C^\infty_V.\;I\,\ \subset\,\ \cO_V+I_{\infty}\,\ \subset\,\ C^\infty_V.
\eeq
Notice these inclusions are \emph{strict} $\subsetneq$ when $U\ne\emptyset$, because $I_\infty=C^\infty_V\!.\;I+C^\infty_V\!.\ \overline{\!I\;}$. The third is really the algebra of all functions which are smooth on $V$ and holomorphic on $U$ since, by Lemma \ref{IdealFact},
$$
\cO_V+I_{\infty}\=C^\infty_V\times\_{C^\infty_U}\cO_U.
$$
Therefore, on replacing $U$ by $2U$, this is the maximal class of functions that realises the slogan \eqref{*}. However, for the purposes of constructing Kuranishi structures, \eqref{sentence} means it will be important for us to use the smaller class\footnote{For instance if $I\subset\cO_V$ is $(x^2)\subset\C[x]$ then $\overline x^2\in\cO_V\!+I_\infty$ but $\overline x^2\not\in\cO_V\!+C^\infty_V\!.I$.} of more restricted functions given by the second algebra in \eqref{algs}.

\begin{defn}\label{A} Let $\cA_{\;U/V}$ be the sheaf of algebras $\cO_V+(C^\infty_V\!.\;I)$ on $V$.
\end{defn}

We will use both $\cA_1:=\cA_{\;U/V}$ and $\cA_2:=\cA_{\;2U/V}=\cO_V+(C^\infty_V\!.\;I^2)$; the latter is our choice of a class of functions that satisfies \eqref{*}. Let
$$
I_1\ :=\ C^\infty_V\!.\;I\ \subset\ \cA_1
$$
denote the kernel of the restriction map from 
$\cA_1$-functions on $V$ to \emph{holomorphic} functions on $U$:
\beq{restric}
0\To I_1\To\cA_1\To\cO_U\To0.
\eeq
Replacing $U\subset V$ by $2U\subset V$ the ideal $I_1$ becomes $I_1^2\subset\cA_2$, sitting in the exact sequence
\beq{restric2}
0\To I_1^2\To\cA_2\To\cO_{2U}\To0.
\eeq
We denote by $I_2$ the kernel of $\cA_2\onto\cO_U$,
$$
0\To I_2\To\cA_2\To\cO_U\To0.
$$
Then $I_2=I+I^2_1=(\cO_V+I_1).\;I=\cA_1.\;I$, so $I_2$ is an ideal in $\cA_1$ but not in $C^\infty_V$. In fact
$$
I_1\=C^\infty_V.\;I_2 \quad\text{and}\quad \cA_1\=\cA_2+I_1,
$$
so \eqref{restric2} determines $I_2,\,I_1,\,I_\infty=I_1+\,\overline{\!I\;}\!_1$ and $\cA_1$. Hence 
we build our next definition around \eqref{restric2}.

%

\subsection{$\cA_2$-embeddings} 
Our ambient spaces will be locally ringed spaces which take the above form $(V,\cA_2)$ locally. We could define them by gluing such local pieces, but it is more natural to relax the condition that the local pieces $V$ be smooth varieties and work globally with a real $C^\infty$ manifold $V$ from the beginning.

\begin{defn}\label{UinV}
Fix a scheme $U$ and let $V$ be a real $C^\infty$ manifold. An $\cA_2$-embedding $\iota\colon U\into V$ is the data of
\begin{itemize}[leftmargin=11mm]
\item a smooth thickening $2U\!$ of $U\!$ in the sense of Definition \ref{2u},
\item an embedding of locally ringed spaces\;\footnote{Often called a closed immersion of locally ringed spaces. It amounts to an injection of sets $\iota\colon U\into V$ and a surjective algebra homomorphism $\pi\colon C^\infty_V\onto C^\infty_{2U}$ such that $\pi\(\;\overline f\;\)=\overline{\pi(f)}$ and $\pi(f)(x)=f(\iota(x))$ for all $x\in U$ and $f\in C^\infty_V$. Of course $(2U,C^\infty_{2U})$ is just $U$ with the sheaf of rings $C^\infty_{2U}$; the $2$ is solely for emphasis.} $(2U,C^\infty_{2U})\into(V,C^\infty_V)$, and
\item a sheaf of subalgebras $\cA_2\subset C^\infty_V$ closed under holomorphic operations,\footnote{That is $H(f_1,..,f_n)\in\cA_2$ when $f_j\in\cA_2$ and $H\colon\C^n\to\C$ is a $C^\infty$ function which is holomorphic on a neighbourhood of any point of $\im\($\scalebox{1.5}{$\times$}${\!}_i\;f_i\)$.}
\end{itemize}
satisfying $\dim\_{\;\R\!}V=2\dim 2U$ and the following analogue of \eqref{restric2}. Set
$$
I_2\ :=\ \ker\(\cA_2\to C^\infty_U\) \quad\text{and}\quad I_1:=C^\infty_V\!.\;I_2.
$$
Then $I_1$ should be locally finitely generated over $C^\infty_V$ and the composition $\cA_2\to C^\infty_V\to C^\infty_{2U}$ should have image $\cO\_{2U}\subset C^\infty_{2U}$ and kernel $I_1^2\subset C^\infty_V$.
\end{defn}

(The $\dim\_{\;\R\!}V=2\dim 2U$ condition rules out pathologies like taking $U$ to be the origin in $V:=\C^2_{x,y}$ with $\cO_{2U}=\C[x,y]/(x^2,y)$ and $\cA_2=x^2C^\infty_V+\C[x]$. Here the $\cA_2$ functions would all be constant on the $y$-axis $\{x=0\}$.)\smallskip

Given an $\cA_2$-embedding we set
\beq{A1A2I}
\cA_1\ :=\ \cA_2+I_1
\eeq
and note that this makes $I_2$ an $\cA_1$-module. Since $C^\infty_V\to C^\infty_U$ preserves complex conjugation it follows from its definition that
\beq{inclus1}
I_\infty\ :=\ I_1+\,\overline{\!I\;}\!_1\ \subseteq\ \ker\(C^\infty_V\to C^\infty_U\).
\eeq
Since $I_1^2=\ker(\cA_2\to C^\infty_{2U})\subset\ker\(C^\infty_V\to C^\infty_{2U}\)$, we also have
\beq{inclus2}
I_1^2+\,\overline{\!I\;}\!\_1\hspace{-.8mm}^2\ \subseteq\ \ker\(C^\infty_V\to C^\infty_{2U}\).
\eeq
It will follow from Proposition \ref{sane} below that these inclusions (\ref{inclus1}, \ref{inclus2}) are equalities.

Of course by design the local holomorphic model $2U_V\into V$ of Section \ref{Afn} satisfies the conditions of Definition \ref{UinV}, essentially by \eqref{restric2}. The only thing to check is closedness under holomorphic operations. So fix $\cA_2\ni f_j=h_j+g_j$ with $h_j\in\cO_V$ and  $g_j\in I_1^2$, a point $x\in U\subseteq V$ and $H\colon\C^n\to\C$ holomorphic on a neighbourhood of $(f_1(x),\dots,f_n(x))$. Then by Hadamard's Lemma in this neighbourhood,
$$
H\(f_1,...,f_n\)\=H\(h_1,...,h_n\)+\sum\nolimits_ia_ig_i,
$$
where the $a_i$ are made from the holomorphic derivatives $\partial_iH$  and there are no $\bar{g}_i$ terms because the antiholomorphic derivatives $\dbar_iH$ vanish. The first term on the right hand side is holomorphic and the second is in $I_1^2$, so their sum is in $\cA_2$.

Next we show a converse: that any $\cA_2$-embedding is \emph{locally} isomorphic to a holomorphic embedding $2U_V\into V$ such that the sheaves $I_1\subset\cA_1\subset C^\infty_V$ and $I_2\subset\cA_2\subset C^\infty_V$ recover the sheaves of the same name defined in Section \ref{Afn}. Therefore any \emph{global} $\cA_2$-embedding can be written as a union of local holomorphic models glued over overlaps by $C^\infty$ automorphisms that preserve the $\cA_2$ subsheaf.

\begin{prop}\label{sane}
Let $U\into V$ be an $\cA_2$-embedding and $x\in U$. On shrinking $V\supseteq U\ni x$ we may assume $U\into V$ is a holomorphic embedding of Stein varieties and the $\cA_2$-embedding is just $2U_V\into V$.
\end{prop}

\begin{proof}
Throughout we allow ourselves to shrink $U,V$ as necessary, while insisting they contain $x$. Thus we may assume the smooth thickening $2U$ in Definition \ref{UinV} is isomorphic to $2U_{V'}$, where $V'\supset U$ is an open subset of $\C^n,\ n=\dim 2U,$ with ideal $J\subset\cO_{V'}$. In particular $V$ and $V'$ have the same real dimension $2n$.

Let $x_1,\dots,x_n$ be the holomorphic coordinate functions on $V'$. Their restrictions to $\cO_{2U}\subset C^\infty_{2U}$ lift to functions $X_1,\dots,X_n\in\cA_2\subset C^\infty(V)$, inducing a map $V\to\C^n\cong\R^{2n}$. The derivatives of $\Re X_i,\,\Im X_i$ are linearly independent real 1-forms at $x\in U\subset V$ so the map is an isomorphism near $x$. Thus, shrinking $V$ (and $V'$) if necessary, we may assume the map is an embedding of $V$ as an open subset of $\R^{2n}$ with the same image as $V'\subset\R^{2n}$. Thus we get an isomorphism $\phi\colon V\rt\sim V'$ which is the identity on $C^\infty_{2U}$ (since the latter is generated by $\{\Re x_i,\,\Im x_i\}$ as a $C^\infty$ ring \cite{Joinfty}) and $\phi^*\cO_{V'}\subset\cA_2$ (because $\phi^*(x_i)=X_i\in\cA_2$ and $\cA_2$ is closed under holomorphic operations). This gives a commutative diagram
$$\xymatrix@R=15pt{
0 \ar[r]& \phi^*J \ar[r]\ar@{_(->}[d]& \phi^*\cO_{V'} \ar[r]\ar@{_(->}[d]& \cO_U \ar@{=}[d]\ar[r]& 0 \\
0 \ar[r]& I_2 \ar[r]& \cA_2 \ar[r]& \cO_U \ar[r]& 0.\!}
$$
In particular, multiplying by $C^\infty_V$ it follows that $\phi^*J_1\subseteq I_1$.

Both $\phi^*J_2$ and $I_2$ map to $\cO_{2U}$ with the same image $\cI=\ker(\cO_{2U}\to\cO_U)$ and kernels $\phi^*J_1^2\subseteq I_1^2$ respectively. Thus
\beq{incl}
I_2\=\phi^*J_2+I_1^2 \quad\text{and so}\quad I_1\=\phi^*J_1+I_1^2
\eeq
by multiplying by $C^\infty_V$. We claim that 
by the usual Nakayama-type argument this implies $I_1\subseteq\phi^*J_1$. In fact Definition \ref{UinV} ensures $I_1$ admits local generators $e_1,\dots,e_n$, so \eqref{incl} gives
$$
e_i\=f_i+\sum\nolimits_{j=1}^n a_{ij}e_j \quad\text{for some}\quad f_i\in\phi^*J_1,\ a_{ij}\in I_1.
$$
Since the matrix $A:=\(\delta_{ij}-a_{ij}\)_{i,j=1}^n$ is the identity on $U$ it is invertible on a neighbourhood. Hence by shrinking $V$ if necessary we have
$$
e_i\=\sum\nolimits_{j=1}^n(A^{-1})_{ij}f_j\ \in\ \phi^*J_1,
$$
so $I_1\subseteq\phi^*J_1$. We already observed $\phi^*J_1\subseteq I_1$, so $I_1=\phi^*J_1$. Thus, working inside $C^\infty_V$, the subsheaves $\cA_2$ and $\phi^*\cA_2(V')$ both map to $C^\infty_{2U}$ with the same image $\cO_{2U}$ and the same kernel $I_1^2=\phi^*J_1^2$, so they are the same.
\end{proof}

In summary an $\cA_2$-embedding induces the following embeddings of locally ringed spaces,
\begin{align}
(U,\cO_U)\ &\Into\ (V,\cA_1) \hspace{8mm}\text{with ideal}\ I_1=C^\infty_V\!.I_2, \nonumber \\
(U,\cO_U)\ &\Into\ (V,\cA_2) \hspace{8mm}\text{with ideal}\ I_2, \nonumber \\
(U,C^\infty_U)\ &\Into\ (V,C^\infty_V) \hspace{6.5mm}\text{with ideal}\ I_\infty=I_1+\,\overline{\!I\;}\!_1, \label{ideals} \\
(2U,\cO_{2U})\ &\Into\ (V,\cA_2) \hspace{8mm}\text{with ideal}\ I_1^2, \nonumber \\
\text{and}\quad (U,\cO_U)\ &\Into\ (2U,\cO_{2U}) \hspace{3.7mm}\text{with ideal}\ \cI:=I_2/I_1^2.\nonumber
\end{align}

\subsection{$\cA_1$-bundles}
We can define an $\cA_1$-bundle $E$ of rank $r$ on $(V,\cA_1)$ to be a vector bundle with transition functions in $GL(r,\cA_1)$. We can define the $\cA_1$-sections of $E$ to be those which are locally $r$-tuples of functions in $\cA_1$; this condition is then preserved by the transition functions.

Thus $E$ is equivalent to a locally free sheaf of $\cA_1$-modules (its sheaf of $\cA_1$-sections). From \eqref{ideals} we have the short exact sequence
$$
0\To I_1\To\cA_1\To\cO_U\To0
$$
generalising \eqref{restric} in the holomorphic setting. This gives
\beq{restr}
0\To E\otimes_{\cA_1}I_1\To E\To E|\_U\To0
\eeq
for some rank $r$ locally free sheaf of $\cO_U$-modules $E|_U$. Thus there is a natural notion of restriction of $E$ to give a \emph{holomorphic bundle} $E|_U$ on $U$.

The added flexibility of $\cA_1$ (over $\cO_V$) means that this operation is onto: in Appendix \ref{proof} we prove the following Theorem, showing that holomorphic bundles on $U$ extend to $\cA_1$-bundles on a neighbourhood $V\supset U$. (Notice this is a global, rather than local, result: $U$ could be a projective scheme which we do not shrink even as we allow its neighbourhood $V\supset U$ to shrink.)

\begin{thm}\label{extendE} Let $F$ be a holomorphic vector bundle over an open set $U$ of a projective scheme. Fix an $\cA_2$-embedding $U\subset V$ as in Definition \ref{UinV}. Shrinking $V\!$ if necessary, there exists an $\cA_1$-bundle $E$ on $V$ with $E|_U\cong F$.
\end{thm}

In other words, $F$ can be extended locally to (trivial) holomorphic bundles $E_i$ over the neighbourhoods $V_i$, and while the $(V_i,E_i)$ need not glue holomorphically, they do glue as $\cA_2$-neighbourhoods and $\cA_1$-bundles respectively.\medskip

We can also restrict more complicated $\cA_1$-modules on $V$ to give coherent sheaves of $\cO_U$-modules on $U$. For instance the $\cA_1$-modules defined in \eqref{ideals} sit in the exact sequence 
\beq{I_2I}
0\To I_1^2\To I_2\To\cI\To0.
\eeq
So given an $\cA_1$-bundle $E$, applying \eqref{I_2I} to the $\cA_1$-module $E\otimes_{\cA_1}I_2$ shows it has a restriction map to the coherent sheaf $E|_U\otimes \cI$ on $U$:
\beq{restrE}
0\To E\otimes_{\cA_1} I_1^2\To E\otimes_{\cA_1}I_2\To E\big|_U\otimes \cI\To0.
\eeq
The kernel $E\otimes_{\cA_1}I_1^2$ is a $C^\infty$-module, since $I_1$ is. Therefore it is a \emph{fine} sheaf, due to the existence of partitions of unity, and
\beq{surgj}
\Gamma_{\!\cA_1}\big(E\otimes_{\cA_1} I_2\big)\Onto\Gamma\big(E|_U\otimes \cI\big)\To0
\eeq
is a surjection. (We use the notation $\Gamma_{\!\cA_1}$ to emphasise that $E\otimes_{\cA_1} I_2$ is a sheaf of $\cA_1$-modules.) One consequence that we will find useful later is the following.

\begin{thm}\label{extsec} In the setting of Theorem \ref{extendE}, fix a surjection 
\beq{surjonto}
F^*\Onto\cI.
\eeq
Then, shrinking $V$ if necessary, there exists a surjection $E^*\onto I_2$ of $\cA_1$-modules which restricts on $U$ to \eqref{surjonto}.
\end{thm}

\begin{proof}
By \eqref{surgj} we can lift $F^*\onto\cI$ to an $\cA_1$-section $E^*\to I_2$. Denote its image by $J_2\subseteq I_2\subset C^\infty_V$. The images of $J_2\to C^\infty_{2U}$ and $I_2\to C^\infty_{2U}$ are both $\cI$, with the kernel of the second map $I_1^2$. Therefore $I_2=J_2+I_1^2$ just as in \eqref{incl}. The same Nakayama-type argument then shows $J_2=I_2$ after shrinking $V$.
\end{proof}
%

\smallskip
We use $\mathsf T_{\!V}$ to denote the \emph{real} tangent bundle of a $C^\infty$ manifold $V$. When $V$ is a complex manifold it has a holomorphic tangent bundle $T_V$; ignoring its complex (and holomorphic) structure makes it a real $C^\infty$ bundle canonically isomorphic to $\mathsf T_{\!V}$.

Next we note that when $U\into V$ is an $\cA_2$-embedding the restriction of $\mathsf T_{\!V}$ to $U$ naturally admits a holomorphic structure. We denote the resulting holomorphic bundle by $T_V|_U$ (and its dual by $\Omega_V|_U$) even though $V$ is not usually a complex manifold.

\begin{defn}\label{holotan}
Let $\Omega_V|_U$ denote the holomorphic bundle $\Omega_{2U}|_U$ and let $T_V|_U:=\(\Omega_{2U}|_U\)^*$. These are isomorphic, as real $C^\infty$ vector bundles, to the restrictions of the real (co)tangent bundles $\mathsf T^*_{\!V},\,\mathsf T_{\!V}$ to $(U,C^\infty_{U,\R})$ respectively.
\end{defn}

We can prove this locally. The holomorphic local chart of Proposition \ref{sane} gives an exact sequence $I^2/I^4\to\Omega_V|_{2U}\to\Omega_{2U}\to0$ of coherent sheaves. By the Leibnitz rule the first map $d$ factors through $I.\;\Omega_V|_{2U}$, so restricting the sequence to $U$ gives the required isomorphism
\beq{Omegiso}
\rt0\Omega_V|_U\rt\sim\Omega_{2U}|_U\To0.
\eeq
Finally the holomorphic bundle $\Omega_V|_U$ is isomorphic as a real $C^\infty$ bundle to $\mathsf T_{\!V}^*|\_U$ in the usual way: we embed 
the $(1,0)$-forms $\Omega_V|_U$ into the complex 1-forms $\mathsf T_{\!V}^*|\_U\otimes\C$ and take real parts.

For any $f\in\cA_2$ this also allows us to define a \emph{holomorphic} section
\beq{holodf}
df|_U\ \in\ H^0\(\Omega_V|_U\).
\eeq
 
\section{Complex Kuranishi neighbourhoods}
We give a partially holomorphic version of Joyce's definition \cite[Definition 2.2]{JoKur1} of a $\mu$-Kuranishi neighbourhood, using the $\cA_2$-embeddings of Definition \ref{UinV}.

\begin{defn} \label{K}
Let $M$ be a projective scheme. A complex Kuranishi neighbourhood of a point $x\in M$ is an open neighbourhood $x\in U\subseteq M$ and quadruple $(V,E,s,\iota)$, where
\begin{itemize}
\item $V$ is a $C^\infty$ manifold,
\item $\iota\colon U\into V$ is an $\cA_2$-embedding, defining subsheaves $I_1\subset\cA_1\subset C^\infty_V$ and $I_2\subset\cA_2\subset C^\infty_V$ as in \emph{(\ref{ideals})},
\item $E$ is an $\cA_1$-bundle over $V$, and
\item $s$ is a section of the sheaf $E\otimes_{\cA_1}I_2$
\end{itemize}
inducing a surjection of $\cA_1$-modules
\beq{surjA}
E^*\rt{\,s\,}\hspace{-5.5mm}\To I_2.
\eeq
\end{defn}

We call $v:=\dim V-\rk E$ the virtual dimension of the Kuranishi neighbourhood.
The surjection \eqref{surjA} allows us to think of $s$ as cutting out $U\subset V$ in a \emph{scheme-theoretic sense} even though $(E,s)$ are not necessarily holomorphic. We picture the neighbourhood as 
$$
\xymatrix@=16pt{
&& E\dto  \\
U \ar[r]^-{\iota\ }_-{\sim\ } & s^{-1}(0)\ \subset\hspace{-5mm} & V,\ar@/^{-2ex}/[u]_s}
$$
If $U$ and $V$ are Stein we call $(V,E,s,\iota)$ a \emph{Stein Kuranishi neighbourhood} of $x$. If $U=M$ we call it a \emph{global Kuranishi chart}.

Joyce's definition of $\mu$-Kuranishi neighbourhood is obtained in the same way by weakening the holomorphic requirements as follows. He takes $M\supseteq U$ to be topological spaces, $V,E,s$ to be $C^\infty$, requires $\iota$ to be a homeomorphism to its image, and replaces \eqref{surjA} by $s^{-1}(0)=\iota(U)$ \emph{as sets}.\footnote{This does not actually lose the ``scheme-theoretic" information on $U$ since we have the section $s$.} (In fact Joyce uses $\iota^{-1}\colon s^{-1}(0)\to U$ and calls it $\psi$.) In particular \emph{a complex Kuranishi neighbourhood is naturally a $\mu$-Kuranishi neighbourhood}.

\subsection*{Maps and morphisms}
We will need Joyce's $O(s),\,O(s^2)$ notation. Let $\cA$ denote any one of $C^\infty_V$, $\cA_1$, $\cA_2$ or $\cO_V$.
 
\begin{defn}\cite[Definition 2.1]{JoKur1}\label{defO} Given an $\cA$-section $s$ of an $\cA$-bundle $E$ on a smooth manifold $V\!$, we say an $\cA$-section $t$ of $F$ is
\begin{itemize}
\item $O(s)$ if there exists an $\cA$-map $\pi\colon E\to F$ such that $\pi(s)=t$,
\item $O\big(s^{\,2}\big)$ if there is an $\cA$-map $\pi\colon E^{\otimes2}\to F$ such that $t=\pi(s\otimes s)$,
\end{itemize}
in a neighbourhood of any point $x\in V$.
\end{defn}

Joyce \cite[Definition 2.3]{JoKur1} defines morphisms of $\mu$-Kuranishi neighbourhoods covering a map of topological spaces $U_1\to U_2$. For our purposes (of defining transition maps between neighbourhoods) we only need a definition of morphisms covering the identity map $U_{12}\to U_{12}$ between schemes.  Here we use the usual notation $U_{i_1i_2\cdots i_n}:=U_{i_1}\cap U_{i_2}\cap\cdots\cap U_{i_n}$ for overlaps. We first define \emph{maps} of Kuranishi neighbourhoods; later morphisms will be defined to be certain equivalence classes of maps.

\begin{defn}\label{KNdef}
Let $x\in M$ be a point of a projective scheme. Suppose $(V_1,E_1,s_1,\iota_1)$,$\,(V_2,E_2,s_2,\iota_2)$ are complex Kuranishi neighbourhoods for open neighbourhoods $U_1,\,U_2$ of $x$.
A map $(V_1,E_1,s_1,\iota_1)\to(V_2,E_2,s_2,\iota_2)$ over $U_{12}$ is a triple $(V_1^\circ,\psi,\Phi)$, where
\begin{itemize}
\item $V_1^\circ\subset V_1$ is an open neighbourhood\;\footnote{So we allow $s_1^{-1}(0)\cap V_1^\circ$ to be larger than $\iota_1(U_{12})$.
} of $\iota_1(U_{12})$,
\item $\psi\colon V_1^\circ\to V_2$ is an $\cA_2$-map such that, as maps of locally ringed spaces,
\beq{psicomp}
\psi\circ\iota_1\=\iota_2\,\colon\,(U_{12},\cO_{U_{12}})\To(V_2,\cA_2),
\eeq
\item $\Phi\colon E_1|\_{V_1^\circ}\to\psi^*E_2$ is an $\cA_1$-map of vector bundles such that
\item $\Phi\big(s_1|_{V_1^\circ}\big)=\psi^*(s_2)+O(s_1^2)$.
\end{itemize}
\end{defn}
Some explanation is needed here. Firstly, by an $\cA_2$-map we of course mean $\psi$ is a map of locally ringed spaces $(V_1^\circ,\cA_2)\to(V_2,\cA_2)$.

The condition \eqref{psicomp} implies $\psi^*$ maps the ideal sheaf $I_2\subset\cA_2(V_2)$ of \eqref{ideals} to $I_2\subset\cA_2(V_1^\circ)$. Since $\cA_1=\cA_2+C^\infty_V.\;I_2$ \eqref{A1A2I} this shows $\psi^*$ maps $\cA_1(V_2)$ to $\cA_1(V_1)$. Therefore $\psi$ is also an $\cA_1$-map: a map of locally ringed spaces
$$
\psi\,\colon\,(V_1^\circ,\cA_1)\To(V_2,\cA_1).
$$
In particular $\psi^*E_2$ is an $\cA_1$-bundle, as used in the third bullet point. 

So we are allowing ourselves to shrink $V_1$ (to $V_1^\circ$ --- i.e. we really only use the germ of the neighbourhood $V_1$ of $\iota_1(U_{12})$) whereupon we get a commutative diagram of locally ringed spaces
$$
\xymatrix@R=18pt@C=0pt{
E_1 \ar[d]\ar[rr]^{\Phi} && E_2 \ar[d] \\
(V_1^\circ,\cA_1) \ar[rr]^{\psi}\ar@{<-^)}[dr]+<-5pt,13pt><-.5ex>_-{\iota_1} && (V_2,\cA_1) \ar@{<-_)}[dl]+<5pt,13pt><.5ex>^-{\iota_2} \\ & \hspace{-4mm}\(U_{12},\cO_{U_{12}}\).\hspace{-4.5mm}}
$$
In particular, using $\iota_1,\,\iota_2$  to identify $(U_{12},\cO_{U_{12}})$ as a subspace of both $(V_1^\circ,\cA)$ and $(V_2,\cA)$ (where $\cA$ can be either of $\cA_1$ or $\cA_2$) we see that
\beq{psiid}
\psi\big|_{U_{12}}\=\id\_{\;U_{12}}.
\eeq
Conversely, all we know about
\beq{psi2U}
\psi\big|_{2(U_{12})_{V_1}}\,\colon\,2(U_{12})_{V_1}\To 2(U_{12})_{V_2}
\eeq
is that it is \emph{holomorphic} and restricts to the identity on $U_{12}$, because $\psi$ is an $\cA_2$-map. (The moral is that while the $U$s glue, the $2U$s do not in general, but their holomorphic structure is preserved.) Importantly, this means that
\beq{Dpsi}
D\psi\big|_{U_{12}}\,\colon\,T_{V_1}\big|_{U_{12}}\To T_{V_2}\big|_{U_{12}}
\eeq
is a holomorphic map of holomorphic bundles; cf. Definition \ref{holotan} and \eqref{holodf}.

Joyce uses the notation $(V_{ij},\phi_{ij},\hat\phi_{ij})$ in place of our $(V_1^\circ,\psi,\Phi)$, and allows them all to be $C^\infty$. In particular, \emph{any holomorphic map of complex Kuranishi neighbourhoods defines a morphism of $\mu$-Kuranishi neighbourhoods in the sense of \cite[Definition 2.3]{JoKur1}.}\medskip

Suppose we are given maps
$$
(V_1,E_1,s_1,\iota_1)\rt{(V_1^\circ,\psi_{12},\Phi_{12})}(V_2,E_2,s_2,\iota_2)\rt{(V_2^\circ,\psi_{23},\Phi_{23})}(V_3,E_3,s_3,\iota_3)
$$
over open subsets $U_{12}$ and $U_{23}$ respectively. If we shrink $V_1^\circ$ to $V_1^{\bullet}:=\psi_{12}^{-1}(V_2^\circ)$, so that it maps into $V_2^\circ$ under $\psi_{12}$, then we can define the composition over $U_{123}$ in the obvious way.

\begin{defn} We define the composition to be
$$
\Big(V_1^{\bullet},\,\psi_{23}\circ\psi_{12}|\_{V_1^{\bullet}},\,\Phi_{23}\circ\Phi_{12}|\_{V_1^{\bullet}}\Big)
$$
from $(V_1,E_1,s_1,\iota_1)$ to $(V_3,E_3,s_3,\iota_3)$ over $U_{123}$.
\end{defn}
\noindent Since this is a genuine composition of maps (of germs), it is associative.

\subsection*{Homotopies}
We will next identify maps related by an equivalence relation which may be thought of as a homotopy between them (in the sense of homotopies between maps of chain complexes, rather than homotopies of topological spaces).

So fix $M$ a projective scheme, $x\in M$ a point, and two complex Kuranishi neighbourhoods $(V_1,E_1,s_1,\iota_1),$ $(V_2,E_2,s_2,\iota_2)$ for open neighbourhoods $U_1,\,U_2$ of $x$. Recall the holomorphic bundles $\Omega_{V_i}|_{U_i},\,T_{V_i}|_{U_i}$ of Definition \ref{holotan}.

Denote by $[s_1]\in\Gamma\(E_1|_{U_1}\otimes\cI\)$ the image of $s_1$ under the restriction map \eqref{restrE}. Composing with the map $d$ to $E_1\otimes\Omega_{V_1}|_{U_1}$ gives $ds_1$, which we can consider as a holomorphic map $T_{V_1}|_{U_1}\to E_1|_{U_1}$.
\begin{defn}\label{Homotopy}
Let $(V_1^\circ,\psi,\Phi)$ and $\(\wt V_1^\circ,\wt\psi,\wt\Phi\)$ be two maps $(V_1,E_1,s_1,\iota_1)\to(V_2,E_2,s_2,\iota_2)$ over $U:=U_{12}$. A homotopy between these maps is a holomorphic map $H\colon E_1|_{U}\to T_{V_2}|_{U}$ such that the vertical maps of horizontal 2-term complexes
$$\xymatrix@R=7pt{
T_{V_1}\big|_{U} \ar[dd]_{D\psi|\_U}\ar[r]^{ds_1}& E_1\big|_{U} \ar[dd]^{\Phi|\_U}\ar@{..>}[ddl]_H && T_{V_1}\big|_{U} \ar[dd]_{D\wt\psi|\_U}\ar[r]^{ds_1}& E_1\big|_{U} \ar[dd]^{\wt\Phi|\_U} \\ && \mathrm{and} \\
T_{V_2}\big|_{U} \ar[r]^{ds_2}& E_2\big|_{U} && T_{V_2}\big|_{U} \ar[r]^{ds_2}& E_2\big|_{U}}
$$
differ by $H$, considered as a chain homotopy.
\end{defn}

Here we are again abusing notation slightly, identifying $U=U_{12}$ with a subset of each of $V_1^\circ,\,\wt V_1^\circ\subset V_1$ and $V_2$ via $\iota_1$ and $\iota_2$ so that, by \eqref{psiid}, $\psi|_{\iota_1(U)}\colon U\to U$ is just $\id_{\;U}$. We recall from \eqref{Dpsi} that $D\psi|_{U}$ and $D\wt\psi|_{U}$ are holomorphic. \medskip

Explicitly then, Definition \ref{Homotopy} is the following two identities,
\begin{eqnarray} \label{psieqn}
D\wt\psi|\_{U} &=& D\psi|\_{U}+H\circ ds_1, \\
\wt\Phi|\_{U} &=& \Phi|\_{U}+ds_2\circ H. \label{phieqn}
\end{eqnarray}

We can rephrase \eqref{psieqn} in terms of the holomorphic maps $\psi|\_{2U_{V_1}},\ \wt\psi|\_{2U_{V_1}}$ of \eqref{psi2U}, which act $2U_{V_1}\to 2U_{V_2}$. Namely \eqref{psieqn} is equivalent to
\beq{psieqn2}
\text{on restriction to }2U_{V_1},\ \wt\psi\ \mathrm{and}\ \psi\ \mathrm{differ\ by}\ H[s_1].
\eeq
Here $H[s_1]$ is the image of $[s_1]$ under
\beq{Hs1}
\Gamma\big(E_1|_{U}\otimes\cI\big)\rt{H\otimes1}\Gamma\big(T_{V_2}|\_{U}\otimes\cI\big),
\eeq
and holomorphic maps $2U_{V_1}\to2U_{V_2}$ which restrict on $U$ to $\id\colon U\to U$ form a torsor over the right hand group. This correspondence is defined by the actions on functions:
\beq{ans1}
\wt\psi^*(f)\ =\ \psi^*(f)+\Langle H[s_1],\;\psi^*df\Rangle \quad\text{for }f\in\cO_{2U_{V_2}},
\eeq
using the isomorphism \eqref{Omegiso} and the pairing between $T_{V_2}|_{U}$ and $\Omega_{V_2}|_{U}$. \medskip

\begin{rmk} \normalfont
Below we will identify maps which admit a homotopy $H$ between them. This allows us to shrink $V_1^\circ$, so our complex Kuranishi neighbourhoods really only see the  germ of $V_1$ around $U$. In fact Definition \ref{Homotopy} uses only the first infinitesimal neighbourhood $2U_{V_1}$ of $U\subset V_1$, so it would have been natural to base the theory of complex Kuranishi structures on ambient spaces which are locally scheme-theoretic doublings inside smooth varieties.\footnote{In retrospect this is more-or-less what Joyce's d-manifolds do, in a $C^\infty$ setting. If we had read \cite{Jvir} sooner we might have written a theory of \emph{holomorphic d-manifolds} instead. However our complex Kuranishi spaces have the advantage of making the intersection theory of Sections \ref{Cones}-\ref{sheaf4} much easier, once we have the global Kuranishi chart of Section \ref{global}.} Such a theory would also be easier to connect to perfect obstruction theories. Our choices were made to keep things as close as possible to Joyce's Kuranishi theory, in order to make contact with \cite{BJ} later.
\end{rmk}

For maps over $U=U_{12}$ we write $(V_1^\circ,\psi,\Phi)\sim\(\wt V_1^\circ,\wt\psi,\wt\Phi\)$ if there exists a homotopy $H$ between them. Using $H=0$ proves the binary relation $\sim$ is reflexive. Considering
$$
-H\,\colon\,E_1|\_U\To T_{V_2}|\_U
$$
to define a homotopy in the opposite direction proves symmetry. Finally, to compose the homotopies
$$
(V_1^\circ,\psi,\Phi)\overset{H}{\scalebox{2}[1]{$\ \sim\ $}}\(\wt V_1^\circ,\wt\psi,\wt\Phi\)\overset{\wt H}{\scalebox{2}[1]{$\ \sim\ $}}\(\overline V_1^\circ,\overline\psi,\overline\Phi\)
$$
we use $(H+\wt H\;)\colon E_1|_U\to T_{V_2}|_U$, proving transitivity. 
Therefore $\sim$ is an equivalence relation.

\begin{defn}\label{morf} A morphism of complex Kuranishi neighbourhoods over $U_{12}$ is a $\sim$-equivalence class of maps over $U_{12}$.
\end{defn}

We note that composition of maps also descends to equivalence classes of maps. Fix complex Kuranishi neighbourhoods over $U_1,\,U_2,\,U_3$ and maps between them over $U_{12}$ and $U_{23}$,
$$
(V_1,E_1,s_1,\iota_1)\xrightarrow[(\wt V_1^\circ,\wt\psi_{12},\wt\Phi_{12})]{(V_1^\circ,\psi_{12},\Phi_{12})}(V_2,E_2,s_2,\iota_2)\rt{(V_2^\circ,\psi_{23},\Phi_{23})}(V_3,E_3,s_3,\iota_3).
$$
Given a homotopy $H\colon E_1|_{U_{12}}\to T_{V_2}|_{U_{12}}$ between the first two over $U_{12}$, the composition 
$$
E_1|\_{U_{123}}\rt{H}T_{V_2}|\_{U_{123}}\rt{D\psi\_{23}}T_{V_3}|\_{U_{123}}
$$
is a  homotopy from the upper composition to the lower over $U_{123}$. Similarly given maps 
$$
(V_1,E_1,s_1,\iota_1)\xrightarrow{(V_1^\circ,\psi_{12},\Phi_{12})}(V_2,E_2,s_2,\iota_2)\xrightarrow[(\wt V_2^\circ,\wt\psi_{23},\wt\Phi_{23})]{(V_2^\circ,\psi_{23},\Phi_{23})}(V_3,E_3,s_3,\iota_3),
$$
and a homotopy $H\colon E_2|_{U_{23}}\to T_{V_3}|_{U_{23}}$ between the latter two, the composition
$$
E_1|\_{U_{123}}\rt{\Phi_{12}}E_2|\_{U_{123}}\rt{H}T_{V_3}|\_{U_{123}}
$$
is a homotopy from the upper composition to the lower over $U_{123}$.

We next check that morphisms in our sense naturally induce morphisms in the sense of Joyce. In particular our coordinate changes are also coordinate changes in the sense of \cite[Definition 2.6]{JoKur1}.

\begin{prop}\label{prp} If two maps $(V_1^\circ,\psi,\Phi)$, $(\wt V_1^\circ,\wt\psi,\wt\Phi)$ of complex Kuranishi neighbourhoods are equivalent under $\sim$ they define the same morphism of $\mu$-Kuranishi neighbourhoods in the sense of \,\cite[Equation 2.1]{JoKur1}.
\end{prop}

To prove this we need some notation for real vector bundles. (These are unavoidable since $\mathsf T_{\!V}$ is not in general a complex bundle.)

When $U\into V$ is an $\cA_2$-embedding we recall the holomorphic bundle $T_V|_U$ of Definition \ref{holotan}. As noted there, when we forget about its complex (and holomorphic) structure and consider it as a real $C^\infty$ bundle (of twice the rank) it is naturally isomorphic to the restriction $\mathsf T_{\!V}|_U$ of the real tangent bundle of $V$. In terms of sheaves of sections, this means we have an isomorphism of $C^\infty_{U,\R}$-modules
$$
T_V|_U\otimes\_{\cO_U}C^\infty_{U}\ \cong\ \mathsf T_{\!V}\otimes\_{C^\infty_{V,\R}}C^\infty_{U,\R}.
$$
By ignoring their complex structures we can also think of the $\cA_1$-bundles $E_1,E_2$ as real $C^\infty$ bundles (of twice the rank). When we do so we denote them by $\mathsf E_1,\mathsf E_2$ respectively. We use the same notation for their sheaves of sections $E_i\otimes_{\cA_1}C^\infty_{V_i}$, which are locally free modules over $C^\infty_{V_i,\R}\subset C^\infty_{V_i}$.

For $V_i$ complex (vector spaces or bundles) we use the natural embedding
$$
T_2\colon V_1\otimes_\C V_2\,\Into V_1\otimes_\R V_2,\ \quad v_1\otimes_\C v_2\,\Mapsto\,\tfrac12\big[ v_1\otimes_\R v_2-(iv_1)\otimes_\R(iv_2)\big],
$$
with left inverse $v_1\otimes_\C v_2\longmapsfrom v_1\otimes_\R v_2$. This expresses the space $V_1\otimes_\C V_2$ of complex linear tensors (those on which the $i\otimes_\R1$ and $1\otimes_\R i$ actions agree) inside the space of all real tensors $V_1\otimes_\R V_2$ as the $-1$ eigenspace of $i\otimes i$.

Similarly we will need the analogue $T_3\colon V_1\otimes_\C V_2\otimes_\C V_3\Into V_1\otimes_\R V_2\otimes_\R V_3$ for triple tensor products; this maps $v_1\otimes\_\C v_2\otimes\_\C v_3$ to
$$
\tfrac14\big[ v_1\otimes\_\R v_2\otimes\_\R v_3-(iv_1)\otimes\_\R(iv_2)\otimes\_\R v_3-(iv_1)\otimes\_\R v_2\otimes\_\R(iv_3)-v_1\otimes\_\R(iv_2)\otimes\_\R(iv_3)\big]
$$
with left inverse $v_1\otimes\_\C v_2\otimes\_\C v_3\longmapsfrom v_1\otimes\_\R v_2\otimes\_\R v_3$.

Suppose $V$ is a complex vector space with dual $V^*$. For clarity we temporarily denote $V$ by $V_\R$ when we forget its complex structure and consider only the underlying real vector space. Then its real dual $(V_\R)^*$ can be identified with $V^*$ (as a real vector space) via the real linear isomorphisms
$$
\xymatrix{V^*\ \ar[r]<.4ex>^(.4){\Re}& \ (V_\R)^*, \ar[l]<.4ex>} \qquad f\,\Mapsto\,\Re f, \qquad g-ig(i\ \cdot\ )\,\longmapsfrom\,g.
$$
For $V,W$ complex we combine this map $\Re$ with $T_3$ across the top of the following diagram,
\beq{halfcommute}
\xymatrix@C=70pt{V\otimes\_\C V^*\otimes\_\C W \ar[r]^(.45){(1\otimes\Re\otimes1)\circ T_3}\ar[d]_(.54){\ev\_\C\!}^{\!\otimes1} & V_\R\otimes\_\R (V_\R)^*\otimes\_\R W_\R \ar[d]_(.54){\ev\_\R\!}^{\!\otimes1} \\
W \ar[r]^{\frac12\!\id}& W_\R.\!\!}
\eeq
Then a computation shows that it commutes if we make the lower map $\frac12\!\id$.


\begin{proof}[Proof of Proposition \ref{prp}] 
We work over $U=U_{12}$. The homotopy $H$ of Definition \ref{Homotopy} is a holomorphic section of $E_1^*|\_U\otimes\_{\cO_U}T_{V_2}|\_U$. Forgetting the complex (and holomorphic) structure we may think of $2H$ as a smooth section of $\mathsf{E}_1^*|\_U\otimes\mathsf{T}\_{\!V_2}|\_U$.\footnote{Throughout this proof, unless $\otimes$ and Hom are otherwise decorated, they will denote $\otimes\_{C^\infty_\R}$ and $\Hom\_{C^\infty_\R}$ respectively. The 2 in $2H$ is to cancel the $\tfrac12$ in \eqref{halfcommute}.} It lifts to a section
$$
\Lambda\ \in\ \Gamma\big(\mathsf{E}_1^*\otimes\psi^*\;\mathsf{T}_{\!V_2}\big)\ \cong\ \Hom\!\(\mathsf{E}_1,\psi^*\;\mathsf{T}_{\!V_2}\),
$$
because $\mathsf{E}_1^*\otimes\psi^*\mathsf{T}_{V_2}\otimes I_{\infty,\R}$ is a $C^\infty_{V_1,\;\R}$-module and hence a fine sheaf by the existence of partitions of unity.
We will prove this $\Lambda$ satisfies \eqref{fye2} and \eqref{pseye2} below.  Together these give \cite[Equation 2.1]{JoKur1}, thus proving that $\Lambda$ defines an equivalence $(V_1^\circ,\psi,\Phi)\sim(\wt V_1^\circ,\wt\psi,\wt\Phi)$ in Joyce's sense. Therefore they give the same morphism of $\mu$-Kuranishi neighbourhoods. \medskip

Firstly, pick a connection $\nabla$ on $\mathsf E_2$ and set
$$
\epsilon\ :=\ \wt\Phi-\Phi-\psi^*\nabla(s_2)\circ\Lambda\ \in\ \Hom(\mathsf E_1,\psi^*\mathsf E_2).
$$
Here we pair $\Lambda\in\Gamma(\mathsf{E}_1^*\otimes\psi^*\mathsf{T}_{V_2})$ and $\psi^*\nabla(s_2)\in\Gamma(\psi^*(\mathsf T^*_{V_2}\otimes\mathsf{E}_2))$ by contracting across $\psi^*\mathsf{T}_{V_2}$.
Now $\psi|_U=\id_U$, so \eqref{phieqn} and \eqref{halfcommute} imply that $\epsilon$ vanishes on $U$,
$$
\epsilon\ \in\ 
\Gamma\(\mathsf{E}_1^*\otimes\psi^*\mathsf{E}_{2}\otimes I_{\infty,\R}\).
$$
We claim this means $\epsilon=O(s_1)$. The morphism $\mathsf{E}_1^*\rt{s_1}I_{\infty,\R}$ is onto by \eqref{realsec}, so tensoring with $\mathsf{E}_1^*\otimes\psi^*\mathsf{E}_2$ gives the surjection
$$
\mathsf{E_1^*\otimes\psi^*E_2\otimes E_1^*}\rt{1\otimes1\otimes s_1\!}\hspace{-3mm}\to\mathsf{E_1^*\otimes\psi^* E_2}\otimes I_{\infty,\R}.
$$
The kernel is a $C^\infty_{\R}$-module and so fine. Thus $\epsilon$ is in the image of a section of the left hand side, which, by \cite[Definition 2.1(i)]{JoKur1}, means it is $O(s_1)$:
\beq{fye2}
\wt\Phi\=\Phi+\psi^*\(\nabla s_2\)\circ\Lambda+O(s_1).
\eeq
\smallskip

For the second identity we note that by \eqref{psieqn2} and \eqref{halfcommute},
$$
\wt\psi\=\psi+\Lambda[s_1]\quad\text{on }2U_{V_1}.
$$
\smallskip
Here we pair $\Lambda\in\Gamma(\mathsf{E}_1^*\otimes\psi^*\mathsf{T}_{V_2})$ and $s_1\in\Gamma(\mathsf{E}_1\otimes I_{\infty,\R})$ by \vspace{-1mm} contracting across $\mathsf{E}_1$ to give
$$
\Lambda[s_1]\ \in\ \Gamma\big(\psi^*\mathsf{T}_{V_2}\otimes I_{\infty,\R}\big)
$$
whose restriction to $2U$ is --- by \eqref{halfcommute} --- the contraction
$$
H[s_1]\ \in\ \Gamma\(T_{V_2}|_U\otimes_{\cO_U}\cI\)\ \subset\ \Gamma\(\psi^*T_{V_2}|_{2U}\)\ \subset\ \Gamma\(\psi^*\mathsf{T}_{V_2}|_{2U}\)
$$ 
of \eqref{Hs1}. By \eqref{ans1} we know
$$
f\circ\wt\psi-f\circ\psi-\Langle\psi^*(df),\,H[s_1]\Rangle\=0 \ \text{ on }\ 2U_{V_1}
$$
for any $f\in \cO_{2U_{V_2}}$, so
\beq{DIFFREF}
f\circ\wt\psi-f\circ\psi-\Langle\psi^*(df),\,\Lambda[s_1]\Rangle
\ \in\ I_1^2+\,\overline{\!I\;}^2_{\!1}\ \subset\ (I_\infty)^2
\eeq
for any $f\in\cA_2(V_2)$. By taking complex conjugates it also holds for $\overline f\in\cA_2(V_2)$.
We claim this implies \eqref{DIFFREF} holds for any $f\in C^\infty_{V_2}$.

By working locally and using Proposition \ref{sane} we may assume $V_2$ is an open neighbourhood of $0\in\C^n$ with holomorphic coordinates $x_1,\dots,x_n$ which lie in $\cA_2$. Given $f\in C^\infty_{V_2}$, by Hadamard's Lemma we can write
\begin{align}\nonumber
f({\bf x}+{\bf y})-f({\bf x})\=\,&\sum_{i=1}^n\left(\frac{\partial f}{\partial x_i}({\bf x})y_i+\frac{\partial f}{\partial\overline x_i}({\bf x})\overline y_i\right) \\ \label{Hadam2} &+\sum_{i,j=1}^n\left(a_{ij}({\bf x},{\bf y})y_iy_j+b_{ij}({\bf x},{\bf y})y_i\overline y_j+c_{ij}({\bf x},{\bf y})\overline y_i\overline y_j\right)
\end{align}
for ${\bf x},\,{\bf y}$ in some smaller neighbourhood $\mathring V_2$ of $0\in V_2$ and some $a_{ij},b_{ij},c_{ij}\in C^\infty_{\mathring V_2\times\mathring V_2}$.
We apply this to ${\bf x}=\psi$ and ${\bf y}=\widetilde\psi-\psi$ so the left hand side of \eqref{Hadam2} is $f\circ\widetilde\psi-f\circ\psi$. The first term on the right hand side becomes
\beqa
&& \!\sum_{i=1}^n\left[\frac{\partial f}{\partial x_i}({\bf x})\(x_i\circ\widetilde\psi-x_i\circ\psi\)+\frac{\partial f}{\partial\overline x_i}({\bf x})\(\overline x_i\circ\widetilde\psi-\overline x_i\circ\psi\)\right] \\
\!&\stackrel{\eqref{DIFFREF}}=&\!
\sum_{i=1}^n\left[\frac{\partial f}{\partial x_i}({\bf x})\Langle\psi^*(dx_i),\,\Lambda[s_1]\Rangle+\frac{\partial f}{\partial\overline x_i}({\bf x})\Langle\psi^*(d\overline x_i),\,\Lambda[s_1]\Rangle\right]\hspace{-3mm}\mod(I_\infty)^2 \\
&=& \Langle\psi^*(df),\,\Lambda[s_1]\Rangle\hspace{-1mm}\mod(I_\infty)^2.
\eeqa
In the second line we applied \eqref{DIFFREF} to $x_i\in\cA_2(V_2)$ and $\overline x_i$, giving a formula that also shows the second sum in \eqref{Hadam2} lies in $(I_\infty)^2$ because $s_1$ lies in $I_\infty$.

Thus \eqref{DIFFREF} holds for any $f\in C^\infty_{V_2}$. In particular, for any $f\in C^\infty_{V_2,\R}$,
$$
\varepsilon\ :=\ f\circ\wt\psi-f\circ\psi-\Langle\psi^*(df),\,\Lambda[s_1]\Rangle\ \in\ (I_{\infty,\R})^2.
$$
But $\mathsf{E}_1^*\rt{s_1}I_{\infty,\R}$ is onto, so tensoring with $\mathsf E_1^*$ shows that
\beq{os21}
\mathsf{E_1^*\otimes E_1^*}\rt{s_1\otimes s_1}(I_{\infty,\R})^2
\eeq
is also onto. Thus $\varepsilon$ lifts to a section $\sigma$ of $\mathsf{E_1^*\otimes E_1^*}$ such that $\varepsilon=\langle\sigma,s_1\otimes s_1\rangle$, proving that
\beq{os22}
f\circ\wt\psi\=f\circ\psi+\Langle\psi^*(df),\,\Lambda[s_1]\Rangle+O(s_1^2).
\eeq
In the language of \cite[Definition 2.1(v)]{JoKur1} this says
\beq{pseye2}
\wt\psi\=\psi+\Lambda[s_1]+O(s_1^2).\qedhere
\eeq
\end{proof}


\section{Perfect obstruction theories} In this Section we will show that locally --- on Stein open sets --- complex Kuranishi neighbourhoods are equivalent to perfect obstruction theories.

Let $M$ be a scheme with  truncated cotangent complex $\LL_M$. Given any embedding $M\into V$ into a smooth variety with ideal $I$, the complex $I/I^2\to\Omega_V|_M$ is a representative of $\LL_M$.

Recall \cite{BF} a perfect obstruction theory on $M$ is a ``virtual cotangent bundle" $E\udot\in D^b(\mathrm{Coh}\,M)$ of perfect amplitude\footnote{For us, $M$ will always have enough locally free sheaves, so that we can resolve $E\udot$ and replace it by a quasi-isomorphic two term complex of vector bundles $E\udot=\{E^{-1}\to E^0\}$.} $[-1,0]$ and a morphism
\beq{pot}
E\udot\To\LL_M
\eeq
in $D^b(\mathrm{Coh}\,M)$ which induces an isomorphism $h^0(E\udot)\rt\sim h^0(\LL_M)=\Omega_M$ and a surjection $h^{-1}(E\udot)\onto h^{-1}(\LL_M)$.

By a \emph{morphism} between perfect obstruction theories $F\udot\to\LL_M$ and $E\udot\to\LL_M$ we mean a commutative diagram
\beq{mapot}\xymatrix@R=15pt@C=5pt{
F\udot \ar[rr]\ar[dr]&& E\udot\ar[dl] \\ & \LL_M}
\eeq
in $D^b(\mathrm{Coh}\,M)$. We call it an isomorphism if the horizontal arrow is a quasi-isomorphism.

\begin{thm}\label{affKurpot} A global complex Kuranishi chart on $M$ endows it with a perfect obstruction theory.

When $M$ is Stein this sets up one-to-one correspondences between\footnote{Holomorphic Kuranishi charts are defined the same way as $\mu$- and complex Kuranishi charts by using holomorphic data everywhere. We do not give the details as we do not use them in this paper except to state a stronger result in this Theorem.}
\begin{itemize}
\item global complex Kuranishi charts, global holomorphic Kuranishi charts, and perfect obstruction theories,
\item isomorphisms of global complex Kuranishi charts, holomorphic isomorphisms of global holomorphic Kuranishi charts, and isomorphisms of perfect obstruction theories.
\end{itemize}
\end{thm}

\begin{proof}
We begin with the easy part: showing that global complex Kuranishi charts induce perfect obstruction theories, and that morphisms between global complex Kuranishi charts induce morphisms between perfect obstruction theories.
This part works for any scheme $M$; we do not need it to be Stein.
\smallskip

\noindent\textbf{K$\so$P.}
Fix a global complex Kuranishi chart $(V,E,s,\iota)$ for $M$. (Throughout we will suppress $\iota$, just writing $s^{-1}(0)=M\subset V$.) Since $(E,s)$ are $\cA_1$ we know $E|_M$ \eqref{restr} and $s|_M\in\Gamma(E^*|_M\otimes\cI)$ \eqref{restrE} are holomorphic. This defines the perfect obstruction theory
$$
\xymatrix@R=15pt{
E^*|_M \ar[r]^-{ds}\ar@{->>}[d]_s& \Omega_V|_M \ar@{=}[d] \\
\cI \ar[r]^-d& \Omega_V|_M.\!}
$$
Now fix a morphism of global complex Kuranishi charts for $M$. By Definition \ref{morf} this is an equivalence class of the maps of Definition \ref{KNdef}. We pick one $(\psi,\Phi)\colon(V,E,s)\to(W,F,t)$, where $\psi$ is an $\cA_2$-map from an open neighbourhood $V^\circ\subset V$ of $M$ to $W$, and $\Phi$ is an $\cA_1$ map $E\to\psi^*F$ such that $\Phi(s)=\psi^*(t)+O(s^2)$. Then we get maps of complexes
\beq{mappots}
\xymatrix@R=5pt{
E^*|_M \ar[r]^-{ds}\ar@{->>}[dd]_s& \Omega_V|_M \ar@{=}[dd]&&& 
\psi^*F^*|_M \ar[r]^-{\psi^*dt}\ar@{->>}[dd]_{\psi^*t}& \psi^*\Omega_W|_M \ar@{=}[dd] \\
&&&\ar@{=>}[l]_{\ \Phi^*}^{\ \psi^*} \\
\cI \ar[r]^-d& \Omega_V|_M &&& \psi^*(\cJ) \ar[r]^-{\psi^*d}& \psi^*\Omega_W|_M\!}
\eeq
by applying $\Phi^*$ between the top left corners of the diagrams and $\psi^*$ between the other corners. Here $\cJ$ is the ideal of $M\subset 2M_W$ generated by $t$. The key point is that $\Phi(s)=\psi^*(t)+O(s^2)$ implies
$$
[\Phi(s)]\=[\psi^*(t)] \quad\text{in}\quad \psi^*F^*|_M\otimes I_2/I_1^2\=\psi^*F^*|_M\otimes\cI,
$$
so \eqref{mappots} commutes.
Since $\psi|_M=\id\colon M\to M$ and the cotangent complex is functorial, the map on the bottom row is $\id\colon\LL_M\to\LL_M$ in $D^b(\mathrm{Coh}\,M)$. Thus \eqref{mappots} is a morphism of perfect obstruction theories in the sense of \eqref{mapot}.

Given another map of complex Kuranishi charts in the same equivalence class, by Definition \ref{Homotopy} the maps of complexes on the top row of \eqref{mappots} differ by a chain homotopy $H$. Thus they induce the same map in $D^b(\mathrm{Coh}\,M)$. Given an isomorphism of complex Kuranishi charts we can apply this construction to the inverse morphism to get a map of perfect obstruction theories which is inverse to the one we already produced (since the above construction respects composition). Thus isomorphisms induce isomorphisms.
\smallskip


\noindent\textbf{P$\so$K.} Conversely, suppose $M$ is Stein and carries a perfect obstruction theory \eqref{pot}; we will construct a global \emph{holomorphic} Kuranishi chart that induces it. Fix an embedding $M\into V$ into a smooth variety with ideal $I$. Shrinking $V$ if necessary we may assume that $V$ is also Stein. Since $M$ is Stein, $E\udot$ admits a global free resolution $\{E^{-1}\to E^0\}$ and the morphism \eqref{pot} is represented by a genuine map of complexes
$$\xymatrix@R=15pt{
E^{-1} \ar[r]\ar[d]& E^0 \ar[d] \\
I/I^2 \ar[r]^-d& \Omega_V|_M.}
$$
By adding the acyclic complex $\cO_M^{\oplus N}\rt\sim\cO_M^{\oplus N}$ to $E\udot$ we may assume that the vertical arrows are surjective. Thus $K:=\ker\big(E^0\to\Omega_V|_M\big)$ is locally free. Since $h^0(E\udot)\to h^0(\LL_M)$ is an isomorphism and $M$ is Stein the bundle injection $K\into E^0$ lifts to a bundle injection $K\into E^{-1}$; replacing both $E^i$ by their (locally free) quotients by $K$ we have now represented \eqref{pot} by
\beq{bed}
\xymatrix@R=15pt{
E^{-1} \ar[r]\ar@{->>}[d]& \Omega_V|_M \ar@{=}[d] \\
I/I^2 \ar[r]^-d& \Omega_V|_M.\!}
\eeq
The surjectivity of the left hand vertical arrow follows from the isomorphism $h^0(E\udot)\rt\sim h^0(\LL_M)$ and the surjectivity of $h^{-1}(E\udot)\onto h^{-1}(\LL_M)$.

By shrinking $V$ if necessary we may assume it admits a holomorphic bundle $E$ such that $E|_M\cong(E^{-1})^*$; see Proposition \ref{steinbdl} in Appendix \ref{appA}.
Since $I$ surjects onto (the pushforward from $M$ to $V$ of) $I/I^2$ we may choose a lift of $E^{-1}\onto I/I^2$ to a holomorphic map
$$
E^*\rt{s}\hspace{-3mm}\to I
$$
which is a surjection after possibly shrinking $V\!$, by the Nakayama Lemma. Therefore $s$ is a section of $E$ cutting out $M\subset V$, giving us a global holomorphic Kuranishi chart 
\beq{gchart}
(V,E,s,\iota),
\eeq
where $\iota$ is the inclusion $M\subset V$. And its associated per\-fect obstruction theory is \eqref{bed} by construction, and this is the one \eqref{pot} we started with.\medskip

Next suppose we have two global complex Kuranishi charts $(V,E,s)$ and $(W,F,t)$, and an isomorphism of the induced perfect obstruction theories on $M$: a commutative diagram in $D^b(\mathrm{Coh}\,M)$
\beq{mappp}
\xymatrix@R=15pt{ E\udot \ar[d]& F\udot \ar[d]\ar[l]_\sim \\
\LL_M & \LL_M,\!\! \ar[l]_-{\id}}
\eeq
where $E\udot=\{E^*|_M\rt{ds}\Omega_V|_M\}$ and $F\udot=\{F^*|_M\rt{dt}\Omega_W|_M\}$.
We need to show that $(V,E,s),\ (W,F,t)$ are canonically isomorphic as complex Kuranishi charts, and furthermore that the isomorphism can be taken to be holomorphic if $(V,E,s)$ and $(W,F,t)$ are holomorphic.

Since $M$ is Stein and $E\udot,\,F\udot$ are complexes of vector bundles, we can find holomorphic maps of complexes representing three of the arrows in \eqref{mappp}:
\beq{2dgs}
\xymatrix@C=2pt{
& E\udot \ar[d]+<0pt,12pt>&&&& F\udot \ar[llll]\ar[d]+<0pt,12pt> \\
\big\{\cI \ar[rr]_-d&& \Omega_V|_M\big\} &\ \ &\ar@{:>}[ll] \big\{\cJ \ar[rr]_-d&& \Omega_W|_M\big\},\!}
\eeq
where $\cJ$ is the ideal sheaf of $M\subset 2M_W$. The dashed arrow is currently just a map in the derived category. Define the sheaf
$$
K\ :=\ \ker\big(F^*|_M\Onto\cJ\big)\ =\ \ker\big(h^{-1}(F\udot)\Onto h^{-1}(\LL_M)\big).
$$
Consider the diagonal map of complexes in \eqref{2dgs} defined by the composition $F\udot\to E\udot\to\{\cI\to\Omega_V|_M\}$. It kills $K$ so factors through the quotient $\{(F^*|_M)/K\to\Omega_W|_M\}$ of $F\udot$ by $K$. But this is $\{\cJ\to\Omega_W|_M\}$.

Therefore we have filled in the horizontal dotted arrow in \eqref{2dgs} with a genuine map of complexes representing $\id\colon\LL_M\to\LL_M$. Mapping each of the two complexes on the bottom row of \eqref{2dgs} to its first term in degree $-1$ thus gives a commutative diagram
\beq{KScom}\xymatrix@R=15pt{
\LL_M \ar[d] & \LL_M \ar[l]_\id\ar[d] \\ \cI[1] & \cJ[1].\! \ar[l]_\alpha}
\eeq
As in Section \ref{2smooth} the vertical arrows are the Kodaira-Spencer extension classes in
$$
\Ext^1(\LL_M,\cI) \quad\text{and}\quad \Ext^1(\LL_M,\cJ)
$$
of the thickenings $2M_V$ and $2M_W$ of $M$ by the square zero ideals $\cI$ and $\cJ$ respectively. Since $\alpha$ \eqref{KScom} maps one to the other, it induces a map 
\beq{2map}
2M_V\To2M_W
\eeq
between the thickenings (restricting to the identity on $M\subset2M_V$) by the functoriality of their description \eqref{sqext}. It fits into the commutative diagram
$$
\xymatrix@R=15pt{
0 \ar[r]& \cJ \ar[r]\ar[d]_\alpha& \cO_{2M_W} \ar[r]\ar[d]& \cO_M \ar@{=}[d]\ar[r] & 0 \\
0 \ar[r]& \cI \ar[r]& \cO_{2M_V} \ar[r]& \cO_M \ar[r] & 0.\!}
$$
Since $M$ is Stein we may assume that $V$ and $W$ are too. Let $I,\,J$ denote the ideals of $M$ inside them, so $\cI=I/I^2$ and $\cJ=J/J^2$.
Since $W$ is smooth we can lift \eqref{2map} to a holomorphic map
\beq{psi9}
\psi\ \colon\,V\To W
\eeq
after possibly shrinking $V$; see Proposition \ref{steinext} in Appendix \ref{appA}. By Proposition \ref{steinbdl} we may further assume the holomorphic bundles and map $E\to F$ in (the dual of) \eqref{2dgs} are the restrictions to $M$ of holomorphic bundles (which we denote by the same letters $E,F$) and a holomorphic map
\beq{phi9}
\Phi\ \colon\,E\To\psi^*F \ \text{ over }\,V.
\eeq
So we can now compare $\Phi(s)$ and $\psi^*(t)$. Restricting $s,\,t$ to $2M_V,\,2M_W$ they become sections $[s],\,[t]$ of $E^*|_M\otimes I/I^2$ and $F^*|_M\otimes J/J^2$ respectively. The commutativity of \eqref{2dgs} says that $\psi^*[t]=\Phi[s]$, so
$$
\psi^*(t)\ =\ \Phi(s)\mod \psi^*F\otimes I^2.
$$
But the surjectivity of $E^*\onto I$ implies the surjectivity of
$$
\psi^*F\otimes E^*\otimes E^*\rt{1\otimes s\otimes s}\hspace{-3mm}\to\psi^*F\otimes I^2\To 0,
$$
so there exists a section $\pi\in\Gamma(\psi^*F\otimes E^*\otimes E^*)=\Hom(E\otimes E,\psi^*F)$ such that
$$
\psi^*(t)\ =\ \Phi(s)+\pi(s\otimes s).
$$
By Definition \ref{defO} this says $\psi^*(t)=\Phi(s)+O(s^2)$, so we get a morphism $\Psi\colon(V,E,s)\to(W,F,t)$ of complex Kuranishi neighbourhoods in the sense of Definition \ref{KNdef}. This morphism induces the map of perfect obstruction theories \eqref{mappots}, which by construction is the map \eqref{mappp} we started with.

We claim the equivalence class of this morphism $\Psi$ is unique. Fix another morphism $\Psi'\colon(V,E,s)\to(W,F,t)$ inducing the same map $F\udot\to E\udot$ in $D^b(\mathrm{Coh}\,M)$. Since $M$ is Stein, $\Psi$ and $\Psi'$ induce the same \emph{homotopy class} of maps of chain complexes $F\udot\to E\udot$ across the top of \eqref{2dgs}.
Thus the morphisms of chain complexes
$$
(E\udot)^\vee\To\big(F\udot\big)^\vee
$$
induced by $\Psi$ and $\Psi'$ differ by a homotopy $H$, so $\Psi\sim\Psi'$ are equivalent by Definition \ref{Homotopy}.

Applying the same construction to the map of perfect obstruction theories inverse to \eqref{mappp}, and noting it is compatible with compositions, shows the unique morphism $\Psi$ is an isomorphism of global complex Kuranishi charts. Thus 
the global complex Kuranishi chart \eqref{gchart} is unique up to unique isomorphism. And we have shown everything can be chosen to be holomorphic, since $\psi$ and $\Phi$ in \eqref{psi9} and \eqref{phi9} are holomorphic. \medskip

Finally, replacing $F\udot$ by $E\udot$ in \eqref{mappp}, we have shown that an automorphism of perfect obstruction theories induces a unique automorphism of global complex Kuranishi charts. And again it can be chosen to be a holomorphic automorphism of global holomorphic Kuranishi charts.
\end{proof}

\section{Complex Kuranishi spaces} Now that we have the local models for complex Kuranishi spaces (i.e. local complex Kuranishi neighbourhoods) we show how to glue them using coordinate changes.

\begin{defn}\label{cc} Fix complex Kuranishi neighbourhoods $(V_1,E_1,s_1,\iota_1)$ and $(V_2,E_2,s_2,\iota_2)$ of $U_1,\,U_2$ respectively. A \emph{coordinate change} between them is a morphism $(V_1,E_1,s_1,\iota_1)\to(V_2,E_2,s_2,\iota_2)$ over $U_{12}$ which admits a morphism in the opposite direction such that the two compositions are equivalent to $\id_{(V_i,E_i,s_i,\iota_i)}$ over $U_{12}$.
\end{defn}

Notice these coordinate changes do \emph{not} preserve the dimension of the ambient space $V$ in general, though they do preserve the virtual dimension
\beq{v}
v\ :=\ \dim_{\C}V-\rk E
\eeq
because a homotopy between maps, as defined in Definition \ref{Homotopy}, induces a quasi-isomorphism between virtual tangent bundles $\{T_V|_U\to E_1|_U\}$. 

We define a complex Kuranishi space to be a set of complex Kuranishi neighbourhoods glued together by coordinate changes satisfying a cocyle condition. For our purposes it will be sufficient to restrict attention to those for which the underlying scheme is projective. 

\begin{defn}\label{cK}  Let $M$ be a complex projective scheme. A \emph{complex Kuranishi structure} on $M$ of virtual dimension $v$ is
\begin{itemize}
\item a finite open cover $\{U_i\}_{i\in I}$ for $M$,
\item a complex Kuranishi neighbourhood structure $(V_i,E_i,s_i,\iota_i)$ of virtual dimension $v$ on each $U_i$, and
\item for each $i,j\in I$ such that $U_{ij}\ne\emptyset$, a coordinate change $(V_{ij}^\circ,\psi_{ij},\Phi_{ij})$ from $(V_i,E_i,s_i,\iota_i)$ to $(V_j,E_j,s_j,\iota_j)$ over $U_{ij}$,
\end{itemize}
such that $(V_{jk}^\circ,\psi_{jk},\Phi_{jk})\circ(V_{ij}^\circ,\psi_{ij},\Phi_{ij})\sim(V_{ik}^\circ,\psi_{ik},\Phi_{ik})$ over $U_{ijk}$.
\end{defn}

We call the data of $M$, plus a complex Kuranishi structure on it, a ``complex Kuranishi space".

We do not need a definition of morphisms in the full generality of Joyce's morphisms of $\mu$-Kuranishi spaces \cite[Definition 2.21]{JoKur1} since our underlying space $M$ is fixed. For our purposes it is sufficient to define isomorphisms of complex Kuranishi structures on $M$. So suppose we have two complex Kuranishi structures $\big(I,U_i,(V_i,E_i,s_i,\iota_i),(V_{ij}^\circ,\psi_{ij},\Phi_{ij})\big)$ and $\big(I',U_i',(V_i',E_i',s_i',\iota_i'),(V'_{ij}\!^\circ,\psi'_{ij},\Phi'_{ij})\big)$. An isomorphism between them is just compatible coordinate changes between them -- i.e. coordinate changes which agree on overlaps.

\begin{defn} An isomorphism between two complex Kuranishi structures is, for each $i\in I,\,i'\in I'$ such that $U_i\cap U'_{i'}\ne\emptyset$, a coordinate change
$$
f_{ii'}\colon(V_i,E_i,s_i,\iota_i)\To(V'_{i'},E'_{i'},s'_{i'},\iota'_{i'})
$$
over $U_i\cap U'_{i'}$ in the sense of Definition \ref{cc}, satisfying two conditions.
\begin{itemize}
\item For every $j\in I$ with $U_{ij}\cap U'_{i'}\ne\emptyset$,
$$
f_{ji'}\circ(V_{ij}^\circ,\psi_{ij},\Phi_{ij})\ \sim\ f_{ii'} \quad\mathrm{over}\quad U_{ij}\cap U'_{i'}.
$$
\item For every $j'\in I'$ with $U_i\cap U'_{i'j'}\ne\emptyset$,
$$
f_{ij'}\ \sim\ (V'_{i'j'}\hspace{-2.5mm}^\circ\,,\psi'_{i'j'},\Phi'_{i'j'})\circ f_{ii'} \quad\mathrm{over}\quad U_i\cap U'_{i'j'}.
$$
\end{itemize}
\end{defn}

These are inverted and composed in the obvious way, giving an equivalence relation on complex Kuranishi structures on $M$.
Since the definition follows \cite[Definition 2.21]{JoKur1} verbatim (only differing in our use of $\cA$-functions in place of $C^\infty$) it is clear that \emph{isomorphic complex Kuranishi structures on $M$ define isomorphic $\mu$-Kuranishi spaces in the sense of Joyce}.

\subsection{Equivalence with weak perfect obstruction theories}
So complex Kuranishi spaces are made by gluing complex Kuranishi charts. Refining these charts, we may assume that they are \emph{Stein}. Then they are equivalent to perfect obstruction theories, and the gluing maps are equivalent to isomorphisms of perfect obstruction theories, as we showed in Theorem \ref{affKurpot}. Therefore complex Kuranishi spaces are equivalent to local perfect obstruction theories, glued together by isomorphisms.

\begin{defn}\label{wpot} A \emph{weak perfect obstruction theory} on a projective scheme $M$ is a collection of perfect obstruction theories $E_i\udot\to\LL_{\;U_i}$ on a Stein open cover $U_i$ of $M$, together with isomorphisms (``gluing maps") on overlaps $U_{ij}$,
\beq{is}
\Phi_{ij}\ \in\ \Hom_{D^b(\mathrm{Coh}\,U_{ij})\!}\big(E\udot_i|\_{U_{ij}},E\udot_j|\_{U_{ij}}\big)
\eeq
which intertwine the maps to $\LL_{\;U_{ij}}$ and satisfy the cocycle condition
\beq{cocyle}
\Phi_{jk}\big|_{U_{ijk}}\circ\Phi_{ij}\big|_{U_{ijk}}\ =\ \Phi_{ik}\big|_{U_{ijk}}\quad\mathrm{in}\ D^b(\mathrm{Coh}\,U_{ijk}).
\eeq
\end{defn}

In this language, Theorem \ref{affKurpot} proves Theorem \ref{K=P}: that weak perfect obstruction theories are equivalent to complex Kuranishi structures.

In particular, a complex Kuranishi structure is weaker than a perfect obstruction theory.\footnote{Dominic Joyce pointed out it is a kind of a sheafification of a perfect obstruction theory, and suggested the example of a smooth scheme $M$ with an obstruction bundle $F$ and class $\alpha\in\Ext^2(\Omega_M,F)$. This defines a triangle $F[1]\to E\udot_\alpha\to\Omega_M$ and so a perfect obstruction theory. It depends on $\alpha$ but not as a \emph{weak} perfect obstruction theory.} Since the $E\udot_i$ are not sheaves the $\Phi_{ij}$ are \emph{not} enough data to glue them to a global perfect obstruction theory. However the local obstruction sheaves Ob$_i=h^1((E\udot_i)^\vee)$ are sheaves, so do glue via the automorphisms $h^1(\Phi_{ij}^\vee)$ to define a global obstruction sheaf Ob.

In fact there is now a slew of definitions close to, but not equivalent to, perfect obstruction theories, all of which give virtual cycles. They are totally ordered $(1)\Rightarrow(2)\Rightarrow(3)\Rightarrow(4)\Rightarrow(5)$ as follows.
\begin{enumerate}
\item \emph{Quasi-smooth derived structure} \cite{TVe}. The truncation $\pi_0$ gives $M$ a perfect obstruction theory, but in general there is no converse \cite{Sch}.
\item \emph{Perfect obstruction theory} \cite{BF}.
\item \emph{Weak perfect obstruction theory} $\cong$ \emph{complex Kuranishi structure}.
\item \emph{Almost perfect obstruction theory} \cite{KS}. This requires the isomorphisms \eqref{is} \'etale locally but relaxes the cocycle condition \eqref{cocyle}, requiring only that it holds on for the $h^1(\Phi_{ij}^\vee)$ --- so the obstruction sheaf $h^1((E\udot)^\vee)$ glues.
\item \emph{Semi-perfect obstruction theory} \cite{CL}. Here there need not be isomorphisms \eqref{is} on overlaps\footnote{We thanks Huai-Liang Chang for showing us an explicit example of a semi-perfect obstruction theory ``glued" from \emph{non-isomorphic} perfect obstruction theories.} but there should be isomorphisms of obstruction sheaves satisfying the cocycle condition and respecting obstruction maps in the sense of \cite[Definition 3.1]{CL}.
\end{enumerate}

\section{Global complex Kuranishi charts}\label{global}
Suppose $M$ is projective with perfect obstruction theory $E\udot\to\LL_M$. In this Section we will construct a global complex Kuranishi chart for $M$ such that the weak perfect obstruction theory corresponding to it by Theorem \ref{K=P} is precisely $E\udot\to\LL_M$. This will prove Theorem \ref{1}. We go slowly, step by step.

\subsection*{Normal form of the perfect obstruction theory.} It will be important in Section \ref{Cones} that --- to begin with --- we work entirely in the algebraic category; this will enable us to relate our global complex Kuranishi chart to the Behrend-Fantechi cone. So we begin the same way here. Pick a projective embedding $M\subset\PP^N$ with ideal $J$, so $\{J/J^2\to\Omega_{\PP^N}|_M\}$ is a representative of $\LL_M$. Then replacing $E\udot$ by a sufficiently negative locally free resolution, we may assume $E\udot\to\LL_M$ is represented by a genuine map of complexes of sheaves
$$\xymatrix@=16pt{
\dots \ar[r]& E^{-2} \ar[r]& E^{-1} \ar[r]\ar[d]& E^0 \ar[d] \\
&& J/J^2 \ar[r]& \Omega_{\PP^N}\big|_M.}
$$
We may assume the vertical arrows are surjections by adding acyclic complexes $\cO(-i)\rt{\id}\cO(-i)$ to $E\udot$ for $i\gg0$ if necessary. Then we may replace $E^{-1}$ by its quotient by im$\,E^{-2}$ to get a normal form
\beq{normal}\xymatrix@=16pt{
E^{-1} \ar[r]\ar@{->>}[d]& E^0 \ar@{->>}[d]<-.5ex> \\
J/J^2 \ar[r]^-d& \Omega_{\PP^N}\big|_M.\!\!}
\eeq
Let $K$ be the algebraic vector bundle kernel of the right hand vertical arrow. Since we are working \emph{globally} (rather than over an affine or Stein open set) we cannot make the top row of \eqref{normal} smaller by dividing by the acyclic complex $K\xrightarrow\sim K$ as we did in \eqref{bed}, because $K\into E^0$ will not lift to $K\into E^{-1}$ in general. Instead\footnote{This issue complicates the construction of our global Kuranishi chart. On a first reading it is advisable to pretend that $K=0$, or at least that the Kodaira-Spencer class $e$ \eqref{ee} is zero so we can take our smooth ambient space to be the total space of a bundle $\wt K^*$ over a neighbourhood of $M$ in $\PP^N$ (here $\wt K|_M=K$). In the general case the twist by the Kodaira-Spencer class will mean there is no global holomorphic smooth ambient space, just the first order approximation $2M$ produced in Section \ref{twoM}.} 
we make the bottom row \emph{bigger} by replacing $\PP^N$ by a certain square zero thickening by the sheaf $K$ along $M$.  Let
\beq{ee}
e\ \in\ \Ext^1\!\(\Omega_{\PP^N}\big|_M,K\)\rt{d^*}\Ext^1(J/J^2,K)\ \ni\ d^*(e)
\eeq
be the (Kodaira-Spencer) extension class of the short exact sequence
\beq{eeses}
0\To K\To E^0\To\Omega_{\PP^N}|_M\To0.
\eeq
Then $d^*(e)$ defines an extension
\beq{idef}
\cI\ :=\ J/J^2\times\_{\Omega_{\PP^N}|\_M}\!E^0
\eeq
of $J/J^2$ by $K$ fitting into the commutative diagram
$$\xymatrix@R=15pt{
\,K \ar@{=}[r]& \,K \\
\cI \ar[r]\ar@{->>}[d]\ar@{<-_)}[u]+<0pt,-8pt>& E^0 \ar@{->>}[d]\ar@{<-_)}[u]+<0pt,-8pt> \\
J/J^2 \ar[r]& \Omega_{\PP^N}\big|_M.}
$$
This describes a quasi-isomorphism from the 2-term complex on the central row to that on the bottom row, which is $\LL_M$. And $\cI$ \eqref{idef} admits a natural map from $E^{-1}$ by \eqref{normal}, allowing us to put the perfect obstruction theory into the new algebraic normal form
\beq{normf}\xymatrix@R=3pt{
E\udot\ar[dd] && E^{-1} \ar[r]\ar@{->>}[dd]& E^0 \ar@{=}[dd] \\ &= \\
\LL_M && \cI \ar[r]& E^0.\!}
\eeq
The central vertical arrow is a surjection because $h^{-1}(E\udot)\to h^{-1}(\LL_M)$ is.
\subsection{The global smooth thickening $2M$.}\label{twoM} The lower horizontal complex representing $\LL_M$ in \eqref{normf} has a canonical map to $\cI[1]$, giving a Kodaira-Spencer class
\beq{KSJ}
\LL_M[-1]\To\cI
\eeq
whose cone $E_0$ is locally free. By Proposition \ref{conefree} this defines a smooth holomorphic thickening $2M$ of $M$.

The structure of this $2M$ is easy to understand \emph{locally}. Choose
$$
\text{a cover of } M \text{ by a finite number of Zariski open affine sets } M_i,
$$
sufficiently small that the bundles $E_1,\,K,\,\Omega_{\PP^N}$ are all trivial on restriction to each $M_i$. The extension classes $e,\,d^*(e)$ of \eqref{ee} are zero on each $M_i$ and
$K_i:=K|_{M_i}$ is the restriction of the trivial algebraic bundle $\wt K_i$ on a neighbourhood of $M_i\subset\PP^N$.
Then on $M_i$ the Kodaira-Spencer class \eqref{KSJ} is
\beq{KSf}
\LL_{M_i}[-1]\rt{(f,0)}J/J^2\oplus K,
\eeq
where $f$ is the Kodaira-Spencer class of the smooth thickening $(2M)_{\PP^N}$. Therefore \eqref{KSf} defines the smooth thickening $(2M_i)_{\wt K^*_i}$ of $M_i$ inside\vspace{-1mm} the total space of $\wt K^*_i$. It is the scheme-theoretic union of $(2M_i)_{\PP^N}$ and the first order neighbourhood of $M_i$ inside the total space of $K^*_i$ (and thus independent of the choice of $\wt K_i$).\footnote{This is basically the observation $(x^2,y)\cap(x,y^2)=(x,y)^2$. Precisely it is $(J^2+\m)\cap(J+\m^2)=J^2+\m.J+\m^2\subset\cO_{\PP^N}[z_1,\dots,z_k]$, where $z_1,\dots,z_k$ are local trivialising sections of $K^*_i$ and $\m=(z_1,\dots,z_k)$ is the ideal of the zero section $M_i\subset K_i$.} That is,
\beq{2Mloc}
2M\ \stackrel{\mathrm{loc}}=\ (2M)_{\PP^N}\cup\_M(2M)_{K^*}.
\eeq
It is important to note that in general this is not true \emph{globally} holomorphically because of the twisting by $e$, which we describe next.\medskip

We represent $e\in H^1\(T_{\PP^N}|\_M\otimes K\)$ as a \v Cech cocyle $e_{ij}\in\Gamma\(M_{ij},T_{\PP^N}|\_M\otimes K\)$, and reglue $(2M)_{\PP^N}\cup_M(2M)_{K^*}$ over $M_{ij}$ by the automorphism defined by its action on functions by
\beq{reglue}
\ \cO_{(2M)_{\PP^N}\cup\_M(2M)_{K^*}}\=\cO_{(2M)_{\PP^N}}\oplus K \righttoleftarrow
\qquad (f,g)\Mapsto
\(f,g+\langle df,e_{ij}\rangle\).\!\!
\eeq
Here we consider $df$ as a section of $\Omega_{\PP^N}\big|_M\cong\Omega_{(2M)_{\PP^N}}\big|_M$ by \eqref{Omegiso}.

Since the $e_{ij}$ satisfy the cocyle condition this regluing defines a global scheme $(2M)_{\PP^N}\cup\_M(2M)_{K^*\!,\,e}$ such that
$$
2M\=(2M)_{\PP^N}\cup\_M(2M)_{K^*\!,\,e}\;,
$$
and such that, by construction,
\beq{key}
I_{M/2M}\rt{d}\Omega_{2M}|_M \,\text{ is the map }\, \cI\To E^0 \,\text{ of }\, \eqref{normf}.
\eeq

\subsection{The $\cA_2$-embedding}\label{secA2}
Given a holomorphic bundle $K$ on a scheme $M$ we use the notation
$$
\mathsf{K}\ :=\ K\otimes\_{\cO_M}C^\infty_M
$$
for all three of (i) the $C^\infty$ bundle underlying $K$, (ii) its total space, and (iii) its sheaf of sections.

In Section \ref{twoM} we constructed a scheme $2M$ from the Kodaira-Spencer element $e\in H^1\(T_{\PP^N}|_M\otimes_{\cO_M}K\)$ which is the 
extension class of the sequence of holomorphic bundles \eqref{eeses} on $M$,
$$
0\To K\To E^0\To\Omega_{\PP^N}|_M\To0.
$$
By Theorem \ref{extendE} the holomorphic bundle $K$ over $M$ extends to an $\cA_1$-bundle $\wt{\mathsf K}$ over a neighbourhood $\mathring\PP^N\supset M$. 
In this Section we will show how to give an $\cA_2$-embedding of $2M$ into $\wt{\mathsf K}^*$ to prove the following result.

\begin{thm}\label{summary} The total space of $\wt{\mathsf K}^*$ admits a sheaf of algebras $\cA_2\subset C^\infty_{\wt{\mathsf K}^*}$\vspace{-1mm} and an $\cA_2$-embedding $\phi\colon(2M,\cO_{2M})\into(\wt{\mathsf K}^*\!,\,\cA_2)$.
\end{thm}

In particular we will find $\phi$ induces a $C^\infty$ isomorphism\vspace{-.7mm} $C^\infty_{2M}\cong C^\infty_{(2M)_{\wt{\mathsf K}^*}}$ between the twisted thickening $2M$ and the \emph{untwisted} thickening\footnote{This depends only on the holomorphic data $(2M)_{\PP^N}$ and $K$ by \eqref{2Mloc}.} $\im\phi=(2M)_{\wt{\mathsf K}^*}$ of $M$ inside $\wt{\mathsf K}^*$ (corresponding to the trivial Kodaira-Spencer class $e=0$).\vspace{-.5mm} That is, we can \emph{non-holomorphically} untwist $e$; see Lemma \ref{off} below.\smallskip

Using the $\C^*$ action on $\wt{\mathsf K}^*$ we decompose the smooth functions on it as
$$
C^\infty_{\wt{\mathsf K}^*}\=C^\infty_{\mathring\PP^N}\ \oplus\ \wt{\mathsf K}.C^\infty_{\wt{\mathsf K}^*}
\ +\ \overline{\wt{\mathsf K}}\,.C^\infty_{\wt{\mathsf K}^*}.
$$
Using this we let $I_1$ be the ideal generated by $J$ in the first factor and the whole second factor,
\beq{I1def}
I_1\ :=\ J.C^\infty_{\wt{\mathsf K}^*}\,+\,\wt{\mathsf K}.C^\infty_{\wt{\mathsf K}^*}.
\eeq
This is the standard $I_1\subset C^\infty_{\wt{\mathsf K}^*}$ corresponding to the\vspace{-1mm} standard embedding $M\subset\mathring\PP^N\subset\wt{\mathsf K}^*$ in the zero section, with
$$
\frac{C^\infty_{\wt{\mathsf K}^*}}{I_1^2+\overline{I_1\!}^{\;2}}\=C^\infty_{(2M)_{\wt{\mathsf K}^*}}.
$$
Notice $(2M)_{\wt{\mathsf K}^*}$ is the \emph{untwisted} thickening of $M$ corresponding to the trivial Kodaira-Spencer class $e=0$. Therefore we next modify the holomorphic structure along the first order neighbourhood of $M$ inside $\wt{\mathsf K}^*$ to change it from $(2M)_{\wt{\mathsf K}^*}$ to $2M$. So $I_1$ is standard but the $\cA_2$-structure will not be.

By the exact sequence \eqref{restric2}, to define an $\cA_2$-structure on $\wt{\mathsf K}^*$ with the given $I_1$ \eqref{I1def} we just need an inclusion of algebras $\cO_{2M}\into C^\infty_{\wt{\mathsf K}^*}\big/I_1^2$ with the right properties. We can then set $\cA_2:=C^\infty_{\wt{\mathsf K}^*}\times\_{C^\infty_{\wt{\mathsf K}^*}/I_1^2}\cO_{2M}$.\medskip

For this we use the sheaves of algebras
$$
C^\infty_J\ :=\ \frac{C^\infty_{\mathring\PP^N}}{J.C^\infty_{\mathring\PP^N}}\quad\text{and}\quad C^\infty_{\!J^2}\ :=\ \frac{C^\infty_{\mathring\PP^N}}{J^2.C^\infty_{\mathring\PP^N}}\,.
$$
Note the first is an $\cO_M=\cO_{\mathring\PP^N}/J$-module, so we may form
$$
K_J\ :=\ K\otimes\_{\cO_M}C^\infty_J\=\wt{\mathsf K}\big/J.\wt{\mathsf K}.
$$
Its $C^\infty_J$-module structure makes it a $C^\infty_{\!J^2}$-module, thus defining an algebra $C^\infty_{\!J^2}\oplus K_J$ which by \eqref{I1def} is a subalgebra of $C^\infty_{\wt{\mathsf K}^*}/I_1^2$,
$$
C^\infty_{\!J^2}\oplus K_J\ \subset\ C^\infty_{\wt{\mathsf K}^*}\big/I_1^2.
$$
So we will embed $\cO_{2M}\into C^\infty_{\!J^2}\oplus K_J$. Locally we have the embedding
$$
\cO_{2M_i}\ \cong\ \cO_{(2M_i)_{\PP^N}}\oplus K|_{M_i}\ \Into\ C^\infty_{\!J^2}|_{M_i}\oplus K_J|_{M_i}
$$
by the obvious diagonal map; we will twist this to take account of $e$.

The map 
$H^1\(T_{\PP^N}|_M\otimes_{\cO_M}K)\to H^1\(T_{\PP^N}|_M\otimes_{\cO_M}K_J\)=0$ takes $e$ to zero since $C^\infty_J$ is fine. Thus we can write the \u Cech cocyle $(e_{ij})$ as a non-holomorphic \u Cech coboundary: there exists
\beq{mathsfei2}
e_i\,\in\,\Gamma\(M_i,\,T_{\PP^N}|_M\otimes_{\cO_M}K_J\)
\ \text{ such that }\ e_i|\_{M_{ij}}-e_j|\_{M_{ij}}\=e_{ij}.
\eeq
These $e_i$ define \emph{non-holomorphic} algebra automorphisms
\beq{phii}
\phi_i\,\colon\,C^\infty_{\!J^2}\oplus K_J\,\righttoleftarrow \qquad (f,k)\ \Mapsto\ \(f,\,k+\Langle\partial f,e_i\Rangle\).
\eeq
These \emph{untwist} the gluing \eqref{reglue} in the sense that $\Phi_{ij}=\phi_j^{-1}\big|_{M_{ij}}\circ\phi_i\big|_{M_{ij}}$ by \eqref{mathsfei2}. Therefore, if we map $\cO_{2M_j}$ to $(C^\infty_{\!J^2}\oplus K_J)|_{M_j}$ by $\phi_j$ over $M_j$, then over $M_{ij}$ the gluing $\Phi_{ij}$ \eqref{reglue} glues this to $\phi_j\circ\Phi_{ij}=\phi_i$. So we get a well-defined map
\beq{cup}
\phi\=\bigcup\nolimits_i\phi_i\,\colon\,\cO_{2M}\Into C^\infty_{\!J^2}\oplus K_J\ \subset\ C^\infty_{\wt{\mathsf K}^*}\big/I_1^2.
\eeq
Hence we get a sheaf of algebras
\beq{A2df}
\cA_2\ :=\ C^\infty_{\wt{\mathsf K}^*}\times\_{C^\infty_{\wt{\mathsf K}^*}/I_1^2}\phi\(\cO_{2M}\)\ \subset\ C^\infty_{\wt{\mathsf K}^*}
\eeq
which we will check defines an $\cA_2$-embedding of $2M$ as in Definition \ref{UinV}.

\begin{lem}\label{off}
There exists a lift $\wt\phi\colon C^\infty_{2M}\to C^\infty_{(2M)_{\wt{\mathsf K}^*}}$ of the\vspace{-4pt} composition
$\phi'\colon\cO_{2M}\stackrel\phi\into C^\infty_{\wt{\mathsf K}^*}\big/I_1^2\onto C^\infty_{\wt{\mathsf K}^*}\big/\(I_1^2+\overline{I_1\!}^{\;2}\)=C^\infty_{(2M)_{\wt{\mathsf K}^*}}$ giving a commutative diagram
$$
\xymatrix@R=15pt@C=40pt{
\cO_{2M} \ar@{^(->}[dr]^-{\phi'} \\
C^\infty_{2M} \ar@{<-_)}[u]-<0pt,10pt>\ar[r]^(.4)\sim_(.4){\wt\phi}& C^\infty_{(2M)_{K^*}}}
$$
with $\wt\phi$ an isomorphism. In particular, $\phi'$ is injective.
\end{lem}

\begin{proof}
We introduce a sheaf of algebras
\beq{Sdef}
C^\infty_{M_{\;\overline{\!J\,}}}\ :=\ \frac{C^\infty_{\mathring\PP^N}}{C^\infty_{\mathring\PP^N}.J+C^\infty_{\mathring\PP^N}.\,\overline{\!J\,}^2}\=
\frac{C^\infty_{(2M)_{\mathring\PP^N}}}{C^\infty_{(2M)_{\mathring\PP^N}}\!.\;J}\,.
\eeq
It is an $\cO_M$-module since it is an $\cO_{\mathring\PP^N}$-module which is annihilated by $J$. Then the local pieces \eqref{2Mloc} of $2M$ have
\beq{lp}
C^\infty_{2M}\ \stackrel{\mathrm{loc}}=\
C^\infty_{(2M)_{\mathring\PP^N}}\oplus\,\(K\otimes\_{\cO_M}C^\infty_{M_{\;\overline{\!J\,}}}\)
\,\oplus\,\overline{\(K\otimes\_{\cO_M}C^\infty_{M_{\;\overline{\!J\,}}}\)}
\,\oplus\,\(\mathsf K\otimes_{C^\infty_M}\overline{\mathsf K}\).
\eeq

These local pieces are glued over overlaps $M_{ij}$ by the holomorphic gluing map $\Phi_{ij}$ \eqref{reglue}. On $C^\infty_{2M_{ij}}$ we claim its action is given by
\begin{align}\nonumber
(f,k_1,k_2,k_3)\ \Mapsto\ \Big(f,\ k_1\,+\,&\Langle\partial f,e_{ij}\Rangle,\ k_2+\Langle\dbar f,\overline e_{ij}\Rangle, \\ \label{formu}
&k_3+\Langle\partial\dbar f,e_{ij}\overline e_{ij}\Rangle+
\Langle\dbar k_1,\overline e_{ij}\Rangle+\Langle\partial k_2,e_{ij}\Rangle\Big).
\end{align}
To make sense of this formula lift $f$ to $C^\infty_{\mathring\PP^N}$ where it is defined modulo $C^\infty_{\mathring\PP^N}.J^2+C^\infty_{\mathring\PP^N}.\,\overline{\!J\,}{}^2$. Thus $\partial f$ is a $C^\infty$ section of $\Omega_{\mathring\PP^N}$ defined modulo $C^\infty_{\mathring\PP^N}.J+C^\infty_{\mathring\PP^N}.\,\overline{\!J\,}{}^2$. (Similarly $\dbar f$ is well-defined modulo $C^\infty_{\mathring\PP^N}.J^2+C^\infty_{\mathring\PP^N}.\,\overline{\!J\,}$.) Hence by \eqref{Sdef}, $\partial f$ is well-defined in $\Omega_{\mathring\PP^N}|_M\otimes_{\cO_M}C^\infty_{M_{\;\overline{\!J\,}}}$. Pairing with the image of $e_{ij}$ under $T_{\mathring\PP^N}|_M\otimes K\to(T_{\mathring\PP^N}|_M\otimes K)\otimes_{\cO_M}C^\infty_{M_{\;\overline{\!J\,}}}$ gives $\Langle\partial f,e_{ij}\Rangle$ in the second summand $K\otimes\_{\cO_M}C^\infty_{M_{\;\overline{\!J\,}}}$ of \eqref{lp}. $\Langle\dbar f,\overline e_{ij}\Rangle$ and $\Langle\partial\dbar f,e_{ij}\overline e_{ij}\Rangle$ are similar.
Next $K\otimes\_{\cO_M}C^\infty_{M_{\;\overline{\!J\,}}}$ has a $\dbar$-operator $1\otimes\dbar$ taking values in $(K\otimes\overline{\Omega}_{\mathring\PP^N}|_M)\otimes_{\cO_M}C^\infty_{\mathring\PP^N}\big/\(C^\infty_{\mathring\PP^N}.J+C^\infty_{\mathring\PP^N}.\,\overline{\!J\,}\)=(K\otimes\overline{\Omega}_{\mathring\PP^N}|_M)\otimes_{\cO_M}C^\infty_M$. This pairs with the image of $\overline{e}_{ij}$ in $(T_{\mathring\PP^N}|_M\otimes K)\otimes_{\cO_M}C^\infty_M$ to give $\Langle\dbar k_1,\overline e_{ij}\Rangle$ in the last summand $\mathsf K\otimes_{C^\infty_M}\overline{\mathsf K}$ of \eqref{lp}. Finally $\Langle\partial k_2,e_{ij}\Rangle$ is similar.
 
To prove the claim \eqref{formu} we use the observation below Definition \ref{SMFTNS} that the map $\Phi_{ij}\colon C^\infty_{2M_{ij}}\righttoleftarrow$ can\vspace{-1pt} be computed locally. Thus by shrinking $M_{ij}\subset\mathring\PP^N$ we may assume $K$ is a trivial bundle and we may replace $\PP^N$ by $\C^N$, to which $K$ extends as $\C^N\times A$ for some vector space $A$. Shrinking $\C^N$ to an open neighbourhood $\mathring\C^N\supset M_{ij}$ we can lift $e_{ij}$ to a holomorphic map $\mathring\C^N\to A^{\oplus N}$, so $\langle f,e_{ij}(x)\rangle\in\C^N$ for $x\in\mathring\C^N$ and $f\in A^*$. Then $\Phi_{ij}$ lifts to the local holomorphic automorphism
$$
\mathring\C^N\times A\ni(x,f)\Mapsto\(x+\langle f,e_{ij}(x)\rangle,\,f\).
$$
Computing how this pulls back $C^\infty$ functions on $\mathring\C^N\times A$ gives the formula \eqref{formu} claimed.

Since the formula \eqref{formu} does not use the holomorphicity of $e_{ij}$ we can now use the (non-holomorphic) $e_i$ of \eqref{mathsfei2} to untwist $C^\infty_{2M}$. The $e_i$ define \emph{non-holomorphic} automorphisms $\wt\phi_i\colon C^\infty_{(2M_i)_{\wt{\mathsf{K}}^*}}\righttoleftarrow$ by a similar formula over $M_i$, taking $(f,k_1,k_2,k_3)$ to
\beq{fii}
\Big(f,\ k_1+\Langle\partial f,e_i\Rangle,\ k_2+\Langle\dbar f,\overline e_i\Rangle,\ k_3+\Langle\partial\dbar f,e_i\overline e_i\Rangle+
\Langle\dbar k_1,\overline e_i\Rangle+\Langle\partial k_2,e_i\Rangle\Big).
\eeq
As the identity plus a nilpotent map, this is invertible. An explicit formula for its inverse is
\begin{align*}
&(f,k_1,k_2,k_3)\Mapsto\Big(f,\ k_1-\Langle\partial f,e_i\Rangle,\ k_2-\Langle\dbar f,\overline e_i\Rangle,\\
&\qquad k_3-\Langle\partial\dbar f,e_i\overline e_i\Rangle-
\Langle\;\dbar \(k_1-\Langle\partial f,e_i\Rangle\),\overline e_i\Rangle-\Langle\partial \(k_2-\Langle\dbar f,\overline e_i\Rangle\),e_i\Rangle\Big).
\end{align*}
Then by \eqref{mathsfei2} we have
$$
\wt\phi_j|_{M_{ij}}\circ\Phi_{ij}\=\wt\phi_i|\_{M_{ij}}.
$$
So it defines a global (non-holomorphic!) isomorphism
\beq{cft}
\wt\phi\,=\,\bigcup\nolimits_i\wt\phi_i\,\colon\,C^\infty_{2M}\rt\sim C^\infty_{(2M)_{K^*}}.
\eeq

Since the local automorphism $\wt\phi_i$ \eqref{fii} is a lift of the composition
$$
\phi'_i\ \colon\,\cO_{2M_i}\stackrel{\eqref{cup}}{\Into} C^\infty_{\wt{\mathsf K}^*}/I_1^2\To C^\infty_{(2M_i)_{\wt{\mathsf{K}}^*}},
$$
the global isomorphism $\wt\phi$ \eqref{cft} is a lift of their gluing $\phi'$.
\end{proof}

So by \eqref{cft} we have an embedding of locally ringed spaces $\(2M,C^\infty_{2M}\)\into\(\wt{\mathsf K}^*,C^\infty_{\wt{\mathsf K}^*}\)$ such that the composition $\cA_2\subset C^\infty_{\wt{\mathsf K}^*}\onto C^\infty_{2M}$ has image $\cO_{2M}$ by \eqref{A2df} and Lemma \ref{off}. It has kernel $I_1^2$ by construction \eqref{A2df}.

Now we show that $I_1$ \eqref{I1def} arises in the way stipulated by Definition \ref{UinV}. As usual let $\cI\subset\cO_{2M}$ be the ideal of $M$, so that
$$
I_2\ :=\ \ker\(\cA_2\To C^\infty_M\)\=C^\infty_{\wt{\mathsf K}^*}\times\_{C^\infty_{\wt{\mathsf K}^*}/I_1^2}\cI
$$
and
$$
C^\infty_{\wt{\mathsf K}^*}.I_2\=C^\infty_{\wt{\mathsf K}^*}\times\_{C^\infty_{\wt{\mathsf K}^*}/I_1^2}\Big(\(C^\infty_{\wt{\mathsf K}^*}/I_1^2\).\;\cI\Big).
$$
Locally, on $M_i$, we have $\cI=J/J^2\oplus K\subset\cO_{(2M)_{\PP^N}}\oplus K$   
so we see from \eqref{I1def} that $\(C^\infty_{\wt{\mathsf K}^*}/I_1^2\).\;\cI=I_1/I_1^2$. Since $\phi_i$ \eqref{phii} clearly preserves $I_1/I_1^2\subset C^\infty_{\wt{\mathsf K}^*}/I_1^2$ we deduce that
$$
C^\infty_{\wt{\mathsf K}^*}.I_2\=C^\infty_{\wt{\mathsf K}^*}\times\_{C^\infty_{\wt{\mathsf K}^*}/I_1^2}\(I_1/I_1^2\)\=I_1.
$$\smallskip

Finally it remains to show that $H\(g_1,\ldots,g_n\)\in\cA_2$ when $g_1,\ldots,g_n\in\cA_2$ and $H\colon\C^n\to\C$ is holomorphic in a neighbourhood of any point of $\im\($\scalebox{1.5}{$\times$}${\!}_j\;g_j\)$. By the definition of $\cA_2$ \eqref{A2df}, it is enough to show that the projection of $H\(g_1,\ldots,g_n\)$ to $C^\infty_{\wt{\mathsf K}^*}/I_1^2$ lies in $\phi\(\cO_{2M}\)$.

Given $h_1,\ldots,h_n\in I_1^2$ and $x\in M$, Hadamard's Lemma applied to $H$ in a neighbourhood of $(g_1(x),\dots,g_n(x))$ gives
$$
H\(g_1+h_1,\ldots,g_n+h_n\)-H\(g_1,\ldots,g_n\)\=\sum\nolimits_ia_ih_i\ \in\ I_1^2,
$$
with no $\bar{h}_i$ terms because the antiholomorphic derivatives $\dbar_iH$ vanish. Thus $H$ descends to a map
$$
\overline H\ \colon\Big(C^\infty_{\wt{\mathsf K}^*}/I_1^2\Big)^{\oplus n}\To C^\infty_{\wt{\mathsf K}^*}/I_1^2,
$$
and we are left with showing that $\overline H$ preserves $\phi(\cO_{2M})\subset C^\infty_{\wt{\mathsf K}^*}/I_1^2$.

This is a local property, and $\overline H$ obviously preserves $\cO_{2M}$, so it is sufficient to prove that it commutes with $\phi_i$ over $2M_i$:
$$
\overline H\(\phi_i(u_1),\ldots,\phi_i(u_n)\)\=\phi_i\(\;\overline H(u_1,\ldots,u_n)\).
$$
This follows from another application of Hadamard's Lemma which, together with the formula \eqref{phii} for $\phi_i$, shows that, working modulo $I_1^2$,
\beqa
\overline H\(\phi_i(u_1),\ldots,\phi_i(u_n)\) \!\!&=&\!\!
H\(u_1,\ldots,u_n\)+\sum\nolimits_j(\partial_j H)\(u_1,\ldots,u_n\).\Langle\partial u_j, e_i\Rangle \\
\!\!&=&\!\! H\(u_1,\ldots,u_n\)+\Langle\partial\(H(u_1,\ldots,u_n)\), e_i\Rangle \\
\!\!&=&\!\! \phi_i\(\;\overline H(u_1,\ldots,u_n)\).
\eeqa

\subsection{The global complex Kuranishi chart}\label{gKc} We can now enhance $M$'s $\cA_2$-embedding of Theorem \ref{summary} to a global complex Kuranishi chart for $M$. Pick trivial algebraic bundles
\beq{tildeEi}
E_i\,\text{ over the total space of }\,\wt K_i^*
\eeq
extending the trivial algebraic bundles $E_1|_{M_i}$, and algebraic sections $s_i$ extending the section of $E_1\otimes\cI|_{M_i}$ given by the central vertical arrow of \eqref{normf}. Restricting to open neighbourhoods $V_i\subset\wt K_i^*$ of $M_i$ gives local holomorphic Kuranishi charts
\beq{lockuran}
(V_i,E_i,s_i).
\eeq
We would like to $\cA$-glue these to give a global chart. By Theorem \ref{extendE} there exists an $\cA_1$-bundle $E$ over some open neighbourhood $V$ of $M\subset\wt K^*$ with $E|_M\cong E_1$ and which is isomorphic over $V_i\cap V$ to $E_i\otimes\_{\cO_{V_i}}\cA_1$\vspace{-1mm} (since it is trivial over $V_i\cap V$). By shrinking $V$ and $V_i$ we may assume that $V_i=\wt K_i^*\cap V$ and that the surjection  $E^{-1}\onto\cI$ on $M$ of \eqref{normf} extends to a surjection of $\cA_1$-modules
$$
E^*\stackrel s\Onto I_2
$$
by Theorem \ref{extsec}. This gives our global complex Kuranishi chart $(V,E,s)$. By construction and \eqref{key} the associated perfect obstruction theory
\beq{Kpot}
\xymatrix@R=18pt{
E^*\big|_M\cong E^{-1} \ar[r]^{ds}\ar[d]_s& \Omega_V\big|_M\cong E^0 \ar@{=}[d]\\
I_2\big/I_1^2\cong\cI \ar[r]& \Omega_V\big|_M}
\eeq
is \eqref{normf}. This proves Theorem \ref{1}. 


\subsection{Purely holomorphic description}\label{phd}
In sum, starting with the local holomorphic charts $(V_i,E_i,s_i)$ \eqref{lockuran}, we have found $\cA_2$-gluings $\psi_{ij}$ of the $V_i$ and $\cA_1$-gluings $\Phi_{ij}$ of the $E_i$. Furthermore we have a section $s$ of the glued bundle $E$ whose restriction $s\colon E^*|_M\to\cI$ to $M_i$ is $s_i\colon E_i^*|_{M_i}\to\cI$. Thus $s|_{V_i}-s_i$ has image in $I_1^2$ so using the same argument as in (\ref{os21}, \ref{os22}),
\beq{close}
s|_{V_i}\=s_i+O\(s_i^2\)
\eeq
in the sense of Definition \ref{defO}. Thus the $(\psi_{ij},\Phi_{ij})$ glue the $s_i$, up to an $O(s_i^2)$ correction, to give $s$. By Definition \ref{KNdef} this means $(\psi_{ij},\Phi_{ij})$ glue the holomorphic Kuranishi charts $(V_i,E_i,s_i)$.

Both $\psi_{ij}$ and $\Phi_{ij}$ satisfy the cocycle condition and so give us the global ambient space $V$ and the global $\cA_1$-bundle $E$ over it. Then we have the global $\cA_1$-section $s$ defining the global Kuranishi chart $(V,E,s)$.

We end by noting that by Theorem \ref{affKurpot} we can replace each of these $\cA$-gluings $(\psi_{ij},\Phi_{ij})$ by a homotopic
\beq{holo'}
\text{ \emph{holomorphic} gluing }\,\(\psi'_{ij},\,\Phi'_{ij}\).
\eeq
The $\psi'_{ij}$ (or $\Phi'_{ij}$) do not generally satisfy the cocycle condition on their own, so we do not get a holomorphic structure on $V$ (or $E$). But as morphisms of complex Kuranishi charts, the pairs \eqref{holo'} do satisfy the cocycle condition up to homotopy in the sense of Definition \ref{morf}. They holomorphically glue the holomorphic Kuranishi charts $(V_i,E_i,s_i)$ to define the same complex Kuranishi structure as the global $\cA$-Kuranishi chart $(V,E,s)$ of Theorem \ref{1}.

\section{Cones}\label{Cones}
Again let $(M,E\udot)$ be a projective scheme with perfect obstruction theory.
From this data Beherend-Fantechi produce a cone $C\subset E_1:=(E^{-1})^*$ by pulling back their intrinsic normal cone from the stack $E_1/E_0$ (called $h^1/h^0\((E\udot)^\vee\)$ in \cite{BF}).

Their construction is algebraic, of course. We can describe it locally in terms of the algebraic Kuranishi charts \eqref{lockuran} made from the perfect obstruction theory. Let $I$ denote the ideal of $M_i\subset\wt K_i$, so that $I/I^2=\cI$. Then the cone is
\beq{coneJ}
\mathrm{Spec}\,\bigoplus_{n\ge0}I^n/I^{n+1}\ \subset\ \mathrm{Spec}\,\bigoplus_{i\ge0}\,\Sym^n(I/I^2)\=\mathrm{Spec}\,\bigoplus_{i\ge0}\,\Sym^n\cI\ \subset\ E_1,
\eeq
This glues over overlaps $M_{ij}$, even though $I$ does not in general. Since $\cI$ and $E_1$ are global sheaves over $M$, the right hand side of \eqref{coneJ} obviously glues. It is less obvious, but also true, that the left hand side glues to give the cone $C\subset E_1$ over $M$.

Denote the inclusion map from the total space of $E|_M$ to the total space of $E$ over $V\subset\wt K^*$ (the global Kuranishi chart of Section \ref{gKc}) by
\beq{j}
j\ \colon\,E_1\Into E. 
\eeq
Inside $E$ we will next show the Borel-Moore fundamental class $j_*[C]$ is homologous to the graph $\Gamma_{\!s}\subset E$ of $s$. For this we will need some notation.

\subsection*{$\epsilon$-neighbourhoods} Fix a projective subscheme $Z$ of a quasi-projective scheme $X$. Pulling back the Fubini-Study metric  from a projective algebraic embedding $X\into\PP^n$ defines a metric $d$ on $X$. Then we let
\begin{equation}\label{retract}
U_X^\epsilon(Z)\ :=\ \big\{x\in X\,\colon\, d(x,Z)<\epsilon\big\}
\end{equation}
with closure $\overline U_{\!X}^{\;\epsilon}(Z)=\{x\in X\colon d(x,Z)\le\epsilon\}$ and boundary
$$
\partial\overline U_{\!X}^{\;\epsilon}(Z)\ :=\ \big\{x\in X\,\colon\, d(x,Z)=\epsilon\big\}.
$$
This is empty if $Z=X$. We will always assume --- but often omit to say --- that $\epsilon>0$ is taken sufficiently small, and we permit ourselves to shrink it during proofs. Thus the choice of $d$ is unimportant, and $Z\subset U_X^\epsilon(Z)$ is a deformation retract. In particular the inclusion induces an isomorphism $H_*(Z)\cong H_*\big(U_X^\epsilon(Z)\big)$. 

\begin{prop}\label{BJ2} In the total space of $E$ over $V\!$, let $U=U_{E}^\epsilon(M)$. Then 
\begin{equation}\label{propr}
\big[\Gamma_{\!s}\cap\overline U\;\big]\ =\ \big[j_*\;C\cap\overline U\;\big]\ \in\ H_*\big(\,\overline U,\,\partial\overline U\take 0_{E}\big).
\end{equation}
Here $0_E$ is the zero section of $E$ and $\partial\overline U\take 0_{E}$ is shorthand for $\partial\overline U\take(\partial\overline U\cap0_E)$.
\end{prop}

\begin{proof}
Consider the union of the graphs
\beq{unions}
\bigcup_{t\in[1,\infty)}\Gamma_{\!ts}\ \subset\ E.
\eeq
This naturally defines a Borel-Moore chain in $E$ with boundary $\Gamma_{\!s}$ at $t=1$. We claim its closure, intersected with $\overline U$, defines a chain in $C_*(\overline U,\partial\overline U\take0_E)$ with additional boundary the fundamental class of $j_*\;C$ at $t=\infty$.

To prove this it is sufficient to work locally. Thus we may replace $(V,E)$ by $\(V_i:=V\cap\wt K_i^*,E_i\)$ \eqref{tildeEi}, $U=U_{E}^\epsilon(M)$ by $U_i:=U\cap E_i$ and $s$ by $s_i$ \eqref{lockuran}. These are all algebraic, and we recall from \eqref{close} that
\beq{close2}
s|_{V_i}\=s_i+O\(s_i^2\).
\eeq
It is standard that, in the Hilbert scheme of algebraic subschemes of $E_i$, there is a unique limit of the algebraic graphs $\Gamma_{\!ts_i}$ and it is the cone \eqref{coneJ} pushed forward by $j$ \eqref{j}. See for example \cite[Remark 5.1.1.]{Fu}. Therefore $j_*[C]$ is also the limit in the space of integral currents.\footnote{We thank Vivek Shende for suggesting we use integral currents. Since they carry fundamental homology classes they are a good replacement in the $C^\infty$ setting for the use of subschemes in the algebraic case.}
So we are left with checking that the estimate \eqref{close2} implies the graphs $\Gamma_{\!ts}$ and $\Gamma_{\!ts_i}$ have the same limits \emph{as integral currents} as $t\to\infty$.

We write $e_i:=s-s_i$. In any hermitian metric on $E_i$ we have an estimate
$$
|e_i|\ \le\ C_1|s_i|^2
$$
from \eqref{close2}. Thus on $\Gamma_{\!ts_i}\cap U_i\subset E_i$ we have
\begin{equation}\label{esti}
|ts_i|\ <\ \epsilon \quad\text{and so}\quad |te_i|\ \le\ \frac{C_1\epsilon^2}{|t|}\,.
\end{equation}
Fix any compactly supported $C^\infty$ test form $\sigma$ on $U_i$ and 
apply Stokes' theorem to the chain
$$
\Delta=\bigcup_{\lambda\in[0,t]}\Gamma_{\!ts_i+\lambda e_i}\ \cap\,U_i
$$
bounding the two graphs. Using \eqref{esti} to bound its thickness, and the fact that $\sigma$ is zero on $\partial\overline U_i$, we get the estimate
\begin{equation}\label{bownd}
\left|\int_{\Gamma_{\!ts_i}\cap\;U_i}\sigma-\int_{\Gamma_{\!ts\;}\cap\;U_i}\sigma\right|\ \le\ \left|\int_\Delta d\sigma\right|+\left|\int_{\overline\Delta\cap\;\partial\overline U_i}\sigma\right|\,\stackrel{\eqref{esti}}\le\,\frac{C_2\epsilon^2}{|t|}\,.
\end{equation}
As $t\to\infty$, \eqref{bownd} tends to zero, so
$$
\lim_{t\to\infty}\int_{\Gamma_{\!ts}\;\cap\;U_i}\sigma\ =\ 
\lim_{t\to\infty}\int_{\Gamma_{\!ts_i}\;\cap\;\overline U_i}\sigma
\ =\ \int_{C\;\cap\;\overline U_i}\sigma.
$$
Thus the integral currents $\Gamma_{\!ts}$ tend to $j_*\;C$ on $U_i$, as required. Finally we note that the graphs and $\Delta$ intersect $0_{E}$ (in the set-theoretic sense) only in $M$, so they all miss $\partial\overline U\cap 0_{E}$.
\end{proof}

\subsection*{Virtual cycles} For a general Joycian $\mu$-Kuranishi structure there is also a global Kuranishi chart \cite[Section 14 and Corollary 4.36]{Jvir}. In that setting the virtual cycle is defined topologically in the deformation retract $\overline U$ of $M$ by taking the homology class 
\beq{gammaU}
\big[\Gamma_{\!s}\cap\overline U\;\big]\ \in\ H_{\dim\_{\R\!}V}\big(\,\overline U,\,\partial\overline U\take0_{E}\big)
\eeq
and capping it with the Thom form
$$
\big[0_{E}\big]\ \in\ H^{\;\rk\_{\R\!}E}\big(\,\overline U,\,\partial\overline  U\take0_{E}\big).
$$
This defines the virtual cycle as their topological intersection
\beq{diffvc}
\Gamma_{\!s}\cdot0_{E}\ \in\ H_i\(\;\overline U,\Z\)\ \cong\ H_i(M,\Z), \qquad i\,=\,\dim\_{\R\!}V-\rk\_{\R\!}E.
\eeq
When $M$ is projective with perfect obstruction theory the graphs $\Gamma_{\!ts}$ limit to the cone $C$ as $t\to\infty$. This has the advantage of lying in $E_1=E|_M$ over $M$, instead of just being inside $E$ over $V$. Now $i=2v$ where $v$ is the virtual dimension $\dim\_{\C\!}V-\rk\_{\C\!}E$ and the 
Behrend-Fantechi virtual cycle is
\beq{BFvc}
0_{E_1}^{\;!}[C]\ \in\ A_v(M)\To H_{2v}(M,\Z).
\eeq

\begin{prop}\label{vc=vc}
The two virtual cycles \emph{(\ref{diffvc}, \ref{BFvc})} agree.
\end{prop}

\begin{proof}
Proposition \ref{BJ2} equates \eqref{gammaU} with $[j_*\;C\cap\overline U\;]\in H_{\dim V}\big(\,\overline U,\,\partial\overline  U\take0_{E}\big)$ by \eqref{propr}. Thus \eqref{diffvc} equals the cap product
\[
H_{2v}\(\;\overline U,\Z\)\ \ni\ j_*\;C\cdot0_{E}\=j_*(C\cdot j^*0_E)\=j_*(C\cdot 0_{E_1})\=j_*(0_{E_1}^{\;!}[C]). \qedhere
\]
\end{proof}

See \cite{LT2, Sie} for an analogous result in the more difficult case of $M$ being a Deligne-Mumford stack. In fact Siebert \cite[Section 3]{Sie} finds a way to take a limit of the graphs $\Gamma_{\!ts}$ to define a generalised cone even in non-holomorphic settings; this recovers the Behrend-Fantechi cone when everything is algebraic.

\section{Moduli of sheaves on Calabi-Yau fourfolds}\label{sheaf4}

Unfortunately Proposition \ref{vc=vc} is not relevant for the topic of the rest of this paper --- namely the virtual cycles of \cite{BJ, OT1} on moduli spaces $M$ of stable sheaves on Calabi-Yau 4-folds. These are defined by intersecting $\Gamma_s$ (or $C$) \emph{not} with the zero section of $E$ (or $E|_M=E_1$) but \emph{with a half dimensional subbundle} instead.

Before describing the theory, we give a brief overview of the plan in the case that the virtual dimension $\vd$ \eqref{vddef} is even.

We will see we can take $E_1$ to be an $SO(2n,\C)$ bundle inside which $C$ is isotropic. Moreover $E$ and $s$ are made from patching together $SO(2n,\C)$ bundles and isotropic sections. Then 
\begin{enumerate}
\item[(i)] the Borisov-Joyce virtual cycle \cite{BJ} equals the intersection of $\Gamma_{\!s}$ with a maximal negative definite real subbundle $iE^\R\subset E$, while
\item[(ii)] the algebraic virtual cycle \cite{OT1} is $(-1)^n$ times by the intersection of $C$ with a positive\footnote{In the sense of \cite[Definition 2.2]{OT1}.} maximal isotropic subbundle $\Lambda\subset E_1=E|_M$.
\end{enumerate}
We prove (i) in Theorem \ref{BJKur}; here the intersection lies naturally in the zero section $V$ of $E$. In (ii) we have to use cosection localisation \cite{KL} to force the intersection into the zero section $M$ of $E|_M$.

We will prove the two cycles are equal in two main steps,
\begin{enumerate}
\item homotoping $\Gamma_{\!s}$ to $C$ as in Proposition \ref{BJ2} (Theorem \ref{IC})  \emph{and}
\item homotoping $E^\R\subset E$ to $\Lambda\subset E$ (Theorem \ref{KiemLi}).
\end{enumerate}
\medskip

Let $(X,\cO_X(1))$ be a smooth polarised Calabi-Yau 4-fold: $K_X\cong\cO_X$. Let $M=M_c(X)$ be the moduli space of $\cO_X(1)$-Gieseker stable sheaves $F$ on $X$ of a fixed Chern character $c$, selected so that all semistable sheaves with Chern character $c$ are stable. Thus $M$ is projective.

We recall the deformation theory of $M$ from \cite{OT1}. Given a universal twisted sheaf $\cF$ on $\pi:X \times M\to M$, we let
\beq{EE}
\EE\ :=\ \tau^{[-2,0]}\(R\pi_* R\hom (\cF, \cF)[3]\).
\eeq
Relative Serre duality down the map $\pi$ gives an isomorphism
\beq{duel}
\theta\,\colon\,\EE \rt\sim \EE^\vee[2] \quad\text{such that}\quad \theta\=\theta^\vee[2]\ \in\ \Ext^2(\EE,\EE^\vee).
\eeq
Using the Atiyah class of $\cF$ as in \cite{HT} defines a morphism
\begin{align}  \label{ObsTh}
\phi\,\colon\,\EE \To \LL_{M}
\end{align}
which, by \cite[Theorem 4.1]{HT}, is an obstruction theory in the sense of \cite[Definition 4.4]{BF}.
In other words, $h^0(\phi)\colon h^0(\EE) \to \Omega_M$ is an isomorphism, and $h^{-1}(\phi) : h^{-1}(\EE) \to h^{-1}(\mathbb{L}_{M})$ is a surjection. Its virtual dimension is \eqref{vddef},
\beq{vd}
\vd\ :=\ \rk\EE\=2-\chi(c,c).
\eeq

However $\phi$ is\emph{ not a perfect obstruction theory} because $\EE$ is perfect of amplitude $[-2,0]$ rather $[-1,0]$. (Equivalently, in general $\Ext^3(F,F)=\Ext^1(F,F)^*$ is nonzero for $F\in M$, so $\EE$ can have $h^{-2}\ne0$.) So we cannot apply \cite{BF} to get a virtual cycle.

To describe the virtual cycles of \cite{BJ, OT1} we use a standard form for the obstruction theory \eqref{ObsTh}.

\begin{prop}\cite[Propositions 4.1 and 4.2]{OT1}\label{form} Fix an embedding $M\subset\PP^N$ with ideal $J\subset\cO_{\PP^N}$. On $M$ there is a bundle $T$, an $SO(r,\C)$ bundle $E_1$, and a surjective map of complexes
\beq{E*E}
\xymatrix@R=2pt{
T \ar[r]^{a^*}& E_1 \ar@{->>}[dd]\ar[r]^a& \ T^* \ar@{->>}[dd] && E\udot\!\ar[dd] \\ &&&= \\
& J/J^2 \ar[r]^-d& \Omega_{\PP^N}\big|_{M}\hspace{-5mm} && \LL_M\!}
\eeq
which represents the obstruction theory $\EE\to\LL_{M}$ \eqref{ObsTh}. 

Here the quadratic form $q$ on $E_1$ gives the isomorphism $q\colon E_1\cong E_1^*$ used in defining $a^*$. The quasi-isomorphism $\EE\cong E\udot$ intertwines the self-dualities
$$
(\id_T,\,q,\,\id_{T^*})\ \colon\, E\udot\rt\sim E_\bullet[2] \quad\text{and}\quad \theta\ \colon\,\EE\xrightarrow[\eqref{duel}]{\ \sim\ }\EE^\vee[2].\vspace{-9mm}
$$
\ $\hfill\square$
\end{prop}\vspace{1mm}

Consider the stupid truncation $\tau E\udot$ of the obstruction theory \eqref{E*E},
\beq{*E}
\xymatrix@R=15pt{
E_1 \ar@{->>}[d]\ar[r]^a& T^*\!\! \ar@{->>}[d] \\
J/J^2 \ar[r]^-d& \Omega_{\PP^N}\big|_{M}.\hspace{-6mm}}
\eeq
As in \cite[Section 4.2]{OT1} this defines a stupid perfect obstruction theory\footnote{This perfect obstruction theory ignores the degree $-2$ term $T$ in $E\udot$, so has the wrong virtual dimension and is not quasi-isomorphism invariant. The papers \cite{BJ, OT1} effectively replace $E_1$ by different half-dimensional subbundles when discarding $T$. \label{fn}} for $M$ of virtual dimension
$$
v\=\rk T-\rk E_1\=\vd-\rk T,
$$
cf. \eqref{v}. By \cite{BF} this defines a cone
\begin{equation}\label{coneev}
C\ \subset\ E_1
\end{equation}
over $M$ by pulling back the intrinsic normal cone from $h^1/h^0\((\tau E\udot)^\vee\)=[E_1/T]$. Considering $q$ as a function on the total space of $E_1$ (quadratic on each fibre), \cite[Proposition 4.3]{OT1} shows it restricts to 0 on $C$.

\begin{prop}\label{iso}
The cone $C\subset E_1$ is isotropic: $q|_C\equiv0.\hfill\square$
\end{prop}

Using the stupid perfect obstruction theory \eqref{*E}, Section \ref{global} now gives us
\beq{ckC}
\text{a global complex Kuranishi chart }\,(V,E,s)\,\text{ for }M.
\eeq
Recall the ambient space $V$ is a neighbourhood of $M$ inside the total space of $\wt K^*$ --- an $\cA_1$-bundle over $\mathring\PP^N$ extending the dual of $K:=\ker\(T^*\onto\Omega_{\PP^N}|_M\)$ over $M$. Its cotangent sheaf satisfies $\Omega_V|_M\cong T^*$. Similarly $E$ is an $\cA_1$-bundle over $V$ extending $E_1$ over $M$, and $s$ is an $\cA_1$-section of $E$ whose restriction to $2M\subset V$ has derivative the map $a$ in \eqref{*E}. The induced diagram \eqref{Kpot},
$$
\xymatrix@R=16pt{
E^* \ar[r]^(.42){ds}\ar[d]_s& \Omega_V\big|_M \ar@{=}[d]\\
I_2/I_1^2 \ar[r]& \Omega_V\big|_M\;.\!\!}
$$
is precisely \eqref{*E}. Therefore
\beq{E*E2}
T_V|_M\rt{(ds)^*}E^*|_M\rt{ds}\Omega_V|_M
\eeq
is the top row $E\udot$ of \eqref{E*E}; in particular it is a complex.\medskip

\subsection*{Real virtual cycle}
Since $SO(r,\C)$ is homotopic to $SO(r,\R)$ there is a real bundle $E_1^\R$ underlying $E_1$ in the sense that
\beq{splitr}
E_1\=E_1^\R\oplus iE_1^\R
\eeq
with $E_1^\R\subset E_1$ a maximal ($\rk_\R E_1^\R=r$) positive definite real subbundle for $q$. In particular $q$ is real and \emph{negative definite} on $iE_1^\R$.

By contrast $q$ \emph{vanishes} on the isotropic cone $C\subset E_1$, so $C$ and $iE_1^\R$ can only intersect in the 0-section $M$. Thus we may intersect $C$ and $iE_1^\R$ inside $E_1$ by taking the homology class
$$
[C]\ \in\ H_{2\rk T}\Big(\;\overline U^{\;\epsilon}_{\!E_1}(M),\ \partial\overline U^{\;\epsilon}_{\!E_1}(M)\take iE_1^\R\Big),
$$
and capping with the Thom form
\beq{thom}
\big[iE_1^\R\big]\ \in\ H^r\Big(\;\overline U^{\;\epsilon}_{\!E_1}(M),\ \partial\overline U^{\;\epsilon}_{\!E_1}(M)\take iE_1^\R\Big).
\eeq
Here the sign of the Thom form of $iE_1^\R$ is fixed by a choice of orientation on the fibres of its normal bundle (the pullback of) $E_1^\R$. But (real) orientations on $E_1^\R$ are equivalent to complex orientations on the $SO(r,\C)$ bundle $E_1$, as described carefully in \cite[Section 2.2]{OT1}. And by \cite[Proposition 4.2]{OT1} complex orientations on $E_1$ are equivalent to complex orientations on $\EE$ \eqref{EE}. A choice of the latter is guaranteed by \cite{CGJ} and fixed once and for all throughout the papers \cite{BJ,OT1} to define their virtual cycles.\footnote{The definitions of complex orientations on orthogonal vector spaces and bundles (\cite[Equation 2.7]{BJ}, \cite[Definition 2.1]{OT1}) and on self-dual complexes (\cite[Definition 2.12]{BJ} and \cite[Equation 59]{OT1} use different sign conventions but are equivalent.}

\begin{defn} By \cite{CGJ} we may fix a complex orientation on $\EE$. As above this orients $E_1^\R$ and hence defines a Thom class \eqref{thom}. We then define the real virtual class to be 
\begin{equation}\label{defint1}
C\cdot iE_1^\R\ \in\ H_{\vd}(M,\Z).
\end{equation}
\end{defn}

To relate \eqref{defint1} to the Borisov-Joyce class \cite{BJ} (from which it will follow it is independent of choices) we show how to replace $C$ by $\Gamma_{\!s}$. For this we extend the splitting \eqref{splitr} over the neighbourhood $V$ of $M\subset\wt K^*$,
\beq{realstr}
E\=E^\R\oplus iE^\R,
\eeq
and extend the holomorphic quadratic form $q$ on $E_1$ to a $C^\infty$ quadratic form $q$ on $E$.
Shrinking $V$ if necessary we may assume that Re\,$q$ is positive definite on $E^\R$. Together the quadratic form and the real structure \eqref{realstr} define a hermitian metric $|\,\cdot\,|$ on $E$.

Pushing $C$ forward from $E_1=E|_M$ to $E$ via $j$ \eqref{j} and setting $U:=U^{\;\epsilon}_{\!E}(M)$ as before we can rewrite \eqref{defint1} by taking the homology class
$$
[j_*\;C]\ \in\ H_{\dim\_\R\!V}\(\;\overline U,\,\partial\overline U\take iE^\R\),
$$
and capping with the Thom form
\beq{2thom}
\big[iE^\R\big]\ \in\ H^{\rk E}\big(\overline U,\,\partial\overline U\take iE^\R\big)
\eeq
to give
\begin{equation}\label{defint15}
j_*\;C\cdot iE^\R\ \in\ H_{\vd}\(\;\overline U,\Z\)\ \cong\ H_{\vd}(M,\Z).
\end{equation}
But $j^*[iE^\R]=[iE_1^\R]$ so \eqref{defint1} and \eqref{defint15} coincide,
\beq{coincide}
C\cdot iE_1^\R\=j_*\;C\cdot iE^\R.
\eeq

Next we claim that for $0<\epsilon\ll1$ sufficiently small, $\Gamma_{\!s}$ also avoids\footnote{In fact the set-theoretic intersection of $\Gamma_{\!s}$ and $iE^\R$ is precisely $M$. With some work this follows from analysing the gluing in Appendix \ref{proof} applied to the local Darboux models $(E_i,s_i)$ of \cite{BBJ} and \cite[Equation 68]{OT1} consisting of isotropic sections $s_i$ of orthogonal bundles $E_i$ over the open cover $V_i$ of $V$.} $\partial\overline U\cap iE^\R$.
Since $\Gamma_{\!ts}$ limits to the isotropic cone $C$ as $t\to\infty$ we can replace $s$ by $ts$ for $t$ sufficiently large (or equivalently, by the $\C^*$-invariance of $C$, take $\epsilon$ sufficiently small) to ensure that on $\Gamma_{\!s}\cap\partial\overline U$,
\begin{itemize}
\item $q$ is as small as we like, and
\item $|s|^2$ is as close as we like to $\epsilon$.
\end{itemize}
So we can easily arrange for the following to hold,
$$
|q(s,s)|\ <\ |s|^2 \,\text{ on }\, \Gamma_{\!s}\cap\partial\overline U.
$$
Since the estimate only improves as we scale $s$ by $t>1$ we note for later the stronger inequality
\beq{t<<}
|q(e,e)|\ <\ |e|^2 \,\text{ for }\,e\,\in\,\Gamma_{\!ts}\cap\partial\overline U,\ \ \forall t\,\ge\,1.
\eeq 
Conversely, on $iE^\R$, by the definition of the Hermitian metric $|\,\cdot\,|$,
\beq{>>}
q(ie,ie)\=-|e|^2 \,\text{ for }\, ie\,\in \,iE^\R\cap\partial\overline U.
\eeq
Therefore $\Gamma_{\!s}$ and $iE^\R$ do not meet in $\partial\overline U$. So we may cap
$$
[\Gamma_{\!s}]\ \in\ H_{\dim_\R\!V}\(\;\overline U,\,\partial\overline U\take iE^\R\)
$$
with the Thom form  $\big[iE^\R\big]$ \eqref{2thom} to define
\beq{defint2}
\Gamma_{\!s}\cdot iE^\R\ \in\ H_{\vd}\(\;\overline U,\Z\)\ \cong\ H_{\vd}(M,\Z).
\eeq
Then the analogue of Proposition \ref{BJ2} for this intersection is the following.

%

\begin{thm}\label{IC} Under the above conditions \eqref{defint1} and \eqref{defint2} coincide,
\beq{train}
\Gamma_{\!s}\cdot iE^\R\=C\cdot iE_1^\R.
\eeq
\end{thm}

\begin{proof}
The proof of Proposition \ref{BJ2} showed the chain \eqref{unions},
$$
\Gamma\ :=\ \bigcup\nolimits_{t\in[1,\infty)}\Gamma_{\!ts},
$$
bounds $[\Gamma_{\!s}]-j_*[C]$. Comparing \eqref{t<<} with \eqref{>>} shows $\Gamma$ is disjoint from $iE^\R\cap\partial\overline U$. Thus we get the following strengthening of \eqref{propr},
$$
[\Gamma_s]\=j_*[C]\ \in\ H_{\dim V}\big(\overline U,\,\partial\overline U\take iE^\R\big).
$$
Capping with \eqref{2thom} now gives $\Gamma_{\!s}\cdot iE^\R=j_*\;C\cdot iE^\R$, which by \eqref{coincide} gives $C\cdot iE_1^\R$ as required.
\end{proof}

By the projection formula applied to the homology class of $\Gamma_{\!s}$ and the cohomology class of $0_{E^\R}$, the left hand side of \eqref{train} can be expressed as an intersection in $E^\R$.
Letting $s^+$ denote the projection of $s$ to $E^\R$ it gives
\beq{vir=vir}
\Gamma_{\!s^+}\cdot0_{E^\R}\=C\cdot iE^\R_1.
\eeq
The left hand side is the virtual cycle associated to the global $\mu$-Kuranishi chart
\beq{glKur}
\(V,E_1^\R,s^+\).
\eeq
The next Section is devoted to proving that this is a $\mu$-Kuranishi chart for $M$ compatible with (one of) Borisov-Joyce's $\mu$-Kuranishi structure on $M$. While their $\mu$-Kuranishi structures depend on choices, the cobordism (or homology) class of the resulting virtual cycle $[M]_{BJ}^{\vir}$ is uniquely defined by \cite[Corollary 3.19]{BJ}.

\section{The Borisov-Joyce virtual cycle}\label{BJvc}
In this Section we prove the following result.

\begin{thm}\label{BJKur} The global complex Kuranishi chart \eqref{glKur} induces one of the Borisov-Joyce $\mu$-Kuranishi structures of \cite[Theorem 3.15]{BJ} on $M$. Thus, by \eqref{vir=vir},
$$
[M]^{\vir}_{BJ}\=\Gamma_{\!s^+}\cdot0_{E^\R}\=C\cdot iE^\R_1.
$$
\end{thm}

\subsection{Review of Borisov-Joyce's construction}\label{construct} We begin by describing Borisov-Joyce's $\mu$-Kuranishi structure on $M$. They rely heavily on the canonical $(-2)$-shifted symplectic derived scheme structure on $M$. The derived part of this structure is defined in \cite{TVa} as a refinement\footnote{Strictly speaking \cite{TVa} describes a derived moduli \emph{stack} of sheaves on $X$. We rigidify by removing the $\C^*$ stabilisers from the substack of stable sheaves to recover $M$. At the same time we remove $h^{1}\cong\mathrm{Lie}(\C^*)^*\otimes\cO$ from the derived cotangent bundle, and the dual copy of $\cO$ from $h^{-3}$, to get the $(-2)$-shifted symplectic derived scheme structure on $M$.} of the obstruction theory \eqref{ObsTh}. The $(-2)$-shifted symplectic structure is defined in \cite{PTVV} as a refinement of the Serre duality quadratic structure \eqref{duel}.

By \cite{BG,BBJ,BBBJ} this structure is equivalent to $M$ having
\beq{chts}
\text{a Zariski open cover $\{M_i\}$  by local \emph{Darboux charts} $(D_i,F_i,s_i)$}
\eeq
with a certain compatibility between them which we will describe in \eqref{j0} below. Here
\begin{itemize}
\item $D_i$ is a smooth scheme,
\item $F_i$ is a holomorphic \emph{orthogonal} bundle over it,
\item $s_i$ a holomorphic \emph{isotropic} section of $F_i$,
\end{itemize}
and we fix an isomorphism (which we are supressing in the notation) between the zero scheme of $s_i$ and $M_i$.

It is important that the Darboux charts \eqref{chts} should not really be thought of as Kuranishi charts --- this would \emph{not} give the Borisov-Joyce Kuranishi structure\footnote{Analogously in \cite{OT1} we do not take the Behrend-Fantechi virtual cycle \eqref{BFvc} that comes from considering $(V,E,s)$ as a global Kuranishi chart; instead we take \eqref{vird}.} on $M$. In fact thinking of them as Kuranishi charts would correspond to using the Koszul cdga
\beq{kosdga}
(\Lambda\udot F_i^*,s_i)\=\cdots\To\Lambda^2F_i^*\rt{s_i}F_i^*\rt{s_i}\cO_{D_i}
\eeq
to describe the derived structure on $M_i$, with virtual cotangent bundle
$$
F_i^*|_{M_i}\rt{ds_i} T^*_{D_i}|_{M_i}.
$$
But this is the stupid truncation of the actual virtual cotangent bundle
$$
T_{D_i}|_{M_i}\rt{ds_i}F_i^*|_{M_i}\rt{ds_i} T^*_{D_i}|_{M_i},
$$
which is the virtual cotangent bundle of the full cdga
\begin{multline}\label{-2dga}
\Lambda\udot F_i^*\otimes\Sym\udot\!\(T_{D_i}[2]\)\=\\
\xymatrix@C=14pt{
\cdots \ar[r]& \Lambda^3F_i^*\oplus(T_{D_i}\!\otimes\!\;F_i^*) \ar[rrr]|(.55){\;\tiny\Big(\!\!\!\begin{array}{cccc}s_i\!\!&\!\!\nabla s_i\\0\!\!&\!\!s_i\end{array}\!\!\!\!\Big)\;}
&&& \Lambda^2F_i^*\oplus T_{D_i} \ar[rr]^(.6){(s_i,\,\nabla s_i)}&& F_i^* \ar[r]^-{s_i}& \cO_{D_i}.\hspace{-4mm}}
\end{multline}
Here $\nabla$ is any holomorphic connection on $F_i$ with respect to which its quadratic form is parallel.\footnote{Beware we omit the differentials from the shorthand notation $\Lambda\udot F_i^*\otimes\Sym\udot\!\(T_{D_i}[2]\)$.} This is the correct cdga to assign to the Darboux chart \eqref{chts}, and any cdga \eqref{-2dga} made from a Darboux chart in this way carries a canonical $(-2)$-shifted symplectic derived structure. This describes the $(-2)$-shifted symplectic structure on $M_i$ which comes from restricting the $(-2)$-shifted symplectic structure of \cite{PTVV} on $M$. Stupidly truncating by removing the $T_{D_i}$ terms recovers the Koszul cdga \eqref{kosdga}.

%

The compatibility between (or gluing of) the Darboux charts \eqref{chts} is given by homotopy equivalences of the cdgas \eqref{-2dga}. Given a cover by Darboux charts, the $(-2)$-shifted symplectic derived structure on $M_i$ is equivalent to the one \eqref{-2dga} given by its Darboux chart \eqref{chts}; composing this equivalence with the inverse of the corresponding equivalence from the Darboux chart $(D_j,F_j,s_j)$ defines the transition map of cdgas
\beq{j0}
\Lambda\udot F_i^*\otimes\Sym\udot\!\(T_{D_i}[2]\)\Big|_{D_{ij}} \,\rt{a_{ij}}\ \Lambda\udot F_j^*\otimes\Sym\udot\!\(T_{D_j}[2]\)\Big|_{D_{ij}}.
\eeq
In general this will \emph{not} give an equivalence or gluing of the truncated cdgas \eqref{kosdga}, just some kind of gluing-modulo-lower-order-terms, as described carefully in \cite[Corollary 3.5]{BJ}.

Borisov-Joyce then ``halve" the derived structure on $M$ by
\begin{enumerate}
\item discarding the terms in degrees $\le-2$ --- i.e. truncating \eqref{-2dga} to \eqref{kosdga}, or considering the Darboux chart \eqref{chts} as a complex Kuranishi chart,
\item replacing each $F_i\cong F_i^\R\oplus iF_i^\R$ by a maximal real positive definite subbundle $F_i^\R$, and
\item projecting the $s_i\in\Gamma(F_i)$ to $s_i^+\in\Gamma(F_i^\R)$.
\end{enumerate}
In \cite[Corollary 3.11 and Theorem 3.15(vi)]{BJ} they show the $F_i^\R$ may be chosen so that the gluing-modulo-lower-degree-terms of \cite[Corollary 3.5]{BJ} projects to induce a genuine gluing of the resulting new $\mu$-Kuranishi charts $(D_i,F^\R_i,s^+_i)$. (This is a gluing in the sense of Definition \ref{cK}, but using $C^\infty$ maps and functions.) The resulting $\mu$-Kuranishi structure depends on various choices (such as the initial choice of local Darboux charts) but the induced virtual cycle is independent up to cobordism. \medskip
 
Our approach will be to first show that our local holomorphic Kuranishi charts $(V_i,E_i,s_i)$ of \eqref{lockuran} can be upgraded to Darboux charts as in \eqref{chts}. Yet we will also show that by their very construction they can be made to glue (not just up to lower order terms) as in Definition \ref{cK} to give the global non-holomorphic chart $(V,E,s)$ constructed in Section \ref{gKc}. Thus in this case the Borisov-Joyce gluing procedure described above will give what one might expect: step (1) will recover the global chart $(V,E,s)$, then steps (2) and (3) project it to the global chart $(V,E^\R,s^+)$ of \eqref{glKur}.\medskip

\subsection*{Our Darboux charts}
First we will show that the local holomorphic Kuranishi charts $(V_i,E_i,s_i)$ of \eqref{lockuran} can be made into Darboux charts. 

%
%

\begin{prop}\label{03}
Refining the cover $M=\cup_iM_i$ if necessary, we may choose an orthogonal structure on $E_i$ so that $s_i$ is an isotropic section and $(V_i,E_i,s_i)$ is a Darboux chart for $M_i\subset M$.
\end{prop}

\begin{proof}
Fix a point $x\in M$ and $i$ such that $x\in M_i$. We will allow ourselves to shrink the neighbourhood $M_i$ of $x$ throughout the proof. Repeating the proof for any $x\in M$ and using compactness will give the required refinement of $\cup_i M_i$.

Let $F$ be the sheaf on $X$ corresponding to $x\in M$, and let $q$ be the quadratic form on $\Ext^2(F,F)$ given by Serre duality. 
By \cite[Example 5.16]{BBJ} there is an open neighbourhood $A$ of $0\in\Ext^1(F,F)$, and a map
\beq{ext2}
s\,\colon A\To\Ext^2(F,F) \,\text{ with }\, ds|\_0\,=\,0,\ q(s,s)\,=\,0 \,\text{ and }\,Z(s)\,\cong\,M_i,
\eeq
such that the obstruction theory on $M_i$ is quasi-isomorphic to
$$
\xymatrix@R=2pt{\hspace{8mm}
Q\udot \ \ :=\ \  \big\{T_A|_{M_i} \ar[r]^-{ds}& Q \ar[dd]^-s\ar[r]^-{(ds)^*}& \Omega_A|_{M_i}\big\}\!\! \ar@{=}[dd]<-.3ex> && \EE|_{M_i} \ar[dd] \\ &&& \cong \\
& I/I^2 \ar[r]^d& \Omega_A|_{M_i} && \LL_{\;M_i}.\!\!}
$$
Here $I$ is the ideal sheaf of $M_i\subset A$ and $Q$ is the trivial orthogonal bundle with fibre $\Ext^2(F,F)$. Its quadratic form gives an isomorphism $Q\cong Q^*$ which defines $(ds)^*$ and the self-duality $Q\udot\cong Q_\bullet[2]$ which is intertwined with the Serre duality of \eqref{duel} by the quasi-isomorphism $Q\udot\cong\EE|_{M_i}$.

Furthermore, in \cite[Equation 70]{OT1}
we showed that the self-dual complex $E\udot=\big\{T_V|_M\to E_1\to\Omega_V|_M\big\}$ of \eqref{E*E2} can be written --- on restriction to $M_i$ --- as the orthogonal direct sum of $Q\udot$ and a self-dual acyclic complex $K\udot=\{K^0\to K^1\to K_0\}$
$$
E\udot\ \cong\ Q\udot\oplus K\udot.
$$
Thus the self-duality $E\udot\cong E_\bullet[2]$ is the direct sum of the self-dualities $Q\udot\cong Q_\bullet[2]$ and $K\udot\cong K_\bullet[2]$.

Choosing an isotropic splitting of $K^1$ as in \cite[Footnote 7]{OT1} we may write $K\udot$ as $\{K^0\to K^0\oplus K_0\to K_0\}$ with the maps and orthogonal structure the obvious ones. Shrinking $M_i$ if necessary we may assume all our bundles are trivial, so that
$$
K\udot\ \cong\ \big\{\cO_{M_i}^{\oplus k}\rt{(1,0)}\cO_{M_i}^{\oplus k}\oplus\cO_{M_i}^{\oplus k}\rt{(0,1)}\cO_{M_i}^{\oplus k}\big\}
$$
is the (pullback to $M_i$ of the) virtual cotangent bundle associated to the following Darboux chart for the origin in $\C^k$,\vspace{-1cm}
\beq{affmod2}
\xymatrix@=20pt{
& \(\cO^{\oplus k}\oplus\cO^{\oplus k},q\)\ar[d]  \\
\{0\}\=s^{-1}(0)\ \subset\hspace{-15mm} & \C^k,\ar@/^{-2ex}/[u]_t}
\qquad \begin{array}{c} \\\\\\ q(\underline x,\underline y)\=\sum x_iy_i, \\ t(\underline z)\=(0,\underline z), \\ q(t,t)=0.\end{array}
\eeq
Thus $E\udot|_{M_i}\to\LL_{M_i}$ is isomorphic to the virtual cotangent bundle associated to the product Darboux chart \eqref{ext2}$\,\times\,$\eqref{affmod2}. In particular the perfect obstruction theory given by its stupid truncation $\tau E\udot|_{M_i}:=\big\{E_1\to\Omega_V|_M\big\}\to\LL_M$ is isomorphic to the perfect obstruction theory induced by the product \emph{Kuranishi} (rather than Darboux) chart \eqref{ext2}$\,\times\,$\eqref{affmod2}.

Thus Theorem \ref{affKurpot} gives an isomorphism between the Kuranishi charts $(V_i,E_i,s_i)$ and \eqref{ext2}$\,\times\,$\eqref{affmod2}, allowing us to transfer the Darboux structure over from the latter to the former.
\end{proof}

\subsection{A homotopy}
So now the $(-2)$-shifted symplectic derived structure on $M$ defines cdga gluing maps $a_{ij}$ between the cdgas \eqref{j0} induced by these local Darboux charts. All that we really know about $a_{ij}$ is that its induced action from the virtual cotangent bundle of $M_i$ to that of $M_j$ is the gluing map which gives the virtual cotangent bundle of $M$. By \eqref{E*E2} this is
\beq{lpot}
T_V|_M\rt{ds}E\rt{ds}T^*_V|_M,
\eeq
where $(V,E,s)$ is the global complex Kuranishi chart \eqref{ckC}.

Therefore the gluing equivalence induced by $a_{ij}$ between the horizontal 3-term complexes below is equal, in $D^b(\mathrm{Coh}\,M_{ij})$, to those induced by $(\psi_{ij},\Phi_{ij})$ and $(\psi_{ij}',\Phi_{ij}')$ \eqref{holo'}.
Thus the corresponding maps of complexes differ only by a homotopy $h=(h_1,h_2)$,
\beq{dad}
\xymatrix{
T_{V_i}|_{M_{ij}} \ar[r]^{ds_i}\ar[d]& E_i^*|_{M_{ij}} \ar@{..>}[dl]_(.55){h_2\!}\ar[d]\ar[r]^{ds_i}& \Omega_{V_i}|_{M_{ij}} \ar[d]\ar@{..>}[dl]_(0.55){h_1\!} \\
T_{V_j}|_{M_{ij}} \ar[r]^{ds_j}& E_j^*|_{M_{ij}} \ar[r]^{ds_j}& \Omega_{V_j}|_{M_{ij}}.\!
}\eeq
The $h_2$ term is a problem for us because we want to take stupid truncations (removing the left hand $T_V$ terms). The induced (truncated) maps between the resulting 2-term complexes need not be homotopic unless $h_2=0$. So we next show how to arrange this by lifting homotopies from virtual cotangent bundles on $M$ to the level of cdgas on $V$.

\begin{prop} By altering the equivalence $a_{ij}$ of \eqref{j0} by a homotopy if necessary, we may assume that $h_2=0$.
\end{prop}

\begin{proof}
We begin by solving this problem in a special case, where the equivalence is the identity map on the single cdga $\Lambda\udot E^*\otimes\Sym\udot(T_V[2])$ of \eqref{-2dga} --- with differential $\delta$ made from a section $s\in H^0(V,E)$ and $\nabla s$ as in \eqref{-2dga} --- and the homotopy $(h_1,h_2)$ is $(0,h)$ for some $h\in\Hom(E^*|_M, T_V|_M)$.

Assume $V$ is Stein and set $M=s^{-1}(0)\subset V$. We show how to modify the identity cdga map by a homotopy to alter the induced map \eqref{dad} on the virtual cotangent bundle (which is also the identity of course) by the homotopy $(0,h)$. Lift $h$ to some
$$
H\ \in\ \Hom(E^*, T_V).
$$
Using $\delta$ also for the differential in the virtual cotangent bundle \eqref{lpot}, we must define an automorphism of $\Lambda\udot E^*\otimes\Sym\udot(T_V[2])$, homotopic to the identity, such that the induced automorphism of its virtual cotangent bundle \eqref{lpot} is $\id+\,[\delta,h]$ on $M$. We use the identity $\id$ on $\cO_V$,
\beq{deg1}
f_t\ :=\ \id+\,t\;[\delta,H]\= \id+\,t\;\nabla s\circ H\ \colon\,E^*\To E^*
\eeq
on $E^*$, and 
\beq{deg2}
g_t\ :=\ \id+\,t\;[\delta,H]\= \id+\,t\;H\circ\nabla s\ \colon\,T_V\To T_V
\eeq
on $T_V$, all for fixed $t\in[0,1]$. On the free graded algebra $\Lambda\udot E^*\otimes\Sym\udot(T_V[2])$ on these generators, these define the graded algebra endomorphisms
$$
F_t\ :=\,\Lambda\udot f_t\otimes\Sym\udot(g_t[2])\,\colon\,\Lambda\udot E^*\otimes\Sym\udot(T_V[2])\To\Lambda\udot E^*\otimes\Sym\udot(T_V[2]).
$$
Since $[\delta,F_t]$ vanishes on the generators $\cO_V,\,E^*,\,T_V$, the Leibniz rule gives $[\delta,F_t]=0$. Thus the $F_t$ are actually cdga endomorphisms.

We claim they in fact fit into a \emph{homotopy} of cdga automorphisms, i.e. a map of cdgas
$$
F_t+dt\cdot H_t\ \colon\,\Lambda\udot E^*\otimes\Sym\udot(T_V[2])\To\Lambda\udot E^*\otimes\Sym\udot(T_V[2])\;[t,dt],
$$
where $t,\,dt$ have degrees 0 and 1 respectively, the differential on the right hand side is $\delta\otimes1+1\otimes dt\cdot\frac d{dt}(\ \cdot\ )$, and $\overset{\,_{\mbox{\Huge .}}}{F_t}=[\delta,H_t]$. 

To prove the claim we define $H_t$ to be $H\colon E^*\to T_V$ on $E^*$ and zero on $\cO_V$, $T_V$. We then extend it inductively to higher degrees by the formula
\beq{Hdef}
H_t(ab)\=H_t(a)F_t(b)+(-1)^{|a|}F_t(a)H_t(b)
\eeq
for homogeneous elements $a,b$ of degrees $|a|,|b|$. 
To prove $\overset{\,_{\mbox{\Huge .}}}{F_t}$ equals $[\delta,H_t]$ we check they agree on homogeneous elements of the cdga by induction on degree. They agree on the generators $\cO_V,\,E^*,\,T_V$ by (\ref{deg1}, \ref{deg2}). If they agree on $a$ and $b$ then the Leibniz rule and \eqref{Hdef} show they both give $[\delta,H_t](a)F_t(b)+F_t(a)[\delta,H_t](b)$ when acting on $ab$.

Thus we get a homotopy from $\id=F_0$ to $F_1$. By (\ref{deg1}, \ref{deg2}) $F_1$ is $\id+\,[\delta,H]$ on $E^*$ and $T_V$ and $F_1=\id$ on $\cO_V$. So it induces the map $\id+\,[\delta,h]$ on the virtual cotangent bundle on $M$, as required.\medskip

We apply this to $V=V_{ij}$ and $M=M_{ij}$ to prove the Proposition. Lift $h_2$ to $H_2\in\Hom\!\(E_i^*|_{V_{ij}},T_{V_j}|_{V_{ij}}\)$ and set
$$
H\ :=\ D\(\psi_{ij}'\)^{-1}\circ H_2\ \in\ \Hom\!\(E_i^*|_{V_{ij}},T_{V_i}|_{V_{ij}}\)
$$
so that
\beq{homcomp}
(\psi_{ij}',\Phi_{ij}')+[\delta,H_2]\=(\psi_{ij}',\Phi_{ij}')\circ\Big(\!\id+[\delta,H]\Big).
\eeq
The above construction gives a family of homotopic automorphisms $F_t$ of $\Lambda\udot E_i^*\otimes\Sym\udot\!\(T_{V_i}[2]\)\big|_{V_{ij}}$ with $F_0=\id$ and $F_1=\id+\,[\delta,H]$.
So replacing $(\psi_{ij}',\Phi_{ij}')$ by the homotopic
$$
(\psi_{ij}',\Phi_{ij}')\circ F_1\,\stackrel{\eqref{homcomp}}=\,(\psi_{ij}',\Phi_{ij}')+[\delta,H_2]
$$
we kill the $h_2$ term in the homotopy \eqref{dad} between the induced maps on virtual cotangent bundles.
\end{proof}

Therefore we may assume $a_{ij}$ induces a map of truncated complexes
$$
\xymatrix{
E_i^*|_{M_{ij}} \ar[d]\ar[r]^{ds_i}& \Omega_{V_i}|_{M_{ij}} \ar[d]\ar@{..>}[dl]_(0.55){h_1\!} \\
E_j^*|_{M_{ij}} \ar[r]^{ds_j}& \Omega_{V_j}|_{M_{ij}},\!
}
$$
just as $(\psi_{ij}',\Phi_{ij}')$ does, and the two are the same up to the homotopy $h_1$.

Thus the truncation of $a_{ij}$, mapping the truncation $(\Lambda\udot E_i^*,s_i)$ of \eqref{j0} to $(\Lambda\udot E_j^*,s_j)$ --- or equivalently mapping between\vspace{-.5mm} local holomorphic Kuranishi charts $(V_i,E_i,s_i)$ and $(V_j,E_j,s_j)$ over $V_{ij}$ --- is homotopic to $(\psi_{ij}',\Phi_{ij}')$ in the sense of Definition \ref{Homotopy}. So the stupid truncation of the description \eqref{j0} of the derived structure on $M$ defines the complex Kuranishi structure on $M$ given by the gluing $((\psi_{ij}',\Phi_{ij}'))$. By \eqref{holo'} this is the Kuranishi structure given by our global complex Kuranishi chart $(V,E,s)$ \eqref{ckC}.

$(V,E,s)$ is therefore the output of step (1) of the Borisov-Joyce construction described in Section \ref{construct}.
For step (2) we choose an extension $E=E^\R\oplus iE^\R$ over $V$ of the splitting $E_1=E_1^\R\oplus iE_1^\R$ on $M$. The latter satisfies condition $(*)$ of \cite[Section 3.3]{BJ} by \cite[Example 3.8]{BJ}, so the former does too (at least after shrinking $V$) because $(*)$ is an open condition \cite[Theorem 3.7(a)]{BJ}. Since the restrictions $E_i^\R\subset E_i$ of $E^\R\subset E$ to $V_i$ glue via our $\cA_1$ transition functions $\Phi_{ij}$ over $V_{ij}$, they satisfy the compatibility of \cite[Section 3.4]{BJ} over $V_{ij}$.
So for step (3) we let $s^+$ be the projection of $s$ under $E\onto E^\R$ to get $(V,E^\R,s^+)$ as a global $\mu$-Kuranishi chart for the Borisov-Joyce $\mu$-Kuranishi structure on $M$.

\section{Odd virtual dimension}\label{odd}
We are now ready to prove that $[M]^{\vir}_{BJ}$ is 2-torsion when $\vd$ is odd. We begin on any scheme $Y$ with an $SO(2n+1,\C)$ bundle  $(E,q,o)$ and an isotropic section $s$ with compact zero locus.\footnote{This will be applied later to the isotropic cone $C$ \eqref{coneev} and the pullback of the special orthogonal bundle $E_1$ with its tautological section $\tau\_{E_1}|\_C$.} Here $o\in\Gamma(\Lambda^{2n+1}E)$ is an \emph{orientation} on $(E,q)$ in the sense of \cite[Definition 2.1]{OT1}.

\subsection*{Maximal positive definite real subbundles} As in \cite[Section 2.2]{OT1}, $(E,q)$ admits a maximal positive definite real subbundle $E^\R\subset E$, unique up to homotopy. That is, $E^\R$ has real rank $2n+1$ and $q^\R:=q|_{E^\R}$ is positive definite (i.e. a real metric). Then $iE^\R\subset E$ is complementary to $E^\R\subset E$, defining a splitting
\beq{EE11}
E\ \cong\ E^\R\oplus iE^\R\=E^\R\otimes\_\R\C
\eeq
with respect to which $q=q^\R\otimes\C$. Furthermore \eqref{EE11} defines a complex conjugation operation on $E$ such that $q(\overline e_1,\overline e_2)=\overline{q(e_1,e_2)}$.


\subsection*{Hermitian metric and splittings} Thus $E$ inherits the hermitian metric
$$
h(\ \cdot\ ,\ \cdot\ )\ :=\ q\(\ \cdot\ ,\overline{\ \cdot\ }\).
$$
Then the complex conjugate $\overline\Lambda\subset E$ of any (not necessarily maximal) isotropic subbundle $\Lambda\subset E$ is also isotropic (but not holomorphic in general). Under the pairing $q$ it is isomorphic to $\Lambda^*$, giving a non-holomorphic embedding $\Lambda\oplus\Lambda^*\subset E$. Taking the orthogonal with respect to either $q$ or $h$ --- they give the same result --- gives a $\C$- (but not $\cO$-) linear splitting
\beq{Clin}
E\ \cong\ \Lambda\oplus\Lambda^*\oplus E_\perp,
\eeq
where $E_\perp=\Lambda^\perp/\Lambda$. With respect to this decomposition, $q$ is the direct sum of a nondegenerate quadratic form on $E_\perp$ and the canonical pairing between $\Lambda$ and $\Lambda^*$. Writing \eqref{Clin} as $E=E_\Lambda\oplus E_\perp$ we get a corresponding $\R$-linear splitting of real parts by taking the fixed loci of complex conjugation:
\beq{Rlin}
E^\R\ \cong\ E_\Lambda^\R\oplus E_\perp^\R.
\eeq


\subsection*{Orientation}
In a local chart let $\{e_i\}_{i=1}^m$ be a local basis of sections of $\Lambda$ with dual basis $\{f_i\}$ of $\Lambda^*$.
The orientation $o$ on $E$, together with the canonical orientation $(-i)^m e_1\wedge f_1\wedge\dots \wedge e_m\wedge f_m$ of \cite[Equation 18]{OT1} on $\Lambda\oplus\Lambda^*$, defines a canonical orientation $o_\perp$ on $E_\perp=\Lambda^\perp/\Lambda$ such that
\beq{orperp}
o\=(-i)^m e_1\wedge f_1\wedge\dots \wedge e_m\wedge f_m \wedge o\_\perp.
\eeq
Thus $E_\perp$ is an $SO(2n-2m+1,\C)$ bundle.


\subsection*{Vanishing intersection: divisorial case} 

Given an isotropic section $s$ of $E$ with \emph{compact} zero locus $Z(s)\subset Y$, let $U$ be an $\epsilon$-neighbourhood of $Z(s)$ inside the total space of $E\to Y$. We pair the homology class
$$
\big[\Gamma_{\!s}\big]\ \in\ H_{2\dim Y}\(\;\overline U,\partial\overline U\take iE^\R\)
$$
with the Thom class
$$
[iE^\R]\ \in\ H^{2n+1}\(\;\overline U,\partial\overline U\take iE^\R\)
$$
defined using the real orientation $o^\R$ on its normal bundle (the pullback of $E^\R$) induced from the complex orientation $o$ on $E$ as in \cite[Section 2.2]{OT1}. This defines
\beq{odd0}
\Gamma_{\!s}\cdot iE^\R\ \in\ H_{2\dim Y-2n-1}\(\;\overline U\)\ \cong\ H_{2\dim Y-2n-1}(Z(s)).
\eeq
We will show this is 2-torsion, first in the case that $Z(s)\subset Y$ is a Cartier divisor. This means the normal cone $C_{Z(s)/Y}\subset E_1$ is an isotropic \emph{line subbundle} $\Lambda\subset E$ homologous to $\Gamma_{\!s}$ in $(\;\overline U,\partial\overline U\take iE^\R)$ as in Theorem \ref{IC}, so
$$
\Gamma_{\!s}\cdot iE^\R\=\Lambda\cdot iE_1^\R.
$$
By \eqref{Rlin} and the projection formula this equals
$$
e(E^\R_\perp)\cap\big(\Lambda\cdot iE_\Lambda^\R\big).
$$
Since $E^\R_\perp$ is an $SO(2n-1,\R)$ bundle its Euler class is 2-torsion on the compact cycle $\Lambda\cdot iE_\Lambda^\R$, giving the result claimed.

\subsection*{General case}
To prove \eqref{odd0} is 2-torsion we separate the irreducible components of $Y$ into two types, following \cite[Theorem 5.2]{KP}. On those for which $s$ vanishes on the underlying reduced variety,
\beq{2one}
\Gamma_{\!s}\cdot iE^\R\=e(E^\R)\ \text{ is 2-torsion}
\eeq
because $iE^\R$ is an $SO(2n+1,\R)$ bundle on a compact space.

On the other components we blow up in the zero locus of $s$ to give $p\colon\wt Y\to Y$. Here the zero locus of $p^*s$ --- the exceptional divisor --- is Cartier, so by the last Section $\Gamma_{p^*s}\cdot p^*(iE^\R)$ is 2-torsion. Thus
\beq{2two}
p_*\big(\Gamma_{p^*s}\cdot p^*(iE^\R)\big)\ \text{ is 2-torsion}.
\eeq
Hence, by the projection formula, $\Gamma_{\!s}\cdot iE^\R$ is a sum of 2-torsion contributions (\ref{2one}, \ref{2two}) over all irreducible components of $Y$.


\subsection*{Odd dimensional Borisov-Joyce}
To apply this to Calabi-Yau 4-folds, we take $Y$ to be the total space of the isotropic cone $C\subset E_1$ \eqref{coneev} and the bundle $E$ to be the pullback of $E_1$ with its tautological isotropic section $s=\tau\_{E_1}\big|_C\,$. This has compact zero locus $M$.

Suppose now that $\vd$ \eqref{vd} is odd. In the notation of \eqref{E*E}, $\vd=2\rk T-\rk E_1$, so we find $2n+1:=\rk E_1$ is odd. We have just shown that $\Gamma_{\!s}\cdot iE_1^\R$ is 2-torsion. But by Theorem \ref{BJKur} this intersection is the Borisov-Joyce virtual cycle.

\begin{thm}\label{oddBJ} The Borisov-Joyce virtual cycle $[M]^{\vir}_{BJ}\in H_{\vd}(M,\Z)$ is 2-torsion when $\vd$ \eqref{vd} is odd.$\hfill\square$
\end{thm}\smallskip

Pridham \cite{Pr} has produced a refinement of $[M]^{\vir}_{BJ}$ (tensored with $\C[\![\hbar]\!]$ coefficients) using deformation quantisation. He shows it vanishes when $\vd$ is odd and the moduli space admits a \emph{global algebraic Darboux chart} --- a global chart \eqref{affmodel} in which everything is algebraic and $E$ has a nondegenerate quadratic form with respect to which the section $s$ is isotropic..

\section{\for{toc}{\hspace{-2mm}The algebraic virtual cycle}\except{toc}{The algebraic virtual cycle}} \label{ten}
Now let us assume that $\vd$ \eqref{vddef} is \emph{even}. Equivalently, $r=\rk E_1$ is even; call it $2n$.
The algebraic virtual cycle of \cite[Definition 4.4]{OT1} is then
\begin{equation}\label{vird}
[M]^{\vir}\ =\ \sqrt{0_{E_1}^{\,!}\!}\ [C]\ \in\ A_{\frac12\!\vd}\big(M,\Z\big[\textstyle{\frac12}\big]\big)
\end{equation}
using the square root Gysin map $\surd\;0_{E_1}^{\;!}$ of \cite[Definition 3.3]{OT1}.

We review Kiem-Park's alternative description of this cycle $\surd\;0_{E_1}^{\;!}[C]$. Recall that $C$ is equidimensional. If there exists an irreducible component of $M$ of dimension $\dim C$ it is in fact a smooth connected component $M_j\subseteq M$ on which $C$ is isomorphic to $M_j$ (because $C$ is locally the normal cone $C_{M/A}$ for some smooth $A\supseteq M$ of the same dimension; therefore $\supseteq$ is a local isomorphism).

On the complement of this locus let $p\colon\wt C\to C$ be the blow up of $C$ in the zero section $M\subset C$ and let $\bar p:=p|_{p^{-1}(M)}\colon p^{-1}(M)\to M$ be its restriction to the exceptional divisor --- the zero section of $\wt C\subset\bar p\;^*E_1$. Then by \cite[Theorem 5.2]{KP}
\beq{KiPa}
[M]^{\vir}\=\bar p_{*\!}\left(\;\sum\nolimits_i\sqrt{e}\;(\Lambda^\perp/\Lambda)\cap[F_i]\right)+\,\sum\nolimits_j\sqrt{e}\;(E_1)\cap[C_j]
\eeq
is a sum of two contributions from
\begin{itemize}
\item[($i$)] irreducible components $\wt C_i\subset\wt C$ on which the zero section is a Cartier divisor $F_i\subset\wt C_i$; here we push down 
$\sqrt{e}\;(\Lambda^\perp/\Lambda)$, where $\Lambda\subset \bar p\;^*E_1|_{F_i}$ is the \emph{isotropic} normal bundle of $F_i$ and $\Lambda^\perp/\Lambda$ is oriented as in \eqref{orperp}
\item[($j$)] the connected components $C_j\cong M_j$; here we take the Edidin-Graham class $\sqrt{e}\;(E_1)$.
\end{itemize}

We use this expression to relate $[M]^{\vir}$ to the real virtual class \eqref{defint1} after inverting 2.

\begin{thm} \label{KiemLi}
Under the cycle map
$$
\xymatrix@R=0pt{
A_{\frac12\!\vd}\big(M,\Z\big[\frac12\big]\big)\ \ar[r]&\ H_{\vd}\big(M,\Z\big[\frac12\big]\big), \\
\hspace{9mm}[M]^{\vir}\hspace{8mm} \ar@{|->}[r]&\ \ C\cdot iE_1^\R.\ }
$$
\end{thm}

\begin{proof}
We rewrite the (images in homology of the) summands in \eqref{KiPa}. \smallskip

\emph{Case} ($i$): the cone $\wt C_i\subset\bar p\;^*E_1|_{F_i}$ is just $\Lambda$. As in \eqref{Clin} there is an orthogonal $C^\infty$ splitting
$$
\bar p\;^*E_1\big|_{F_i}\=\(\Lambda\oplus\Lambda^*\)\oplus\(\Lambda^\perp/\Lambda\)\ =:\ E_\Lambda\oplus E_\perp.
$$
Here $\Lambda$ is the obvious maximal isotropic of the first summand $E_\Lambda=\Lambda\oplus\Lambda^*$, so by \eqref{Rlin} and the projection formula,
$$
\Lambda\cdot\bar p\;^*(iE_1^\R)\big|_{F_i}\=\Lambda\cdot(iE_\perp^\R\oplus iE_{\Lambda}^\R)\=e(E^\R_\perp)\cap\big(\Lambda\cdot iE_\Lambda^\R\big)\=e(E^\R_\perp)\cap[F_i].
$$
By \cite[Proposition 3]{EG} $e(E^\R_\perp)=\sqrt{e}(E_\perp)=\sqrt{e}\;(\Lambda/\Lambda^\perp)$, so in sum
\beq{i}
\(\wt C\cdot\bar p\;^*(iE_1^\R)\)\big|_{F_i}\=\sqrt{e}\;(\Lambda/\Lambda^\perp)\cap[F_i].
\eeq

\smallskip
\emph{Case} ($j$): any one of these components $C_j$ is the zero section of $E_1|_{M_j}$, so intersecting it with $iE_1^{\R}|_{M_j}$ using its Thom form as in \eqref{defint1} gives
\beq{j}
C_j\cdot iE_1^{\R}\=e\;(E_1^{\R})\cap[C_j]\=\sqrt{e}\;(E_1)\cap[C_j]
\eeq
in homology, the last equality being \cite[Proposition 3]{EG}.
\medskip

The sum of $\bar p_*$\eqref{i} and \eqref{j} gives $C\cdot iE_1^\R$ on the left hand side, by the projection formula, and \eqref{KiPa} on the right hand side.
\end{proof}

Combining Theorem \ref{KiemLi} with Theorem \ref{BJKur} gives
$$
[M]^{\vir}\=[M]^{\vir}_{BJ}\ \in\ H_{\vd}\(M,\ZZ\),
$$
thus proving Theorem \ref{BBJJ} when $\vd$ is even.

\appendix
\section{Another proof of Theorem \ref{KiemLi}}

For completeness we include a very different proof of Theorem \ref{KiemLi} from an earlier version of this paper. We unravel the cosection localisation definition of the virtual cycle and then use a result of Kiem-Li \cite[Appendix A]{KL} to relate it to the real virtual class \eqref{defint1} after inverting 2.

Just as in \cite[Section 3]{OT1} we may assume without loss of generality that $E_1$ admits a maximal isotropic subbundle
\beq{Ses}
0\To\Lambda\To E_1\To\Lambda^*\To0,
\eeq
which is positive in the sense of \cite[Definition 2.2]{OT1}.\footnote{\label{footor} By \cite[Proposition 2.3]{OT1} this means the isomorphism of real bundles given by the composition $\Lambda\into E_1\cong E_1^\R\oplus iE_1^\R\onto E_1^\R$ takes the real orientation on $\Lambda$ --- induced by its complex structure --- to the real orientation on $E_1^\R$ induced by the complex orientation on $E$ as in \cite[Section 2.2]{OT1}.} If not we pull back to the cover $\rho\colon\widetilde M\to M$ of \cite[Equation 19]{OT1}, where there is a tautological such subbundle $\Lambda_\rho\subset\rho^*E_1$, and prove the Theorem there. Then we push the result down to $M$ by $2^{1-n}\rho_*(h\cup\ \cdot\ )$, where $h\in A^{n(n-1)}(\wt Y, \Z)$ is a canonical class with $\rho_*h=2^{n-1}$  \cite[Equation 20]{OT1} (so that $2^{1-n}\rho_*(h\cup\ \cdot\ )$ a left inverse of $\rho^*$); this is where we need to use $\Z[1/2]$ coefficients.

We first recall how the class \eqref{vird} is defined in \cite[Definition 3.3]{OT1}. It uses \cite[Equation 29]{OT1} as a way to intersect $C$ with $\Lambda$ inside $E_1$. We spell out the construction. Working on the total space $\pi\colon E_1\to M$, we use the tautological section $\tau\_{E_1}$ of $\pi^*E_1$, restricted to $C\subset E_1$. The projection of $\tau\_{E_1}|\_C$ to $\pi^*\Lambda^*$ has zero locus $\Lambda\cap C$. We linearise about this zero locus, Fulton-MacPherson style, by degenerating
\begin{equation}\label{degn}
\Big(C\ \subset\ E_1\Big)\ \rightsquigarrow\ 
\Big(C\_{\Lambda\cap C/C}\ \subset\ \Lambda\oplus\Lambda^*\Big).
\end{equation}
Here both pairs are isotropic cones inside orthogonal bundles, where on the right hand side we use the standard quadratic form on $\Lambda\oplus\Lambda^*$ given by pairing $\Lambda$ and $\Lambda^*$. This degeneration is the deformation to the normal bundle of $\Lambda$ in the total space of $E_1$ \cite[Section 5.1]{Fu},
\begin{equation}\label{degnc}\xymatrix@R=15pt{
\cE\ :=\ \mathrm{Bl}_{\Lambda\times\{0\}}\big(E_1\times\C\big)\take\overline{(E_1\times\{0\})} \ar[r]^-r\ar[d]& E_1 \\ \C.\!}
\end{equation}
The general fibre over $0\ne t\in\C$ is $E_1$. The special fibre over $t=0$ is $\Lambda\oplus\Lambda^*$ --- the total space of the normal bundle to $\Lambda\subset E_1$. Considering the quadratic form on the fibres of $E_1\to M$ to be a section $q$ of $\Sym^2\Omega_{E_1/M}$, we get the nondegenerate quadratic form
$$
t^{-1}r^*(q)\ \in\ \Gamma\big(\Sym^2\;\Omega_{\cE/\C}\big)
$$
on the fibres of \eqref{degnc}. It is a degeneration of the quadratic form $q$ on the general fibre to the standard quadratic form on the special fibre $\Lambda\oplus\Lambda^*$. Taking the proper transform of $C\times\C$ in \eqref{degnc} gives the degeneration \eqref{degn} of $C\subset E_1$ to $C\_{(\Lambda\cap C)/C}\subset\Lambda\oplus\Lambda^*$,
the normal cone of $(\Lambda\cap C)\subset C$.
Since the generic fibres are all isotropic, so is this central fibre.

We consider $\Lambda\oplus\Lambda^*$ to be the pullback bundle $p^*\Lambda^*$, containing $C\_{(\Lambda\cap C)/C}$. To intersect the latter with the 0-section of the former we use the cosection defined by the tautological section of $p^*\Lambda$,
$$
\tau\_\Lambda\,\colon\ p^*\Lambda^*\To\cO_\Lambda.
$$
The fact that $C\subset E_1$ is isotropic becomes, after the degeneration \eqref{degnc}, the identity \cite[Lemma 3.1]{OT1},
\beq{cosec7}
\tau\_\Lambda\big|_{C\_{(\Lambda\cap C)/C}}\ \equiv\ 0.
\eeq
Thus we may apply \cite[Proposition 1.3]{KL} to define a class
\begin{equation}\label{csloc}
0_{p^*\Lambda^*\!,\,\tau\_\Lambda}^{\;!,\,\mathrm{loc}}\big[C\_{(\Lambda\cap C)/C}\big]\ \in\ A_{\frac12\!\vd}(M)
\end{equation}
cosection localised to the 0-section $M\subset\Lambda$ --- i.e. the zero locus of the section $\tau\_\Lambda$ of $p^*\Lambda$. In \cite{OT1} we defined the algebraic virtual cycle to be
$$
[M]^{\vir}\=\sqrt{0_{E_1}^{\,!}}\,[C]\ :=\ (-1)^n\,0_{p^*\Lambda^*\!,\,\tau\_\Lambda}^{\;!,\,\mathrm{loc}}\big[C\_{(\Lambda\cap C)/C}\big].
$$\smallskip

Using \cite[Proposition A.1]{KL} this can be given the following more topological description.

Pick a hermitian metric on $\Lambda$, thus inducing one on $\Lambda^*$ and defining an isomorphism $\Lambda^*\cong\overline\Lambda$ (where $\overline\Lambda$ is $\Lambda$ as a set and an $\R$-bundle, but with the conjugate complex structure --- i.e. $i$ acts on $\overline\Lambda$ by multiplication by $-i$ on $\Lambda$). For a smooth function $\psi\colon[0,\infty)\to(0,\infty)$ consider the section $\sigma_\psi$ of $p^*\overline\Lambda\cong p^*\Lambda^*$ over the total space of $\Lambda$ which takes $v\in\Lambda$ to $\psi(|v|)v\in p^*\overline\Lambda$. Its graph is
\beq{grarf}
\Gamma_{\!\sigma_\psi}\ :=\ \Big\{\big(v,\psi(|v|)v\big)\,:\,v\in\Lambda\Big\}\ \subset\ \Lambda\oplus\overline\Lambda.
\eeq
Choosing $0<\delta\ll1$ and $\psi(|v|)=|v|^{-2}$ for $|v|>\delta$, 
the section $\sigma_\psi$ satisfies the condition required in \cite[Appendix A]{KL},
\begin{equation}\label{non0}
\langle\sigma_\psi,\tau\_\Lambda\;\rangle\ =\ 1
\end{equation}
on the complement of the neighbourhood $U^\delta_\Lambda(M)$ \eqref{retract}. Since $\tau\_\Lambda$ vanishes on $C\_{(\Lambda\cap C)/C}$ \eqref{cosec7} but equals $\psi(|v|)|v|^2$ on $\Gamma_{\!\sigma_\psi}$, the cone and the graph intersect only in the zero section $M\subset\Lambda\oplus\overline\Lambda$. Therefore, setting $U=U^\epsilon_{\!\Lambda\oplus\overline\Lambda}(M)$, we can define the topological intersection
\begin{equation}\label{graphpsi}
\Gamma_{\!\sigma_\psi}\cdot C\_{(\Lambda\cap C)/C}\ \in\ H\_{\vd}(U)\ \cong\ H_{\vd}(M)
\end{equation}
in the usual way as the cap product of the Thom class
\beq{yusual}
\big[\Gamma_{\!\sigma_\psi}\big]\ \in\ H^{2n}\(\;\overline U,\partial\overline U\take\Gamma_{\!\sigma_\psi}\) \quad\text{and}\quad \big[C\_{(\Lambda\cap C)/C}\big]\ \in\ H_{2\rk T}\(\;\overline U,\partial\overline U\take\Gamma_{\!\sigma_\psi}\).
\eeq
(Here $T$ is the bundle in the normal form of \eqref{E*E}.)
By \cite[Proposition A.1]{KL} the intersection \eqref{graphpsi} is the image of \eqref{csloc} under the cycle map $A_{\frac12\!\vd}(M)\to H_{\vd}(M)$ so
$$
[M]^{\vir}\=(-1)^n\;\Gamma_{\!\sigma_\psi}\cdot C\_{(\Lambda\cap C)/C}.\vspace{1mm}
$$

Now deform $\psi$ to the constant function $1$. This deforms the right hand side of \eqref{non0} from 1 to $|v|^2$ through $(0,\infty)$. Thus it is always nonzero, whereas $\tau\_{\Lambda}$ vanishes on $C\_{(\Lambda\cap C)/C}$ \eqref{cosec7}, so again $C\_{(\Lambda\cap C)/C}$ and $\Gamma_{\!\sigma_\psi}$ do not meet outside $M\subset U$. Thus the construction \eqref{yusual} of the intersection \eqref{graphpsi} deforms too, so the homology class of $[M]^{\vir}$ is
$$
(-1)^n\;\Gamma_{\!\sigma_1}\cdot C\_{(\Lambda\cap C)/C}\ \in\ H\_{\vd}(U)\ \cong\ H_{\vd}(M).
$$
But $\Gamma_{\!\sigma_1}\subset\Lambda\oplus\Lambda^*$ is the diagonal in $\Lambda\oplus\overline\Lambda$ --- a maximal positive definite real subbundle. As we deform away from the central fibre of \eqref{degnc} back to the general fibre we can deform $\Gamma_{\!\sigma_1}\subset\Lambda\oplus\Lambda^*$ through other positive definite subbundles to $E_1^\R\subset E_1$ because the space of maximal positive definite subspaces of a quadratic vector space is contractible.\vspace{-.6mm}
As we do so $C\_{(\Lambda\cap C)/C}$ deforms back to $C$ through isotropic cones, which thus miss $E_1^\R$ outside $M$. Therefore again the construction \eqref{yusual} of the intersection \eqref{graphpsi} deforms too, with the homology class of $[M]^{\vir}$ becoming
\beq{orE1}
(-1)^nE_1^\R\cdot C\ \in\ H_{\vd}(M).
\eeq
Here we have to specify the orientation on the normal bundle $iE_1^\R$ to $E_1^\R\subset E_1$ induced by our construction.

To do so we note $\Lambda\oplus\Lambda^*$ has a complex orientation induced by its deformation \eqref{degnc} to $E_1$. This induces a canonical real orientation on the maximal positive definite real subbundle $\Gamma_{\!\sigma_1}\subset\Lambda\oplus\Lambda^*$ by \cite[Section 2.2]{OT1}. On deformation through the graphs $\Gamma_{\!\sigma_\psi}$ \eqref{grarf} back to $\Gamma_{\!\sigma_0}=\Lambda$ this becomes the standard orientation induced by the complex structure on $\Lambda$ as in Footnote \ref{footor}. So once we deform to $E_1^\R\subset E_1$ we have the canonical real orientation on $E_1^\R$ and the canonical real orientation on $E_1$ induced by its complex structure. Together these induce an orientation on the normal bundle $iE_1^\R$ to $E_1^\R\subset E_1$ which differs\footnote{This is the calculation that $(e_1\wedge_\R\dots\wedge_\R e_{2n})\wedge_\R(ie_1\wedge_\R\dots\wedge_\R ie_{2n})$ is equal to $(-1)^n(e_1\wedge_\R ie_1)\wedge_\R(e_2\wedge_\R ie_2)\wedge_\R\dots$, where $\{e_i\}\subset E_1^\R$ is an oriented basis for $E_1$ in the complex sense, or equivalently for $E_1^\R$ in the real sense; see \cite[Section 2.2]{OT1}.} by the sign $(-1)^n$ from the one coming from the isomorphism $\times i\colon E_1^\R\to iE_1^\R$.

So applying $\times i$ to $E_1$, which preserves $C\subset E_1$ and their real orientations induced by their complex structures, takes \eqref{orE1} to the required class
\[
iE_1^\R\cdot C\ \in\ H_{\vd}(M).
\]

\section{Some extension results} \label{appA}
\subsection{Stein extension results}
We thank Richard L\"ark\"ang and Mohan Ramachandran for expert assistance with the following examples of Oka's principle. Recall that by a Stein space we mean a (not necessarily reduced) complex analytic space which is isomorphic to a closed analytic subspace of some $\C^N$.

\begin{prop}\label{steinext}
Given holomorphic maps of analytic schemes
$$
\xymatrix@C=30pt@R=15pt{\ X_0 \ar[r]^{f_0}& Y \\ X \ar@{<-_)}[u]-<0pt,9pt>\ar@{-->}[ur]_f}
$$
with $Y$ smooth and $X_0,\,Y$ Stein, we may shrink $X$ to a Stein neighbourhood of $X_0$ such that there is a a holomorphic map $f\colon X\to Y$ with $f|_{X_0}=f_0$.
\end{prop}

\begin{proof}
Composing $f_0$ with an embedding $Y\into\C^N$ we get $N$ holomorphic functions on $X_0$ which we can extend to a neighbourhood of $X_0\subset X$. So by shrinking $X$ we may assume we have lifted $f_0$ to $f\colon X\to\C^N$.

Since $Y$ is smooth and Stein it has a holomorphic tubular neighbourhood $Y\,\raisebox{-2pt}{$\stackrel\longleftarrow\into$}\,\mathring\C^N\subset\C^N$ by \cite[Corollary 1]{Siu}. Shrinking $X$ we may assume $f$ maps it into $\mathring\C^N$, then composing with $\mathring\C^N\to Y$ gives a map $X\to Y$ extending $f_0$. Finally $X_0\subset X$ admits a Stein neighbourhood by the main theorem of \cite{Siu}, so we use this to replace $X$.
\end{proof}

\begin{prop}\label{steinbdl}
Fix a $E_0$ a holomorphic bundle over $X_0\subset Y$, with $X_0$ Stein and $Y$ smooth. Then there is a holomorphic vector bundle $E$ over a Stein neighbourhood $X\supset X_0$ such that $E|_{X_0}\cong E_0$.
\end{prop}

\begin{proof}
By \cite{Siu} a Stein neighbourhood $X\supset X_0$ exists, and may be embedded $X\into\C^N$. Then there exists a Stein neighbourhood $\mathring\C^N$ of $X_0$ such that $X_0\subset\mathring\C^N$ is a deformation retract. Pulling $E_0$ back by the retraction gives a topological bundle $E$ on $\mathring\C^N$, which by Grauert's Oka principle \cite[Theorem 8.2.1(ii)]{Fo} admits a holomorphic structure which agrees on $X_0^{\mathrm{red}}$ with $E_0|_{X_0^{\mathrm{red}}}$.

So we have a bundle $E|_X$ over $X$ such that $E|_{X_0}$ and $E_0$ are isomorphic on restriction to $X_0^{\mathrm{red}}$. What remains is to deduce that they are isomorphic on $X_0$. Filter $X_0$ by
$$
X_0^{\mathrm{red}}\=X_0^1\ \subset\ \dots\ \subset\ X_0^n\=X_0
$$
so that $X_0^i\subset X_0^{i+1}$ has square zero ideal sheaf $I_i$. Extensions of the bundle $E|_{X_0^i}$ from $X_0^i$ to $X_0^{i+1}$ form a torsor over $H^1\(\cE nd\, E|_{X_0^i}\otimes I_i\)$. Since $X_0^i$ is Stein this vanishes, so the extension $E|_{X_0^{i+1}}$ is unique. By induction the extension $E|_{X_0^n}$ is unique and so equals $E_0$.
\end{proof}

\subsection{Extending $\cA_1$-line bundles}
We first prove Theorem \ref{extendE} when $F$ has rank 1. This means its transition functions commute, so we will be able to use a standard partition of unity argument on the $C^\infty_V$-module $I_1$.

We cover $V=\bigcup_iV_i$ by sufficiently small open sets $V_i$ that (i) the $F|_{U_i:=V_i\cap U}$ are all trivial and (ii) the $V_i$ admit Stein structures compatible with the $\cA_2$ structure on $V\!$. That is, 
the $C^\infty$ structure on $V_i$ is that restricted from $V$\!, and the two holomorphic structures on $2U\cap V_i$ --- induced by the Stein structure on $V_i$ and the $\cA_2$ structure on $V$ --- agree.

Thus $F$ is defined by holomorphic transition functions $\phi_{ij}$ over $U_{ij}$ satisfying $\phi_{jk}\phi_{ik}^{-1}\phi_{ij}\equiv1$ on $U_{ijk}$. Shrinking $V_i$ if necessary we may lift these to $\cA_1$ functions $\Phi_{ij}\colon V_{ij}\to\C^*$ such that on $V_{ijk}$ the cocyle 
$$
\delta_{ijk}\ :=\ \Phi_{jk}\Phi_{ik}^{-1}\Phi_{ij} \quad\text{satisfies}\quad \big|\delta_{ijk}-1\big|\ <\ 1.
$$  
In particular we may use the branch of log on $\C\,\take(-\infty,0]$ which satisfies $\log1=0$ to define new smooth transition functions
\beq{po1}
\Psi_{ij}\ :=\ \exp\left(-\sum\nolimits_kh_k\log\delta_{kij}\right)\cdot\Phi_{ij}\ \text{ on }V_{ij},
\eeq
where the $\{h_k\}$ form a smooth partition of unity for the cover $V=\bigcup_kV_k$. Since $\delta_{ijk}-1\in I_1$ and $\log(1+z)\in(z)\cdot\cO$ we see that $\log\delta_{ijk}\in I_1$. Similarly $e^z-1\in(z)\cdot\cO$ then shows that $\Psi_{ij}-\Phi_{ij}\in I_1$. Thus
$$
\Psi_{ij}\ \in\ \cA_1 \quad\text{and}\quad \Psi_{ij}\big|_{U_{ij}}\=\phi_{ij}.
$$
Of course the $\Psi_{ij}$ were constructed so that the usual cocycle argument\footnote{Taking logs in \eqref{po1} turns the unfamiliar multiplicative argument into the more familiar additive one that a cocyle valued in a fine sheaf is a coboundary.} --- using the identity $\delta_{jk\ell}\;\delta_{ik\ell}^{-1}\delta_{ij\ell}\;\delta_{ijk}^{-1}=1$ --- gives $\Psi_{jk}\Psi_{ik}^{-1}\Psi_{ij}=1$ on $V_{ijk}$. Thus the $\Psi_{ij}$ define $\cA_1$-transition functions for a line bundle on $V$ which restricts to $F$ on $U\!$.

\subsection{Extending $\cA_1$-bundles}\label{proof} 
To prove Theorem \ref{extendE} for a rank $r$ bundle $F$ we may now replace it by  its twist by any holomorphic line bundle on $U$. This extends to an $\cA_1$-bundle on $V$ if and only if $F$ does, because we have already shown the line bundle extends.

So we may assume that $F$ is generated by its global sections, and hence is the pullback of the universal bundle by a holomorphic map $f\colon U\to\Gr$ to a Grassmannian $\Gr\;(\C^N,r)$. We will extend $f$ to an $\cA_1$-map\footnote{That is $F^*\colon C^\infty_{\Gr}\to C^\infty_V$ maps $\cO_{\Gr}$ in the former to $\cA_1$ in the latter.} $F\colon V\to\Gr$ in several steps. This will prove Theorem \ref{extendE}, with the holomorphic tranistion functions of the universal bundle on Gr pulling back to $\cA_1$-transition functions on $V\!$. \medskip

As before we pick a Stein cover $V=\bigcup_iV_i$ with holomorphic extensions $F_i\colon V_i\to\Gr$ of $f|_{U_i}\colon U_i\to\Gr$.

To glue these via a partition of unity argument we need a linear structure on open sets in Gr. Fix a point $0\to S\to\C^N\to Q\to0$ which we call $S\in\Gr$ for short. The standard hermitian metric on $\C^N$ defines the orthogonal splitting $\C^N\cong S\oplus Q$. We have the (linear!) open set
\beqa
e_S\,\colon\,T_S\Gr\ \cong\ \Hom(S,Q) &\Into& \Gr, \\
\phi\quad &\Mapsto& \im(\id\oplus\;\phi),
\eeqa
where $\id\oplus\;\phi\colon S\into S\oplus Q\cong\C^N$. Thus the origin $\phi=0$ maps to $S\in\Gr$.

This puts a linear structure on the Zariski open set of points $[S'\subset\C^N]\in\Gr$ such that $S'\cap S^\perp=\{0\}$. Allowing the point $[S\into\C^N]$ to vary we get an open neighbourhood of the diagonal $\Delta\subset\Gr\times\Gr$,
$$
(\Gr\times\Gr)^\circ\ :=\ \big\{(S,T)\in\Gr\times\Gr\ \colon\,T\cap S^\perp=\{0\}\big\},
$$
and a $C^\infty$ isomorphism
\begin{eqnarray}\label{open}
e\,\colon\,T\Gr &\rt\sim& (\Gr\times\Gr)^\circ, \\
(S,\phi) &\Mapsto& \(S,\,\im(\id\oplus\;\phi)\). \nonumber
\end{eqnarray}
It is holomorphic on each tangent fibre and maps the 0-section to $\Delta$.

Our first (very rough) attempt to glue the $F_i$ together is a map $s\colon V\to\Gr$ which is sufficiently close to $F_i$ over $V_i$ that each $(s,F_i)$ maps $V_i$ to a neighbourhood of the diagonal $\Delta$.

\begin{lem}\label{lemeps}
After possibly shrinking $V$ there exists a $C^\infty$ map $s\colon V\to\Gr$ such that $\(s|_{V_i},F_i\)\colon V_i\to\Gr\times\Gr$ factors through $(\Gr\times\Gr)^\circ$ for all $i$.
\end{lem}

\begin{proof}
We use the Fubini-Study metric of diameter 1 on Gr and fix $0<\epsilon\ll1$ so that the $3\epsilon$-neighbourhood $U^{3\epsilon}(\Delta)$ of $\Delta$ is contained in $(\Gr\times\Gr)^\circ$.

Consider the $F_i$ as (partially defined) sections of the trivial bundle $V\times\Gr\to V$. Any two of them $F_i,\,F_j$ are simultaneously defined on $V_{ij}\subset V$ and equal on $U_{ij}\subset V_{ij}$. Therefore, shrinking $V$ if necessary, we may assume their graphs are within $\epsilon$ of each other on $V_{ij}$. Thus
$$
\bigcup\nolimits_i\bigcup\nolimits_{x\in V_i}\ \{x\}\times B_{\;3\epsilon}\(F_i(x)\)\ \subset\ V\times\Gr
$$
is a bundle over $V$ with contractible fibres --- nonempty unions of balls each of which contains the centre of the other balls --- and local sections connecting them. It therefore admits a global section $s\colon V\to\Gr$.
\end{proof}

Think of \eqref{open} as an isomorphism of bundles over the first factor $\Gr\ni S$. Pulling back by $s$ gives the isomorphism of bundles
$$
e_s\ \colon\,s^*(T\Gr)\ \rt\sim\ s^*(\Gr\times\Gr)^\circ
$$
over $V$. The $F_i$ define sections of the latter over $V_i$, so we can now use the linear structure on the fibres of the former --- and a partition of unity $\{h_i\}$ subordinate to the cover $V=\bigcup_iV_i$ --- to glue them. This gives the section
\beq{Fdef}
F\ :=\ e_s\left(\;\sum\nolimits_ih_i\,e_s^{-1}(F_i)\right)\ \text{ of }\  s^*(\Gr\times\Gr)^\circ.
\eeq
Here we ignore the $i$th term wherever $h_i=0$; since $h_i\ne0$ only on $V_i$ the sum is well defined in $T\Gr$ by Lemma \ref{lemeps}.

The composition of $s^*(\Gr\times\Gr)^0\subset s^*(\Gr\times\Gr)\cong V\times\Gr$ with projection to the second factor defines our $C^\infty$ map $F\colon V\to\Gr$.

\begin{prop}
$F\colon V\to\Gr$ is an $\cA_1$-map with $F|_U=f$.
\end{prop}

\begin{proof}
Using $e$ \eqref{open} to identify the two bundles $(\Gr\times\Gr)^\circ$ and $T\Gr$ over $\Gr$, any map $(a,G)\colon X\to(\Gr\times\Gr)^\circ$ induces a map $e^{-1}\circ(a,G)\colon X\to T\Gr$ which we denote by $(a,\wt G)$. (So $\wt G$ is a section of $a^*(T\Gr)$.) In this notation, \eqref{Fdef} becomes
\beq{Ftilde}
\wt F\=\sum\nolimits_ih_i\;\wt F_{i\;}.
\eeq

Fix any (local) holomorphic function $g$ on $\Gr$. We are required to show that $F^*g\in\cA_1$. Pulling back $g$ from the second factor gives $1\otimes g$ on $\Gr\times\Gr$ and so $\wt g:=e^*(1\otimes g)$ on $T\Gr$. These satisfy
$$
\wt F^*\;\wt g\=\(e^{-1}\circ(s,F)\)^*e^*(1\otimes g)\=(s,F)^*(1\otimes g)\=F^*g.
$$
So we just need to show that $\wt F^*\;\wt g\in\cA_1$ locally about a point $x\in V$. Discarding any $V_i$ whose closure does not contain $x$ we may replace $V$ by an open subset of $\bigcap_iV_i$. Now $F_1$ is holomorphic so $\wt F_1^*\;\wt g=F_1^*g$ is too. Thus it is sufficient to prove that $F^*g-F_1^*g=\wt F^*\;\wt g-\wt F_1^*\;\wt g\in I_1$. (This will also show that $F|_U\equiv F_1|_U\equiv f$, proving the second claim of the Proposition.) Applying the Hadamard Lemma fibrewise on $s^*(T\Gr)\to V$ gives
$$
\wt F^*\;\wt g-\wt F_1^*\;\wt g\=\Langle G,\,\wt F-\wt F_1\Rangle+\Big\langle\,\overline{\!H},\,\overline{\wt F-\wt F_1}\,\Big\rangle
$$
for some $T^*\Gr$-valued functions $G,H$ made from the holomorphic (respectively antiholomorphic) derivative of $\wt g$ down the fibres of $T\Gr$. But $\wt g=e^*(1\otimes g)$ is holomorphic on the fibres of $T\Gr$ because $e$ is, so $H\equiv0$. Thus \eqref{Ftilde} gives
\beq{end}
\wt F^*\;\wt g-\wt F_1^*\;\wt g\=\Big\langle G,\,\sum\nolimits_ih_i\(\wt F_{i\;}-\wt F_1\)\Big\rangle\=\sum\nolimits_{ij}h_i\,G_j\(\wt F_{i\;}-\wt F_1\)_j\;.
\eeq
For the last equality we have picked a holomorphic trivialisation of $T\Gr$ over $\im(s)$ (after shrinking $V$ again) and written $G,\,\wt F_{i\;}-\wt F_1$ in terms of components. By \eqref{end} we are left with showing $\(\wt F_{i\;}-\wt F_1\)_j\in I_1$ for all $i,\;j$.

For this we also pick a local \emph{holomorphic}\footnote{So not \eqref{open}, which is only holomorphic on the fibres.} trivialisation of the bundle $(\Gr\times\Gr)^\circ\to\Gr$ in a neighbourhood of $\Delta|_{\im(s)}$. Thus we can also express the $F_i$ in terms of components $(F_i)_k$. Then by the Hadamard Lemma down the fibres of $s^*(\Gr\times\Gr)^\circ\to V$,
\beq{last}
\(\wt F_{i\;}-\wt F_1\)_j\=\sum\nolimits_kE_{jk}\big[(F_i)_k-(F_1)_k\big],
\eeq
where the $E_{jk}$ are made from the holomorphic derivatives $\partial_k$ of the components of $e^{-1}_s$ down the fibres of $s^*(\Gr\times\Gr)^\circ\to V$. Since $e^{-1}_s$ is fibrewise holomorphic (because $e_s$ is) there are no complex conjugate terms.

Finally then, $F_i|_U\equiv F_1|_U$ means $(F_i)_k-(F_1)_k\in I$, so that \eqref{last} lies in $I_1$ as required.
\end{proof}

\bibliographystyle{halphanum}
\bibliography{references}

\bigskip\noindent
{\tt{j.oh@imperial.ac.uk} \\ \tt{richard.thomas@imperial.ac.uk}} \medskip

\noindent Department of Mathematics \\
Imperial College London\\
London SW7 2AZ \\
United Kingdom

\end{document}